\documentclass[11pt]{amsart}    
\usepackage{latexsym, amsmath, amsthm, amssymb, setspace, verbatim}
\usepackage[all]{xy}
\usepackage{longtable}
\usepackage{hyperref}

\title{ Rokhlin dimension for actions of residually finite groups }
\author{ G{\'a}bor Szab{\'o}, Jianchao Wu \and Joachim Zacharias }
\address{ Fraser Noble Building, Institute of Mathematics, University of Aberdeen, \linebreak \text{}\hspace{2.8mm} Aberdeen AB24 3UE, Scotland, UK }
\email{ gabor.szabo@abdn.ac.uk }
\address{ Department of Mathematics, Pennsylvania State University, \hfill \phantom{-} \linebreak \text{}\hspace{2.8mm} 109 McAllister Building, University Park, PA 16802, USA }
\email{ jianchao.wu@psu.edu }
\address{ School of Mathematics and Statistics, University of Glasgow, \hfill \phantom{-}\linebreak \text{}\hspace{2.8mm} University Gardens, Glasgow Q12 8QW, Scotland, UK }
\email{ joachim.zacharias@glasgow.ac.uk }
\thanks{\emph{Supported by:} SFB 878 \emph{Groups, Geometry and Actions}, GIF Grant 1137-30.6/2011, ERC Advanced Grant ToDyRiC 267079, EPSRC Grant EP/I019227/2 and EPSRC Grant EP/N00874X/1}

\begin{document}

\renewcommand\matrix[1]{\left(\begin{array}{*{10}{c}} #1 \end{array}\right)}  
\newcommand\set[1]{\left\{#1\right\}}  
\newcommand\mset[1]{\left\{\!\!\left\{#1\right\}\!\!\right\}}

\newcommand{\IA}[0]{\mathbb{A}} \newcommand{\IB}[0]{\mathbb{B}}
\newcommand{\IC}[0]{\mathbb{C}} \newcommand{\ID}[0]{\mathbb{D}}
\newcommand{\IE}[0]{\mathbb{E}} \newcommand{\IF}[0]{\mathbb{F}}
\newcommand{\IG}[0]{\mathbb{G}} \newcommand{\IH}[0]{\mathbb{H}}
\newcommand{\II}[0]{\mathbb{I}} \renewcommand{\IJ}[0]{\mathbb{J}}
\newcommand{\IK}[0]{\mathbb{K}} \newcommand{\IL}[0]{\mathbb{L}}
\newcommand{\IM}[0]{\mathbb{M}} \newcommand{\IN}[0]{\mathbb{N}}
\newcommand{\IO}[0]{\mathbb{O}} \newcommand{\IP}[0]{\mathbb{P}}
\newcommand{\IQ}[0]{\mathbb{Q}} \newcommand{\IR}[0]{\mathbb{R}}
\newcommand{\IS}[0]{\mathbb{S}} \newcommand{\IT}[0]{\mathbb{T}}
\newcommand{\IU}[0]{\mathbb{U}} \newcommand{\IV}[0]{\mathbb{V}}
\newcommand{\IW}[0]{\mathbb{W}} \newcommand{\IX}[0]{\mathbb{X}}
\newcommand{\IY}[0]{\mathbb{Y}} \newcommand{\IZ}[0]{\mathbb{Z}}

\newcommand{\CA}[0]{\mathcal{A}} \newcommand{\CB}[0]{\mathcal{B}}
\newcommand{\CC}[0]{\mathcal{C}} \newcommand{\CD}[0]{\mathcal{D}}
\newcommand{\CE}[0]{\mathcal{E}} \newcommand{\CF}[0]{\mathcal{F}}
\newcommand{\CG}[0]{\mathcal{G}} \newcommand{\CH}[0]{\mathcal{H}}
\newcommand{\CI}[0]{\mathcal{I}} \newcommand{\CJ}[0]{\mathcal{J}}
\newcommand{\CK}[0]{\mathcal{K}} \newcommand{\CL}[0]{\mathcal{L}}
\newcommand{\CM}[0]{\mathcal{M}} \newcommand{\CN}[0]{\mathcal{N}}
\newcommand{\CO}[0]{\mathcal{O}} \newcommand{\CP}[0]{\mathcal{P}}
\newcommand{\CQ}[0]{\mathcal{Q}} \newcommand{\CR}[0]{\mathcal{R}}
\newcommand{\CS}[0]{\mathcal{S}} \newcommand{\CT}[0]{\mathcal{T}}
\newcommand{\CU}[0]{\mathcal{U}} \newcommand{\CV}[0]{\mathcal{V}}
\newcommand{\CW}[0]{\mathcal{W}} \newcommand{\CX}[0]{\mathcal{X}}
\newcommand{\CY}[0]{\mathcal{Y}} \newcommand{\CZ}[0]{\mathcal{Z}}

\newcommand{\FA}[0]{\mathfrak{A}} \newcommand{\FB}[0]{\mathfrak{B}}
\newcommand{\FC}[0]{\mathfrak{C}} \newcommand{\FD}[0]{\mathfrak{D}}
\newcommand{\FE}[0]{\mathfrak{E}} \newcommand{\FF}[0]{\mathfrak{F}}
\newcommand{\FG}[0]{\mathfrak{G}} \newcommand{\FH}[0]{\mathfrak{H}}
\newcommand{\FI}[0]{\mathfrak{I}} \newcommand{\FJ}[0]{\mathfrak{J}}
\newcommand{\FK}[0]{\mathfrak{K}} \newcommand{\FL}[0]{\mathfrak{L}}
\newcommand{\FM}[0]{\mathfrak{M}} \newcommand{\FN}[0]{\mathfrak{N}}
\newcommand{\FO}[0]{\mathfrak{O}} \newcommand{\FP}[0]{\mathfrak{P}}
\newcommand{\FQ}[0]{\mathfrak{Q}} \newcommand{\FR}[0]{\mathfrak{R}}
\newcommand{\FS}[0]{\mathfrak{S}} \newcommand{\FT}[0]{\mathfrak{T}}
\newcommand{\FU}[0]{\mathfrak{U}} \newcommand{\FV}[0]{\mathfrak{V}}
\newcommand{\FW}[0]{\mathfrak{W}} \newcommand{\FX}[0]{\mathfrak{X}}
\newcommand{\FY}[0]{\mathfrak{Y}} \newcommand{\FZ}[0]{\mathfrak{Z}}

\newcommand{\Ra}[0]{\Rightarrow}
\newcommand{\La}[0]{\Leftarrow}
\newcommand{\LRa}[0]{\Leftrightarrow}

\renewcommand{\phi}[0]{\varphi}
\newcommand{\eps}[0]{\varepsilon}

\newcommand{\quer}[0]{\overline}
\newcommand{\uber}[0]{\choose}
\newcommand{\ord}[0]{\operatorname{ord}}		
\newcommand{\GL}[0]{\operatorname{GL}}
\newcommand{\supp}[0]{\operatorname{supp}}	
\newcommand{\id}[0]{\operatorname{id}}		
\newcommand{\Sp}[0]{\operatorname{Sp}}		
\newcommand{\eins}[0]{\mathbf{1}}			
\newcommand{\diag}[0]{\operatorname{diag}}
\newcommand{\ind}[0]{\operatorname{ind}}
\newcommand{\auf}[1]{\quad\stackrel{#1}{\longrightarrow}\quad}
\newcommand{\hull}[0]{\operatorname{hull}}
\newcommand{\prim}[0]{\operatorname{Prim}}
\newcommand{\ad}[0]{\operatorname{Ad}}
\newcommand{\quot}[0]{\operatorname{Quot}}
\newcommand{\ext}[0]{\operatorname{Ext}}
\newcommand{\ev}[0]{\operatorname{ev}}
\newcommand{\fin}[0]{{\subset\!\!\!\subset}}
\newcommand{\diam}[0]{\operatorname{diam}}
\newcommand{\Hom}[0]{\operatorname{Hom}}
\newcommand{\Aut}[0]{\operatorname{Aut}}
\newcommand{\del}[0]{\partial}
\newcommand{\inter}[0]{\operatorname{int}}
\newcommand{\dimeins}[0]{\dim^{\!+1}}
\newcommand{\dimnuc}[0]{\dim_{\mathrm{nuc}}}
\newcommand{\dimnuceins}[0]{\dimnuc^{\!+1}}
\newcommand{\dr}[0]{\operatorname{dr}}
\newcommand{\dimrok}[0]{\dim_{\mathrm{Rok}}}
\newcommand{\dimrokeins}[0]{\dimrok^{\!+1}}
\newcommand{\dreins}[0]{\dr^{\!+1}}
\newcommand*\onto{\ensuremath{\joinrel\relbar\joinrel\twoheadrightarrow}} 
\newcommand*\into{\ensuremath{\lhook\joinrel\relbar\joinrel\rightarrow}}  
\newcommand{\im}[0]{\operatorname{im}}
\newcommand{\dst}[0]{\displaystyle}
\newcommand{\cstar}[0]{$\mathrm{C}^*$}
\newcommand{\ann}[0]{\operatorname{Ann}}
\newcommand{\dist}[0]{\operatorname{dist}}
\newcommand{\asdim}[0]{\operatorname{asdim}}
\newcommand{\asdimeins}[0]{\operatorname{asdim}^{\!+1}}
\newcommand{\amdim}[0]{\dim_{\mathrm{am}}}
\newcommand{\amdimeins}[0]{\amdim^{\!+1}}
\newcommand{\dimrokc}[0]{\dim_{\mathrm{Rok}}^{\mathrm{c}}}
\newcommand{\coact}[0]{\operatorname{CoAct}}

\newtheorem{satz}{Satz}[section]		
\newtheorem{cor}[satz]{Corollary}
\newtheorem{lemma}[satz]{Lemma}
\newtheorem{prop}[satz]{Proposition}
\newtheorem{theorem}[satz]{Theorem}
\newtheorem*{theoreme}{Theorem}

\theoremstyle{definition}
\newtheorem{defi}[satz]{Definition}
\newtheorem*{defie}{Definition}
\newtheorem{defprop}[satz]{Definition \& Proposition}
\newtheorem{nota}[satz]{Notation}
\newtheorem*{notae}{Notation}
\newtheorem{rem}[satz]{Remark}
\newtheorem*{reme}{Remark}
\newtheorem{example}[satz]{Example}
\newtheorem{defnot}[satz]{Definition \& Notation}
\newtheorem{question}[satz]{Question}
\newtheorem*{questione}{Question}
\newtheorem{contruction}[satz]{Contruction}

\theoremstyle{plain}
\newcounter{theoremintro}
\newtheorem{thmintro}[theoremintro]{Theorem}
%
%
%
\theoremstyle{definition}
\newtheorem{defnintro}[theoremintro]{Definition}

\newenvironment{bew}{\begin{proof}[Proof]}{\end{proof}}

\numberwithin{equation}{section}
\renewcommand{\theequation}{\thesection\Alph{equation}}

\begin{abstract} 
We introduce the concept of Rokhlin dimension for actions of residually finite groups on \cstar-algebras, extending previous such notions for actions of finite groups and the integers by Hirshberg, Winter and the third author. We are able to extend most of their results to a much larger class of groups: those admitting box spaces of finite asymptotic dimension. This latter condition is a refinement of finite asymptotic dimension and has not been considered previously. In a detailed study we show that finitely generated, virtually nilpotent groups have box spaces with finite asymptotic dimension, providing a large class of examples. We show that actions with finite Rokhlin dimension by groups with finite dimensional box spaces preserve the property of having finite nuclear dimension when passing to the crossed product \cstar-algebra.
We then establish a relation between Rokhlin dimension of residually finite groups acting on compact metric spaces and amenability dimension of the action in the sense of Guentner, Willett and Yu. We show that for free actions of infinite, finitely generated, nilpotent groups on finite dimensional spaces, both these dimensional values are finite. In particular, the associated transformation group \cstar-algebras have finite nuclear dimension. This extends an analogous result about $\IZ^m$-actions by the first author to a significantly larger class of groups, showing that a large class of crossed products by actions of such groups fall under the remit of the Elliott classification programme. 
We also provide results concerning the genericity of finite Rokhlin dimension, and permanence properties with respect to the absorption of a strongly self-absorbing \cstar-algebra.
\end{abstract}

\maketitle



\tableofcontents

\section*{Introduction}

\renewcommand*{\thetheoremintro}{\Alph{theoremintro}}

\noindent
Since its very beginning, the theory of operator algebras has been strongly influenced and even motivated by ideas of dynamical nature. The study of \cstar-dynamical systems, i.e.~group actions of locally compact groups on \cstar-algebras, is interesting in many ways. An ample reason for this is certainly the crossed product construction, which naturally associates a new \cstar-algebra  to a \cstar-dynamical system, reflecting the structure of the group, the coefficient algebra, and the action. This construction has, by now, proved to be a virtually inexhaustible source of interesting examples of \cstar-algebras. As such, it is not surprising that the crossed product construction is harder to understand than most other standard constructions to create new \cstar-algebras. Moreover, crossed products are often simple \cstar-algebras of importance in the Elliott classification programme.
In view of the astonishing progess of \cstar-algebra classification over the last few years~\cite{GongLinNiu14, ElliottNiu15, ElliottGongLinNiu15, TikuisisWhiteWinter15}, which is largely driven by the discovery of various regularity properties (see for example~\cite{Winter14Lin, Winter10, Winter12, RobertTikuisis14, MatuiSato12, KirchbergRordam12, TomsWhiteWinter12, Sato13, MatuiSato14UHF, Winter14, SatoWhiteWinter14, BBSTWW}), a natural and important question appears, which can be regarded as the main motivation of this paper: 

\begin{questione} Given a \cstar-dynamical system $\alpha: G\curvearrowright A$, under what conditions on $\alpha$ does regularity pass from $A$ to $A\rtimes_\alpha G$?
\end{questione}

Of course, the vague term `regularity' leaves some room for interpretation, but this question is interesting for all possible versions of regularity. Within this paper, we will mainly focus on the regularity property of having finite nuclear dimension, which is an approximation property stronger than nuclearity. We will also have a secondary focus on $\CD$-stability in section 9, for a chosen strongly self-absorbing \cstar-algebra $\CD$. 

General results on nuclear dimension of crossed products are hard to come by. What seems to be required is a certain regularity property of the action.
As evidenced by recent developments, Rokhlin-type properties appear to be very good candidates for such a regularity property. The idea of defining a Rokhlin property stems from a basic early result in the dynamics of group actions on measure spaces, the so-called Rokhlin Lemma. This result states that a measure-preserving, aperiodic integer action can be approximated by cyclic shifts in a suitable sense. A von Neumann theoretic translation of this approximation was used successfully by Connes as a key technique towards classifying automorphisms on the hyperfinite II${}_1$-factor up to outer conjugacy, see~\cite{Connes75, Connes77}. Since then, lots of generalizations and reinterpretations have emerged within the realm of von Neumann algebras. Preliminary variants of the Rokhlin property emerged in works of Herman-Jones~\cite{HermanJones82} and Herman-Ocneanu~\cite{HermanOcneanu84} for cyclic group actions on UHF algebras.
Eventually, Kishimoto was successful in fleshing out the Rokhlin property for integer actions on \cstar-algebras, see for instance~\cite{Kishimoto95, Kishimoto96, Kishimoto96_R, Kishimoto98, Kishimoto98II}. Then Izumi started his pioneering work on finite group actions with the Rokhlin property, see~\cite{Izumi04, Izumi04II}.
For a selection of other papers related to Rokhlin-type properties for \cstar-dynamical systems, see~\cite{Phillips11, Phillips12, OsakaPhillips12, Santiago14, Nakamura99, Sato10, MatuiSato12_2, MatuiSato14, Kishimoto02, BratteliKishimotoRobinson07, HirshbergWinter07, Gardella14_5}. (This list is not exhaustive)

As powerful as these earlier Rokhlin properties are, most variants share a common disadvantage. As their definition involves Rokhlin towers consisting of projections, they impose strong restrictions on the \cstar-algebra in question, ruling out important examples such as the Jiang-Su algebra $\CZ$ or any other sufficiently projectionless \cstar-algebra. To circumvent this problem, Hirshberg, Winter and the third author introduced the concept of Rokhlin dimension for actions of finite groups and the integers in~\cite{HirshbergWinterZacharias14}. See also~\cite{Szabo14} for an extension to $\IZ^m$-actions. The concept of Rokhlin dimension yields a generalization of many of the commonly used Rokhlin-type properties, motivated by the idea of covering dimension. Instead of requiring one \mbox{(multi-)} tower consisting of projections to reflect a shift-type behaviour in the given dynamical system, one allows boundedly many of such towers consisting of positive elements, the bound defining the dimensional value. It turns out that this generalization provides a much more flexible concept, while still being strong enough to be compatible with the regularity property of having finite nuclear dimension. For a selection of other papers that, at least in large part, emerged out of the ideas within~\cite{HirshbergWinterZacharias14}, see~\cite{BarlakEndersMatuiSzaboWinter, Szabo14, Gardella14_3, Gardella14_4, HirshbergPhillips15, Liao15, Liao16, HirshbergWu15, HirshbergSzaboWinterWu16, BrownTikuisisZelenberg16}.

The main results in~\cite{HirshbergWinterZacharias14} are a bound on the nuclear dimension of crossed products by single automorphisms (i.e.\ $\IZ$-actions) of finite Rokhlin dimension, a bound for the Rokhlin dimension for minimal actions on finite dimensional compact metric  spaces, as well as a permanence result for $\CZ$-stability and a genericity result for finiteness of Rokhlin dimension in the set of all actions on a given algebra. The first two results were generalized by the first author to $\IZ^m$-actions in~\cite{Szabo14}, who also gave a better framework for establising a bound for the Rokhlin dimension on compact metric spaces requiring only freeness of the action. This used some crucial topological arguments inspired by~\cite{Lindenstrauss95, Gutman14}, the result of which can be regarded as a topological version of the classical Rokhlin Lemma. Combined with the nuclear dimension bound, it shows that \cstar-algebras associated to such free and minimal dynamical systems have finite nuclear dimension and thus fall within the scope of Elliott's classification programme. Such nuclear dimension results were first obtained by Toms and Winter in~\cite{TomsWinter13} by different methods.

The main purpose of this paper is to generalize the concept of Rokhlin dimension to actions of all countable, discrete and residually finite groups and to extend the results from~\cite{HirshbergWinterZacharias14} and~\cite{Szabo14} to this setting. Moreover, we enlarge the class of \cstar-dynamical systems under consideration to non-unital \cstar-algebras, and cocycle actions instead of ordinary actions.
\begin{defnintro}
	Let $A$ be a separable \cstar-algebra, $G$ a countable, residually finite group and $(\alpha,w): G\curvearrowright A$ a cocycle action. The (full) Rokhlin dimension of $\alpha$ is defined as the infimum of natural numbers $d$ such that for any finite-index subgroup $H$ of $G$, there exist equivariant c.p.c.~order zero maps
	\[
	\phi_l: (\CC(G/H), G\text{-shift})\longrightarrow (F_\infty(A), \tilde{\alpha}_\infty) \, ,\quad \text{for}~ l=0,\dots,d
	\]
	with $\phi_0(\eins)+\dots+\phi_d(\eins)=\eins$.
\end{defnintro}
Here the left-hand side involves the canonical action of $G$ on the algebra of functions over the finite quotient $G/H$, while on the right-hand side, $F_\infty(A)$ is Kirchberg's central sequence algebra and $\tilde{\alpha}_\infty$ is the action on $F_\infty(A)$ naturally induced from $\alpha$ (see Definition~\ref{dimrok allgemein}). The definition can be refined by considering only a sufficiently separating sequence $\sigma$ of finite-index subgroups, and the separability of $A$ is not an essential prerequisite for the definition (see Definition~\ref{defi:dimrok}). In fact, the whole setting can be generalized to include also uncountable groups.  We also give an equivalent characterization involving approximate towers of positive elements on the algebraic level, in the spirit of~\cite{HirshbergWinterZacharias14} (see Proposition~\ref{quantifier dimrok}). 

The benefit of this generalization is threefold: it not only allows us to study a wider class of examples, but also unifies all the arguments made (so far separately) for finite groups and the integers, and provides a clearer conceptual picture for how the Rokhlin dimension relates to \cstar-algebraic regularity properties such as the nuclear dimension. The crucial insight is that the applications of the Rokhlin dimension require a complementary ingredient of coarse geometric nature, namely the box space $\square_\sigma G$ (in the sense of Roe, see~\cite[11.24]{RoeCG}) associated to a residually finite group $G$ and some chosen sequence of finite-index subgroups $\sigma = (G_n)_n$. It turns out that a group yields an interesting Rokhlin dimension theory if it admits a box space with finite asymptotic dimension. This insight culminates in Theorem~\ref{Hauptsatz}, where we establish the following upper bound for the nuclear dimension of (twisted) crossed products:

\begin{thmintro}\label{thmintro:dimnuc}
For every \cstar-algebra $A$ and every cocycle action $(\alpha,w): G\curvearrowright A$, we have the upper bound
\[
\dimnuceins(A\rtimes_{\alpha,w} G) \leq \asdimeins(\square_\sigma G)\cdot \dimrokeins(\alpha, \sigma)\cdot\dimnuceins(A) \; .
\]
\end{thmintro}

In the above inequality, the dimensions appearing on the right-hand side are the asymptotic dimension of the box space for $\sigma$, the Rokhlin dimension of $\alpha$ along $\sigma$, and the nuclear dimension of the coefficient algebra $A$, respectively. For notational convenience, each dimension in the formula is increased by $1$, as denoted by the superscripts. By introducing box spaces in this context, this bound is improved and gives a coarse geometric interpretation of the ad-hoc constants appearing in previous such estimates.


The asymptotic dimension and the box space construction have individually played significant roles in coarse geometry and have had remarkable applications to the Baum-Connes conjecture; however, to the best of the authors' knowledge, this is the first time the effects of combining these two notions come under careful investigation. We conduct such an investigation from a general point of view in Section~\ref{sec:coarse-geometry}, even before we can generalize the Rokhlin dimension theory. It turns out that some care has to be taken in the choice of the finite-index subgroups featuring in the box space but that good choices always exist (see Definition~\ref{defi:semi-conj-separating} and Corollary~\ref{cor:semi-conj-separating-exists}). We also give a key characterization of the asymptotic dimension in terms of partitions of unity in Lemma~\ref{decay functions} that is geared towards our main application in Theorem~\ref{Hauptsatz}. 

A central question in Sections~\ref{sec:coarse-geometry} and~\ref{sec:nilpotent-box} is what groups admit box spaces with finite asymptotic dimension, a condition that is vital for the applications of our Rokhlin dimension theory. This is certainly satisfied for finite groups and the integers, and was already (implicitly) crucial in the existing theory for these groups. On the other hand, it follows from classical results in coarse geometry that such a condition implies the amenability of the group. We devote the entire Section~\ref{sec:nilpotent-box} to proving the following (see Theorem~\ref{asdim nilpotent}): 

\begin{thmintro}\label{thmintro:nilpotent-box}
	Every finitely generated, virtually nilpotent group admits a box space with finite asymptotic dimension. 
\end{thmintro}

This provides a reasonably large class of groups to which Theorems~\ref{thmintro:dimnuc} and~\ref{thmintro:D-stability} can be applied. Our proof is done by developing a criterion for a discrete subgroup of a locally compact group to have finite dimensional box spaces that involves the existence of a sequence of expanding automorphisms of the group. For nilpotent groups, we then apply this to known embeddings into certain Lie groups.  For further examples of groups with finite dimensional box spaces, see a preprint by Finn-Sell and the second author~\cite{FinnSellWu14}, which was inspired by and improves the results on box spaces in this paper. Among other things, it is shown there that for every virtually polycyclic group, there exists some box space whose asymptotic dimension is at most the Hirsch length of the group. 

There are further interesting connections of our work to ongoing developments in dimension theory for topological dynamics. Recently, Guentner, Willett and Yu have invented various stronger versions of amenability of discrete group actions on compact metric spaces, see~\cite{GuentnerWillettYu14_1, GuentnerWillettYu15}. In particular, they introduced for such actions a notion that may be termed amenability dimension\footnote{Our choice of this terminology ``amenability dimension'' is based on an early draft of Guentner, Willett and Yu, where they used the term ``$d$-amenable'' to denote what we call having amenability dimension at most $d$. Since then, they have generalized the notion to what they call ``$d$-BLR'', which is closely related to the dynamic asymptotic dimension defined in the same paper (see~\cite[Remark 4.14]{GuentnerWillettYu15}).}. This is a dynamical analogue of asymptotic dimension, the finiteness of which may be regarded as a strong form of amenability of the action. The original motivation to introduce amenability dimension was to facilitate the computation of $K$-theoretic invariants, and in particular to prove the Farrell-Jones conjecture or the Baum-Connes conjecture in certain cases. In fact, their work is in part inspired by the pioneering work of Bartels, L\"uck and Reich on the Farrell-Jones conjecture for hyperbolic groups in~\cite{BartelsLuckReich08}. Somewhat surprisingly, there is a theorem among Guentner-Willett-Yu's results that showcases the compatibility of amenability dimension with nuclear dimension, in a similar fashion as for Rokhlin dimension. This raises the question to what extent amenability dimension and Rokhlin dimension are related concepts. 
After examining amenability dimension in rather \cstar-algebraic terms, we show that for groups with finite dimensional box spaces, the two dimensional invariants are of the same order of magnitude (see Theorem~\ref{dimrok amdim}): 

\begin{thmintro}
	Let $A$ be a \cstar-algebra, $G$ a residually finite group admitting a box space with finite asymptotic dimension, and $(\alpha,w): G\curvearrowright A$ a cocycle action. Then the Rokhlin dimension of $(\alpha,w)$ is finite if and only if its amenability dimension is finite. 
\end{thmintro}

As a major application of this result together with techniques developed in~\cite{Szabo14} and some geometric group theory, we extend the finiteness results for Rokhlin dimension to free actions of infinite, finitely generated, nilpotent groups on finite dimensional compact metric spaces, see Corollary~\ref{free nilpotent dimnuc}. When combined with the astounding recent progress in the Elliott programme by many hands~\cite{GongLinNiu14, ElliottNiu15, ElliottGongLinNiu15, TikuisisWhiteWinter15}, this shows that a large class of simple transformation group \cstar-algebras associated to actions of nilpotent groups is classifiable (see Theorem~\ref{crossed product classifiable}):

\begin{thmintro}\label{thmintro:classification}
	Let $G$ be a finitely generated, infinite, nilpotent group and $X$ a compact metric space of finite covering dimension. Let $\alpha: G\curvearrowright X$ be a free and minimal action. Then the transformation group \cstar-algebra $\CC(X)\rtimes_\alpha G$ is a simple ASH algebra of topological dimension at most $2$.
\end{thmintro}
This considerably extends results which previously have only been known for $\IZ$ and $\IZ^m$-actions. We point out that some time after the initial preprint version of this paper was published, Bartels~\cite{Bartels16} independently proved a similar finiteness result for more general actions of virtually nilpotent groups, by using slightly different methods. As a consequence, Theorem~\ref{thmintro:nilpotent-box} can be extended from nilpotent groups to virtually nilpotent groups (cf.\ Theorem~\ref{thm:Bartels-extension}). 

In the case of noncommutative coefficient algebras, we also explore generalizations of some other results from~\cite{HirshbergWinterZacharias14}. Firstly, we investigate the notion of Rokhlin dimension with commuting towers and its relation to another type of important \cstar-algebraic regularity property \textemdash~\emph{$\CD$-stability} for a strongly self-absorbing \cstar-algebra $\CD$. We show that for groups with finite dimensional box spaces, finite Rokhlin dimension with commuting towers allows $\CD$-stability to pass from the coefficient \cstar-algebra to the (twisted) crossed product (see Theorem~\ref{dimrokc sigma D}):

\begin{thmintro}\label{thmintro:D-stability}
	Let $A$ be a separable, $\CD$-stable \cstar-algebra, let $G$ be a countable, residually finite group admitting a box space with finite asymptotic dimension, and let $(\alpha,w): G\curvearrowright A$ be a cocycle action having finite Rokhlin dimension with commuting towers. Then the twisted crossed product $A\rtimes_{\alpha,w} G$ is $\CD$-stable.
\end{thmintro}


Secondly, we examine the genericity of cocycle actions with small Rokhlin dimension on $\CZ$-stable or UHF-stable \cstar-algebras (see Theorem~\ref{generic dimrok UHF} and Theorem~\ref{generic dimrok Z}): 

\begin{thmintro}\label{thmintro:genericity}
	Let $G$ be a discrete, finitely generated and residually finite group and let $A$ be a separable \cstar-algebra. Then among all cocycle actions $(\alpha,w): G\curvearrowright A$,
	\begin{enumerate}
		\item those with Rokhlin dimension $0$ are generic if $A$ tensorially absorbs the universal UHF algebra; 
		\item those with Rokhlin dimension at most $1$ are generic if $A$ tensorially absorbs the Jiang-Su algebra.  
	\end{enumerate}
\end{thmintro}

The paper is organized as follows. In Section~\ref{sec:prelim}, we fix some notations and review the existing Rokhlin dimension theories for finite groups and $\mathbb{Z}^m$. Section~\ref{sec:coarse-geometry} develops the theory for the asymptotic dimension of box spaces from a general point of view, while Section~\ref{sec:nilpotent-box} focuses on its finiteness for finitely generated virtually nilpotent groups. Section~\ref{sec:Rokhlin} discusses the definitions and basic properties of the Rokhlin dimension for residually finite groups. Section~\ref{sec:dimnuc} elaborates its application to the nuclear dimension of crossed products. Sections~\ref{sec:dimrok-amdim} and~\ref{sec:dimBLR} study the amenability dimension and relate it to the Rokhlin dimension, while Section~\ref{sec:nilpotent-dimnuc} applies our theory to free minimal actions of nilpotent groups on finite dimensional metric spaces. Section~\ref{sec:commuting-towers} investigates Rokhlin dimenison with commuting towers and its application to $\CD$-stability. Section~\ref{sec:genericity} establishes the genericity of finite Rokhlin dimension for actions on $\CD$-stable algebras.

\vspace{3mm}

\textbf{Acknowledgements.} The authors would like to express their gratitude to Rufus Willett and Guoliang Yu for some very helpful discussions. Moreover, the authors are grateful to Siegfried Echterhoff, Stuart White and Wilhelm Winter for pointing out a number of inaccuracies in a previous preprint version of this paper.


\section{Preliminaries}\label{sec:prelim}

\begin{nota}
Unless specified otherwise, we will stick to the following notations throughout the paper.
\begin{itemize}
\item $A$ denotes a \cstar-algebra.
\item $\alpha$ denotes a group action on a \cstar-algebra or a compact metric space.
\item $G$ denotes a countable, discrete group.
\item If $M$ is some set and $F\subset M$ is a finite subset, we write $F\fin M$.
\item For $\eps>0$ and $a,b$ in some normed space, we write $a=_\eps b$ as a shortcut for $\|a-b\|\leq\eps$.
\item Assume that ``$\dim$'' is one of the notions of dimension that appear in this paper, and $X$ an object on which this dimension can be evaluated. Then we sometimes use the convenient notation $\dimeins(X) = 1+\dim(X)$.
\end{itemize}
\end{nota}

First we recall the existing notion of Rokhlin dimension for finite group actions and $\IZ^m$-actions on unital \cstar-algebras:

\defi[cf. {\cite[1.1]{HirshbergWinterZacharias14}}] \label{dimrok finite} 
Let $A$ be a unital \cstar-algebra, $G$ a finite group and let $\alpha: G\curvearrowright A$ be an action.
We say that $\alpha$ has Rokhlin dimension $d$, and write $\dimrok(\alpha)=d$, if $d$ is the smallest natural number with the following property:

For all $F\fin A, \eps>0, n\in\IN$, there exist positive contractions 
$(f_g^{(l)})^{l=0,\dots,d}_{g\in G}$ in $A$ satisfying the following properties:
{
\renewcommand{\theenumi}{\alph{enumi}}
\begin{enumerate}
 \item $\displaystyle \eins_A =_\eps \sum_{l=0}^d \sum_{g\in G} f_g^{(l)} $;
 \item $\|f_{g}^{(l)}f_{h}^{(l)}\|\leq\eps$ for all $l=0,\dots,d$ and $g \neq h$ in $G$;
 \item $\alpha_g(f_{h}^{(l)}) =_\eps f_{gh}^{(l)} $ for all $l=0,\dots,d$ and $g,h\in G$;
 \item $\|[f^{(l)}_g, a]\|\leq\eps$ for all $l=0,\dots,d\; , g\in G$ and $a\in F$.
\end{enumerate}
}
If there is no such $d$, we write $\dimrok(\alpha)=\infty$.

\notae When considering actions of $\IZ^m$ in this section, we denote 
\[
B_n=\set{0,\dots,n-1}^m\subset\IZ^m.
\]

\begin{defi}[cf.~{\cite[1.6]{Szabo14}}] \label{dimrok Z^m} 
Let $A$ be a unital \cstar-algebra, and $\alpha: \IZ^m\curvearrowright A$ an action.
We say that $\alpha$ has Rokhlin dimension $d$, and write $\dimrok(\alpha)=d$, if $d$ is the smallest natural number with the following property:

For all $F\fin A, \eps>0, n\in\IN$, there exist positive contractions 
$(f_v^{(l)})^{l=0,\dots,d}_{v\in B_n}$ in $A$ satisfying the following properties:
{\renewcommand{\theenumi}{\alph{enumi}}
\begin{enumerate}
 \item $\displaystyle \eins_A =_\eps \sum_{l=0}^d \sum_{v\in B_n} f_v^{(l)} $;
 \item $\|f_{v}^{(l)}f_{v'}^{(l)}\|\leq\eps$ for all $l=0,\dots,d$ and $v \neq v'$ in $B_n$;
 \item $\alpha_v(f_{w}^{(l)}) =_\eps f_{v+w}^{(l)} $ for all $l=0,\dots,d$ and $v,w\in B_n$;
 \vspace{1mm}\\ (Note that $f_v^{(l)}$ denotes $f_{(v\mod n\IZ^m)}^{(l)}$, whenever $v\notin B_n$.)
 \item $\|[f^{(l)}_v, a]\|\leq\eps$ for all $l=0,\dots,d\; , v\in B_n$ and $a\in F$.
\end{enumerate}
}
If there is no such $d$, we write $\dimrok(\alpha)=\infty$.
\end{defi}

Using central sequences, there is an elegant reformulation of these definitions. We will omit the easy proofs for the sake of brevity\footnote{However, a more general statement is proved in detail in~\ref{quantifier dimrok}.}. Let $\omega$ be a free filter on $\IN$. As usual, 
let $A_{\omega} = \ell^{\infty}(\IN,A)/c_\omega(\IN,A)$ denote the $\omega$-sequence algebra for a \cstar-algebra $A$. $A$ itself embeds into $A_\omega$ as (representatives of) constant sequences. The relative commutant of $A$ inside $A_\omega$ is called the central sequence algebra of $A$ and is denoted $A_\omega\cap A'$. If $\alpha$ is a group action of a discrete group on $A$, component-wise application of $\alpha$ yields well-defined actions on both $A_\omega$ and $A_\omega\cap A'$, which we will both denote by $\alpha_\omega$. In what follows, we will mostly work with $\omega$ being the filter of all cofinite subsets in $\IN$, and in this case we will insert the symbol $\infty$ in these notations.

\begin{lemma} \label{dimrok alternativ} 
Let $A$ be a separable, unital \cstar-algebra.
\begin{enumerate}
\item Let $\alpha: G\curvearrowright A$ be an action of a finite group. Let $d\in\IN$ be a natural number. Then $\dimrok(\alpha)\leq d$ if and only if there are equivariant c.p.c.~order zero maps
\[
\phi_l: (\CC(G), G\text{-shift}) \longrightarrow (A_\infty\cap A', \alpha_\infty)\quad (l=0,\dots,d)
\]
such that $\phi_0(\eins)+\dots+\phi_d(\eins)=\eins$.
\item Let $\alpha: \IZ^m\curvearrowright A$ be an action. 
Let $d\in\IN$ be a natural number. Then $\dimrok(\alpha)\leq d$ if and only if for all $n\in\IN$, there are equivariant c.p.c.~order zero maps
\[
\phantom{---}\phi_l: (\CC(\IZ^m/n\IZ^m), \IZ^m\text{-shift}) \longrightarrow (A_\infty\cap A', \alpha_\infty)\quad (l=0,\dots,d)
\]
such that $\phi_0(\eins)+\dots+\phi_d(\eins)=\eins$.
\end{enumerate}
\end{lemma}

\begin{rem} The above perspective~\ref{dimrok alternativ} about the existing notions of finite Rokhlin dimension makes the similarities a lot more apparent. Regarding (2), it is not immediately clear why one would have to use the sequence of subgroups $n\IZ^m$, instead of any other suitable separating sequence of finite-index subgroups. When we define Rokhlin dimension for residually finite groups in the upcoming sections, we will introduce a little more flexibility concerning this point.

We note that for integer actions, there have originally been several versions of Rokhlin dimension, see~\cite[2.3]{HirshbergWinterZacharias14}. The above definition for $\IZ^m$-actions~\cite[1.6]{Szabo14} extends the single tower version for $\IZ$-actions, see~\cite[2.3c), 2.9]{HirshbergWinterZacharias14}.
\end{rem}

As we will also treat cocycle actions in this paper, we remind the reader of some basic definitions:

\begin{defi}[cf.~{\cite[2.1]{BusbySmith70} or~\cite[2.1]{PackerRaeburn89}}] 
\label{cocycle actions} 
Let $A$ be a \cstar-algebra and $G$ a discrete group. A cocycle action $(\alpha,w): G\curvearrowright A$ is a map $\alpha: G\to\Aut(A), g\mapsto\alpha_g$ together with a map $w: G\times G\to\CU(\CM(A))$ satisfying
\[
\alpha_r\circ\alpha_s = \ad(w(r,s))\circ\alpha_{rs}
\]
and
\[
\alpha_r(w(s,t))\cdot w(r,st) = w(r,s)\cdot w(rs,t)
\]
for all $r,s,t\in G$.
\end{defi}

One may always assume $\alpha_1 = \id$ and $w(1,t)=w(t,1)=\eins$ for all $t\in G$. If $w=\eins$ is trivial, then this just recovers the definition of an ordinary action $\alpha: G\curvearrowright A$. 

\begin{defi}[cf.~{\cite{PackerRaeburn89}}] 
\label{twisted crossed product}
Let $A$ be a \cstar-algebra and $G$ a discrete group. Let $(\alpha,w): G\curvearrowright A$ be a cocycle action. Then one defines the maximal twisted crossed product $A\rtimes_{\alpha,w} G$ to be the universal \cstar-algebra with the property that it contains a copy of $A$, there is a map $G\to\CU(\CM(A\rtimes_{\alpha,w} G)), g\mapsto u_g$ satisfying
\[
u_g a u_g^* = \alpha_g(a)\quad\text{and}\quad u_gu_h = w(g,h)\cdot u_{gh}
\]
for all $a\in A$ and $g,h\in G$.

The reduced twisted crossed product $A\rtimes_{r,\alpha,w} G$ is defined as the \cstar-algebra inside $\CB(\ell^2(G)\otimes H)$ generated by the following two representations: 

Let $A\into\CB(H)$ be faithfully represented on a Hilbert space. Then consider
\[
\iota: A\to\CB(\ell^2(G)\otimes H)\quad\text{via}\quad \iota(a)(\delta_t\otimes\xi) = \delta_t\otimes \alpha_{t^{-1}}(a)\xi
\]
and
\[
\lambda: G\to\CU(\CB(\ell^2(G)\otimes H))\quad\text{by}\quad \lambda_s(\delta_t\otimes \xi) = \delta_{st}\otimes w(t^{-1}s^{-1},s)\xi.
\]
If $G$ is amenable, then the maximal and reduced twisted crossed products always coincide.
\end{defi}

\rem \label{leftregular} In a more symbolic notation using $A$-valued infinite matrices (that we will use later) one can also write
\[
au_s = \sum_{t\in G} e_{t,s^{-1}t}\otimes \alpha_{t^{-1}}(a)w(t^{-1},s) \quad\in A\rtimes_{r,\alpha,w} G 
\]
for all $a\in A, s\in G$.


\section{Box spaces and asymptotic dimension}\label{sec:coarse-geometry}

\defi Let $G$ be a countable, discrete group. Let $G_n \leq G$ be a decreasing sequence of subgroups with finite index, i.e.\ $[G:G_n]<\infty$ and $G_{n+1} \leq G_n$ for all $n\in\IN$. We say that the sequence $(G_n)_n$ is separating if $\bigcap_{n\in\IN} G_n = \set{1_G}$. $G$ is (by definition) residually finite if and only if such a sequence exists.

\begin{rem} A well-known fact in group theory is that for any finite-index subgroup $H \leq G$, there is a normal finite-index subgroup $N \leq H$ such that $[G:N] \leq [G:H]!$. It follows that a group $G$ is residually finite if and only if it has a separating decreasing sequence of normal finite-index subgroups. 
\end{rem}

Before we come to the next definition, we need to recall some facts. If $(X,d)$ is a metric space and we have a discrete group $G$ acting on $X$ from the right, then the orbit space $X/G$ is defined as the quotient of $X$ by identifying any two points in the same orbit. Let $\pi : X \to X/G$ be the corresponding quotient map. One can define a quotient pseudometric by 
\[
\pi_*(d) ( \pi (x_1), \pi (x_2)) = \inf \{ d(x_1\cdot g_1, x_2\cdot g_2) \mid g_1, g_2 \in G \}.
\]
In good cases, this will be a metric inducing the quotient topology, such as for geometric actions, i.e.~isometric, properly discontinuous and cobounded actions, see~\cite[3.20]{DrutuKapovich}.

Recall also the definition of a coarse disjoint union of finite metric spaces~\cite[1.4.12]{NowakYuLSG}. Moreover, recall that a metric on a discrete set is called proper, if balls are finite sets.

{\defi[following {\cite[11.24]{RoeCG}} and~\cite{Khukhro12}]\label{defi:box-space} Let $G$ be a countable, discrete, residually finite group and let $\sigma=(G_n)_n$ be a decreasing sequence of subgroups of $G$. Let us equip $G$ with a proper, right-invariant metric $d$.  For each $n \in \IN$, let $\pi_n : G \to G/ G_n$ be the quotient map and $(\pi_n)_* (d)$ the quotient metric on $G / G_n$. Then the box space $\square_\sigma G$ associated to $\sigma$ is defined as the coarse disjoint union of the sequence of finite metric spaces $(G / G_n , (\pi_n)_* (d))$. More precisely, it is a coarse space whose underlying set is the disjoint union $\bigsqcup_{n\in\IN} G/G_n$, and whose coarse structure is determined by a metric $d^B$ such that for all $n$, the metric $d^B$ restricts to $(\pi_n)_*(d)$ on the subset $G/G_n$, and such that we have 
\[
\dist_{d^B}(G/G_n, G/G_m) \to \infty
\] 
as $n\neq m$ and $n+m\to\infty$. 
}

\rem Box spaces associated to a family of normal finite-index subgroups have seen a fair amount of study. It is known to coarse geometers that any countable, discrete group admits a proper right-invariant metric, and that the choice of the proper right-invariant metric $d$ on $G$ and the metric $d^B$ does not affect the coarse equivalence class of the resulting box space. This justifies the absence of $d$ and $d^B$ in the notation $\square_\sigma G$. For a rigorous and comprehensive proof of this, we refer the reader to~\cite[Section 1.2, Proposition 1.3.3]{Szabo_Diss}. We remark that even though in the statement of~\cite[Proposition 1.3.3]{Szabo_Diss}, the sequence $\sigma$ is assumed to be separating and made up of normal subgroups, the actual proof does not make use of these conditions. 

Although for us, $\sigma$ is not required to consist of normal subgroups, one needs a condition that rules out certain pathological examples where the coarse structure of the box space $\square_\sigma G$ behaves badly with regard to that of $G$ itself.

{\defi[cf.\ {\cite[3.1]{FinnSellWu14}}] \label{defi:semi-conj-separating} 
A decreasing sequence $\sigma$ of finite-index subgroups of $G$ is called \emph{semi-conjugacy-separating} if for any nontrivial element $g \in G$, there is $H \in \sigma$ such that $H \cap g^G H = \varnothing$, where $g^G$ is the conjugacy class of $g$.
}

{
\nota \label{rfa}
For the sake of brevity, we shall call a semi-conjugacy-separating decreasing sequence of finite-index subgroups of $G$ a \emph{regular residually finite approximation} of $G$ or simply a \emph{regular approximation
}. 

{
\lemma \label{lem:semi-conj-separating}
 Let $G$ be a group and $\sigma$ a decreasing sequence of finite-index subgroups of $G$. Then the following are equivalent:
 \begin{enumerate}
  \item \label{lem:semi-conj-item1} $\sigma$ is semi-conjugacy-separating.
  \item \label{lem:semi-conj-separating:trivial-intersection} $\bigcap_{H \in \sigma} \bigcup_{k \in G} k H k^{-1} = \{1\}$.
  \item \label{lem:semi-conj-separating:unbounded-injective-radii} For any finite subset $F \subset G$, there is $H \in \sigma$ such that the quotient map $G \to G / H$ is injective on $F \cdot k$ for any $k \in G$.
 \end{enumerate}
}

\begin{proof}
 (1) $\Leftrightarrow$ (2): Condition \eqref{lem:semi-conj-separating:trivial-intersection} can be rewritten as: for any nontrivial element $g \in G$, there is $H \in \sigma$ such that for any $k \in G$, $g \not \in k H k^{-1}$. But the statement at the end is equivalent to $k^{-1} g k \not \in H$, which is in turn equivalent to $k^{-1} g k H \cap H = \varnothing$. Thus the entire statement is equivalent to \eqref{lem:semi-conj-item1}.
 
(2) $\Rightarrow$ (3): Given any finite subset $F \subset G$ and $g_1, g_2 \in F$ with $g_1 \neq g_2$, then $g_1 ^{-1} g_2 \neq 1$, so by \eqref{lem:semi-conj-separating:trivial-intersection} we can find $H \in \sigma$ such that  $g_1 ^{-1} g_2 \not \in kH k ^{-1} $ for all $k \in G$. With $\sigma$ decreasing and $F$ finite, we can find $H \in \sigma$ which works for all pairs of different elements in $F$ simultanously, thus verifying (3).

(3) $\Rightarrow$ (2): With $F$ and $H$ as in condition \eqref{lem:semi-conj-separating:unbounded-injective-radii}, it follows that $F^{-1} F  \cap  k H k^{-1} = \{1\}$ for all $k \in G$, which implies condition \eqref{lem:semi-conj-separating:trivial-intersection}.
\end{proof}

\begin{rem} Although we are only interested in decreasing sequences, Definition~\ref{defi:semi-conj-separating} and Lemma~\ref{lem:semi-conj-separating} work equally well for any collection of finite-index subgroups of $G$ that is closed under forming intersections.
\end{rem}

After we fix a proper, right-invariant metric on $G$ and replace $F$ by arbitrarily large balls around $1_G$, the last condition in the above Lemma says that the family of maps $\{ G \to G / H \}_{H \in \sigma}$ achieves arbitrarily large injectivity radii. This suggests the importance of this condition regarding the coarse geometric information of the box space and the group.
On the other hand, condition \eqref{lem:semi-conj-separating:trivial-intersection} directly implies the following. 

\begin{cor}
 Every separating family of finite-index normal subgroups is semi-conjugacy-separating. 
\end{cor}

Since a group is residually finite if and only it has a separating family of finite-index \emph{normal} subgroups, we also have:
 
\begin{cor}\label{cor:semi-conj-separating-exists} 
 A group is residually finite if and only if it has a semi-conjugacy-separating family of finite-index subgroups.
\end{cor}

The condition of having large injectivity radii is relevant to the coarse geometry of the box spaces, because here it actually implies the generally stronger condition of having large isometry radii, as shown in the next Lemma: 

{
\lemma \label{lem:injectivity-radius-isometry-radius}
 Let $(X, d_X)$ be a metric space with an isometric right-action of a group $H$, let $(Y, d_Y)$ be the quotient space by the group action with the induced pseudo-metric, and let $\pi: X \to Y$ be the corresponding projection. Let $x \in X$ and $R> 0$ be such that $\pi$ is injective when restricted to $B_{3R}(x)$. Then $\pi|_{B_{R}(x)}$ is an isometry from $B_{R}(x)$ onto $B_R(\pi(x))$.
}
\begin{proof}
 For any $x_1, x_2 \in B_{R}(x)$ and $h_1, h_2 \in H$ with $h_1 \not= h_2$, we have 
\[
d_X(x_1 \cdot h_1, x_2 \cdot h_1) = d_X(x_1 , x_2 ) \leq 2R,
\] 
while $d_X(x_1 \cdot h_1, x_2 \cdot h_2) = d_X(x_1 , x_2 \cdot h_2 h_1^{-1}) > 2R$ because $x_2 \cdot h_2 h_1^{-1}$, being different from but having the same image under $\pi$ as $x_2$, must lie outside $B_{3R}(x)$. It follows that 
\[
d_Y(\pi(x_1) , \pi(x_2) ) = \inf_{h_1, h_2 \in H} d_X(x_1 \cdot h_1, x_2 \cdot h_2) = d_X(x_1 , x_2 ).
\]
 
 On the other hand, for any $y \in B_R(\pi(x))$, we may pick some $x'\in \pi^{-1}(y)$. As 
\[
R \geq d_X(\pi(x), y) = \inf_{h, h' \in H} d_X(x \cdot h, x' \cdot h') = \inf_{h, h' \in H} d_X(x, x' \cdot h' h^{-1}),
\]
we see that $x' \cdot h' h^{-1} \in \pi^{-1}(y) \cap B_{R}(x)$ and thus $\pi|_{B_{R}(x)}$ is onto $B_R(\pi(x))$. This completes the proof. 
\end{proof}

Our main use of this Lemma will be to perform lifts of uniformly bounded covers, a crucial step in the proof of the upcoming main Lemma of this section.

\rem \label{finite_asdim_box_space_implies_amenability} It is well-known that the box spaces of $G$, as coarse metric spaces, encode important properties of $G$. 
For instance, $\square_\sigma G$ has property A if and only if $G$ is amenable (see~\cite[11.39]{RoeCG} and~\cite[4.4.6]{NowakYuLSG})\footnote{Although these authors only studied box spaces associated to families of \emph{normal} subgroups, we remark that the proof of this result requires only a condition of large injectivity radii, as formulated in the previous Lemmas, whence the proofs also work for semi-conjugacy-separating families.}. 
For us, the finiteness of asymptotic dimension is the more relevant condition. Let us briefly recall the relevant definition:\ the asymptotic dimension of a metric space $(X,d)$ is defined to be the smallest non-negative integer $n$ such that for every $R>0$ there exists a uniformly bounded cover  $\CV = \CV^{(0)} \cup \dots \cup \CV^{(n)} $ of $X$ with Lebesgue number $R$ such that each $\CV^{(l)}$ consists of pairwise disjoint sets (one may even assume $R$-disjointness, meaning that any two different sets in $\CV^{(l)}$ have distance bounded below by $R$). 
Finiteness of asymptotic dimension always implies property A, see~\cite[4.3.6]{NowakYuLSG}; hence in the case of box spaces, it also implies the amenability of $G$. 
When a box space has finite asymptotic dimension, one might be tempted to think that this value encodes the complexity of the group in some sense. This is demonstrated in the main technical result of this section:

{
\lemma \label{decay functions}
Let $G$ be a residually finite group. Consider the action of $G$ on itself by right multiplication and fix a proper, right-invariant metric. Let $\sigma=(G_n)_n$ be a decreasing sequence of finite-index subgroups of $G$.
 Then the following conditions are equivalent for all $s\in\IN$:
\begin{enumerate}
  \item $\sigma$ is semi-conjugacy-separating and the box space $\square_\sigma G$ has asymptotic dimension at most $s$.
  \item For any $R>0$, there exists $n$ and a uniformly bounded cover $\CU = \CU^{(0)} \cup \dots \cup \CU^{(s)} $ of $G$ with Lebesgue number at least $R$, such that each $\CU^{(l)} $ is $G_n$-invariant (with regard to the right multiplication) and has mutually disjoint members.
  \item For every $\eps>0$ and $M\fin G$, there exists $n$ and finitely supported functions $\mu^{(l)}: G\to [0,1]$ for $l=0,\dots,s$ satisfying the following properties:
	\begin{enumerate}
	\item For every $l=0,\dots,s$,  one has
\[
\supp(\mu^{(l)})\cap\supp(\mu^{(l)})h=\emptyset\quad\text{for all}~h\in G_n\setminus\set{1}.
\]
	\item For every $g\in G$, one has
\[
\sum_{l=0}^s \sum_{h\in G_n}  \mu^{(l)}(gh)=1.
\]
	\item For every $l=0,\dots,s$ and $g\in M$, one has 
\[
\|\mu^{(l)}-\mu^{(l)}(g\cdot\_\!\_)\|_\infty\leq\eps.
\]
	\end{enumerate}
\end{enumerate}
}

In the style of \cite[Section 2]{HirshbergWinterZacharias14}, we will refer to a system $\set{\mu^{(l)}}_{l=0,\dots,s}$ of functions satisfying the properties in part (3) of Lemma \ref{decay functions} as a system of decay functions.

\reme Compared with the definition of asymptotic dimension, condition (2) can be viewed as a periodic variant of asymptotic dimension. On the other hand, condition (3) reflects the standard fact that covers with large Lebesgue number give rise to very flat partitions of unity, but again in a periodic way.

\begin{proof}[Proof of~\ref{decay functions}]
For the entire proof, let us fix the following notation: For each $n\in\IN$, denote by $\pi_n: G\to G/G_n$ the quotient map. Let $d$ be a proper, right-invariant metric on $G$ and let $d^B$ be an induced metric on $\square_\sigma G = \bigsqcup_{n\in\IN} G/G_n$ as specified in Definition~\ref{defi:box-space}.

$(1)\implies (2):$ Given $R>0$, there is a cover $\CV = \CV^{(0)} \cup \dots \cup \CV^{(s)} $ of $\square_\sigma G$ with Lebesgue number at least $R$, such that the diameters of members of $\CV$ are uniformly bounded by some $R'\geq R$ and each $\CV^{(l)} $ has mutually disjoint members. Let $\CV' = \CV'^{(0)} \cup \dots \cup \CV'^{(s)} $ denote the induced finite cover on the finite subspace $G/G_n \subset \square_\sigma G$ for some $n$.

Applying Lemma~\ref{lem:semi-conj-separating}(3), we may choose $n \in \IN$ such that for any $g \in G$, $\pi_n$ is injective when restricted to $B_{3R'}(g) = B_{3R'}(1_G)g$. It then follows from Lemma~\ref{lem:injectivity-radius-isometry-radius} that $\pi_n$ restricts to an isometry between $B_{R'}(g)$ and $B_{R'}(\pi_n(g))$ for any $g \in G$. Applying Lemma~\ref{lem:semi-conj-separating}(2), let us also assume that
\[
B_{3R'}(1)\cap \bigcup_{k\in G} kG_nk^{-1} = \set{1}.
\]

By uniform boundedness, each $V\in\CV'$ is contained in $B_{R'}(\pi_n(g))$ for some $g \in G$, and thus we may find $U_V \subset B_{R'}(g)$ that is mapped isometrically onto $V$ under $\pi_n$. Hence $\pi_n^{-1}(V)$ can be written as a disjoint union $\bigsqcup_{h\in G_n} U_V \cdot h$. Define 
\[
\CU^{(l)} := \set{ U_V \cdot h ~|~ V \in \CV'^{(l)},\ h \in G_n }
\]
for all $l = 0,\dots,s$. Then each $\CU^{(l)} $ is $G_n$-invariant (with regard to multiplication from the right), uniformly bounded and has mutually disjoint members. Moreover, we claim that $\CU = \CU^{(0)} \cup \dots \cup \CU^{(s)} $ covers $G$ with Lebesgue number at least $R$. Indeed, Let $X \subset G$ be a subset with diameter less than $R$. Since the cover $\CV'$ of $G / G_n$ has Lebesgue number at least $R$ and $\pi_n$ is contractive, there is some $V \in \CV'$ such that $\pi_n(X) \subset V$, whence 
\[
 X \subset \pi_n^{-1}(\pi_n(X)) \subset \pi_n^{-1}(V) = \bigsqcup_{h\in G_n} U_V\cdot h \; .
\]

By construction, the set $U_V$ has diameter at most $R$, and the subgroup $G_n$ satisfies the equation
\[
B_{3R'}(1)\cap \bigcup_{k\in G} kG_nk^{-1} = \set{1}.
\] 
Thus, for $h_1\neq h_2$ in $G_n$ and $g_1,g_2\in U_V$, we have
\[
d(g_1h_1, g_2h_2) = d(g_1h_1h_2^{-1}g_1^{-1}, g_2g_1^{-1}) > d(g_1h_1h_2^{-1}g_1^{-1},1_G)-2R \geq R.
\]
So we get $\dist_d(U_V\cdot h_1,U_V\cdot h_2) > R$. But this implies that $X$ must be entirely contained in $U_V\cdot h\in\CU$ for some $h\in G_n$, which shows the claim. 

$(2)\implies (3):$ Let $\eps>0$ and $M\fin G$ be given.
Choose $R > \frac{2(2s+3)}{\eps}$ big enough so that $M$ is contained in $B_R(1_G)$. By assumption, there exists $n$ and a uniformly bounded cover $\CU = \CU^{(0)} \cup \dots \cup \CU^{(s)} $ of $G$ with Lebesgue number at least $R$, such that each $\CU^{(l)} $ is $G_n$-invariant and has mutually disjoint members. For each $l=0,\dots,s$, upon combining several members of $\CU^{(l)}$ into one, we may assume that $\CU^{(l)}$ is of the form $\set{ U^{(l)} \cdot h ~|~ h \in G_n }$ for some (necessarily finite) set $U^{(l)} \subset G$ such that $ U^{(l)} \cap \left( U^{(l)} \cdot h \right) = \emptyset $ for any $h \in G_n\setminus\set{1}$. Now define 
\[
\mu^{(l)} : G \to [0,1],~ g \mapsto \frac{\dist_d(g, G\setminus U^{(l)})}{ \sum_{V\in\CU} \dist_d(g, G\setminus V) } .
\]
It is easy to see that $\supp(\mu^{(l)}) = U^{(l)} $ and $\set{ \mu^{(l)}(\_\!\_\cdot h) ~|~ h\in G_n ,~ l=0,\dots,s }$ forms a partition of unity for $G$. This proves properties (a) and (b). By~\cite[4.3.5]{NowakYuLSG}, each $\mu^{(l)}$ is Lipschitz with constant $\frac{2(2s+3)}{R} \leq \eps$, i.e. (c) is satisfied.

$(3)\implies(1):$ Given $R>0$, pick $\eps>0$ with $\eps<\frac{1}{(s+1)}$. Choose $n$ and finitely supported functions $\mu^{(l)} : G \to [0,1]$ for $l=0,\dots,s$ satisfying the requirements of (3) with respect to the pair $\eps$ and $M=B_R(1_G)$. By choosing $n$ large enough, we may also assume that the distance between any two of the subsets $\bigcup_{k=1}^{n-1} G/G_k$ and $G/G_m$ for $m\geq n$ is at least $R$ in $\square_\sigma G$. 
Define $U^{(l)} := \supp(\mu^{(l)})$. Then $U^{(l)} \cap \left( U^{(l)} \cdot h \right) = \emptyset$ for all $l = 0,\dots,s$ and $h \in G_n \setminus\set{1}$, and $\set{ U^{(l)} \cdot h ~|~ h \in G_n ,~ l=0,\dots,s}$ covers $G$. 

We claim that for any $g \in G$, there are $l \in \{0,\dots,s \}$ and $h \in G_n$ such that $B_{R}(g)\subset U^{(l)}h$. Indeed, since $g$ is in the support of at most $(s+1)$ members of the partition of unity $\set{ \mu^{(l)}(\_\!\_\cdot h) ~|~ h\in G_n ,~ l=0,\cdots,s }$ (at most one for each $l$), it follows that there exist $h\in G_n$ and $l\in \{0,\dots,s \}$ such that $\mu^{(l)}(g\cdot h^{-1}) > \frac{1}{s+1}$.  Thus for any $g'\in B_R(g)$, 
\[
\mu^{(l)}(g'\cdot h^{-1}) \geq \mu^{(l)}(g\cdot h^{-1}) - \eps > \frac{1}{s+1} - \frac{1}{s+1} =0.
\]
This shows that $B_R(g) \subset U^{(l)} \cdot h$. 

Note that in particular, this claim implies that $\pi_n$ is injective when restricted to $B_R(g)$. Since $R$ and $g$ are arbitrary and $n$ only depends on $R$, it follows from Lemma~\ref{lem:semi-conj-separating}(3) that $\sigma$ is semi-conjugacy-separating. 

Now for each $m \geq n$, define the collection $\CV^{(l)}_m := \set{ \pi_m (U^{(l)} \cdot h) ~|~ h \in G_n }$, which consists of $[G_n : G_m]$ disjoint subsets of $G/G_m$. Define a cover $\CV = \CV^{(0)} \cup \dots \cup \CV^{(s)} $ of $\square_\sigma G$ by 
\[
\CV^{(0)} := \set{ \bigcup_{k=1}^{n-1} G/G_k } \sqcup \bigcup_{m=n}^{\infty} \CV^{(0)}_m ~\text{and}~ \CV^{(l)} := \bigcup_{m=n}^{\infty} \CV^{(l)}_m\quad\text{for}~l=1,\dots,s.
\]
The diameters of its members are bounded by 
\[
\max \set{ \diam\Big( \bigcup_{m=1}^{n-1} G/G_m \Big), \diam(U^{(0)}), \dots, \diam(U^{(s)}) }
\]
and each $\CV^{(l)}$ consists of disjoint sets. Finally let us show that its Lebesgue number is at least $R$. Given any non-empty subset $X \subset \square_\sigma G$ with radius at most $R$, if we fix any $x \in X$, then $X \subset B_R(x)$. By our choice of $n$, we know that $B_R(x)$ falls entirely in one of the subsets $\bigcup_{k=1}^{n-1} G/G_k$ or $G/G_m$ for $m\geq n$. In the first case, $B_R(x) \subset \bigcup_{k=1}^{n-1} G/G_k \in \CV^{(0)}$. In the case where $B_R(x) \subset G/G_m$ for $m\geq n$, choose some $g \in G$ such that $\pi_m(g) = x$. By the claim we proved above, there are $l \in \{0,\dots,s \}$ and $h \in G_n$ such that $B_{R}(g)\subset U^{(l)}h$, whence $X \subset B_R(x) \subset \pi_m(U^{(l)} \cdot h) \in  \CV^{(l)}$. Therefore $\CV = \CV^{(0)} \cup \dots \cup \CV^{(s)}$ is a uniformly bounded cover of $\square_\sigma G$ with Lebesgue number at least $R$ and each $\CV^{(l)}$ made up of disjoint members, which shows that $\asdim(\square_\sigma G)\leq s$.
\end{proof}

{
\cor \label{asdim asdimbox}
Let $G$ be a countable, discrete, residually finite group and $\sigma=(G_n)_n$ a regular approximation. Let $G$ be equipped with some proper, right-invariant metric, which gives it the structure of a coarse metric space.
Then $\asdim(G)\leq\asdim(\square_\sigma G)$.
}
\begin{proof}
This follows immediately from~\ref{decay functions}(2).
\end{proof}

{
\cor \label{box subgroup}
Let $G$ be a countable, discrete, residually finite group and $\sigma=(G_n)_n$ a regular approximation. Let $H\subset G$ be a subgroup. Then 
\begin{enumerate}
 \item $\kappa=(H\cap G_n)_n$ defines a regular approximation of $H$;
 \item we have $\asdim(\square_\kappa H) \leq \asdim(\square_\sigma G)$;
 \item if $H\subset G$ has finite index, then $\square_\kappa H$ and $\square_\sigma G$ are coarsely equivalent, and in particular $\asdim(\square_\kappa H) =\asdim(\square_\sigma G)$.
\end{enumerate}
}
\begin{proof}
 To prove (1), we observe that for any $n \in \IN$, we have a natural embedding $H / (H\cap G_n) \subset G / G_n$, and for any $h \in H$, we have $h^H \subset h^G$. Consequently $\kappa$ is also semi-conjugacy-separating by definition.
 
 Statement (2) follows directly from~\ref{decay functions}(2) by restricting the obtained cover to $H$.
 
 As for (3), we note that given any proper, right-invariant metric $d$ on $G$, the finite-index condition implies that there exists $R>0$ such that $H$ is an $R$-net in $G$: for any $g \in G$, there exists $h \in H$ such that $d(g, h) \leq R$. It follows that as quotient metric spaces, $H / (H\cap G_n) \subset G / G_n$ is an $R$-net for every $n \in \IN$. Fixing a concrete metric space realization of $\square_\sigma G$, we may realize $\square_\kappa H$ as the metric subspace $\bigsqcup_{n\in \IN} H / (H\cap G_n) \subset \bigsqcup_{n\in \IN} G / G_n$, which is then also an $R$-net. Since an $R$-net is always coarsely equivalent to the ambient space, the assertion is then a consequence of the coarse invariance of asymptotic dimension. 
\end{proof}

\begin{defi} 
\label{domi}
Let $G$ be a countable, discrete and residually finite group. The set of all decreasing sequences of finite-index subgroups of $G$ carries a natural preorder: we say $(H_n)_n$ dominates  $(G_n)_n$, and write $(G_n)_n\precsim (H_n)_n$, if for all $n\in\IN$, there exists $m\in\IN$ with $H_m\subset G_n$. We say that two sequences are order equivalent, and write $(G_n)_n\sim (H_n)_n$, if $(G_n)_n\precsim (H_n)_n$ and $(H_n)_n\precsim (G_n)_n$. 

We denote by $\Lambda(G)$ the set of all regular approximations (cf.~\ref{rfa}) of $G$. It follows from Lemma~\ref{lem:semi-conj-separating}(2) that $\Lambda(G)$ is upward closed: if $(G_n)_n\in\Lambda(G)$ and $(G_n)_n\precsim (H_n)_n$, then $(H_n)_n\in\Lambda(G)$. 

We call a regular approximation $(G_n)_n\in\Lambda(G)$ dominating, if it is dominating with respect to the above order, i.e.~$(H_n)_n\precsim (G_n)_n$ for all $(H_n)_n\in\Lambda(G)$. This is the case if and only if, for every finite-index subgroup $H\subset G$, there exists $n$ such that $G_n\subset H$. 
\end{defi}


{\example Let us consider $G=\IZ$ from the point of view of the above definition. Every element in $\Lambda(\IZ)$ has the form $(k_n\IZ)_n$ for an increasing sequence of natural numbers with $k_n|k_{n+1}$. One has $(k_n\IZ)_n\precsim (l_n\IZ)_n$ if and only if for all $n$, there exists $m$ such that $k_n|l_m$. 

Recall that a supernatural number $\mathfrak{p}$ is a formal product $\mathfrak{p} = \prod p^{n_p}$, where the product runs over all primes and $n_p \in \IN \cup \{ 0,\infty \}$. Another supernatural number $\mathfrak{q}=  \prod p^{m_p}$ divides $\mathfrak{p}$ if $m_q \leq n_p$ for all $p$. We can then establish the following correspondence.
\[
(\Lambda(G)/_\sim ~,~ \precsim) \cong (\text{supernatural numbers ~,~ dividing relation}),
\]
by sending $(k_n\IZ)_n$ to the supernatural number $\mathfrak{p}$ defined by
\[
p^m | \mathfrak{p}\iff p^m|k_n~\text{for some}~n,
\]
where $p$ is any prime natural number. Hence, a sequence $(k_n\IZ)_n\in\Lambda(\IZ)$ is dominating if and only if every prime power devides some member of the sequence $(k_n)_n$. An immediate example of this is $(n!\cdot\IZ)_n$.
}\\

{\example \label{dominating Z^m}
We may generalize the last statement of the previous example to any $m\in\IN$: the sequence $(n!\cdot\IZ^m)_n$ is a dominating regular approximation of $\IZ^m$.
}
\begin{proof}
It suffices to show that any finite-index subgroup of $\IZ^m$ contains a subgroup of the form $n\IZ^m$ for some $n\in\IN$. So let $H\subset\IZ^m$ be a subgroup with finite index. Then $H$ is free abelian with rank $m$. Let $v_1,\dots,v_m$ denote a set generators for $H$. If we view $H$ and $\IZ^m$ as subsets of $\IQ^m$, then we claim that the vectors $v_1,\dots,v_m$ form a $\IQ$-basis. Indeed, let $\lambda_1,\dots,\lambda_m\in\IQ$ be given with $0=\lambda_1v_1+\dots+\lambda_mv_m$. Then we can find a positive integer $k>0$ such that $k\lambda_1,\dots,k\lambda_m\in\IZ$, and we have $0=(k\lambda_1)v_1+\dots+(k\lambda_m)v_m$. Since $H$ is free abelian of rank $m$, it follows that $k\lambda_i=0$ for all $i$, which implies $\lambda_i=0$ for all $i$. This shows that the vectors $v_1,\dots,v_m\in\IQ^m$ are linearly independent in $\IQ^m$, and thus form a basis.

Let $e_1,\dots,e_m$ denote the standard generators of $\IZ^m$. We claim that for each $i=1,\dots,m$, there is a positive integer $k_i$ such that $k_ie_i\in H$. Indeed, we can find some $\lambda_1,\dots,\lambda_m\in\IQ$ with $e_i=\lambda_1v_1+\dots+\lambda_mv_m$. If we choose $k_i$ large enough such that $k_i\lambda_j\in\IZ$ for all $j=1,\dots,m$, then $k_ie_i = (k_i\lambda_1)v_1+\dots+(k_i\lambda_m)v_m\in H$. 

After having found these numbers, let $n$ be the least common multiple of $k_1,\dots,k_m$. Then for all $i=1,\dots,m$, it follows that $\frac{n}{k_i}\in\IZ$ and thus $ne_i=\frac{n}{k_i}\cdot k_i\cdot e_i \in H$. In particular, we have $n\IZ^m\subset H$.
\end{proof}

In applications, the existence of a dominating regular approximation can be quite useful. In the case of finitely generated groups, this turns out to be automatic:

{\prop \label{dominating fg}
Let $G$ be a finitely generated and residually finite group. Then $G$ has a dominating regular approximation consisting of normal subgroups.
}
\begin{proof}
Let $S\fin G$ be a finite generating set. Given some finite group $E$, there are at most as many group homomorphisms from $G$ to $E$ as there are maps from $S$ to $E$, of which there are $|E|^{|S|}$. Now up to isomorphism, there are only countably many finite groups. Since any normal subgroup of $G$ with finite index arises as a kernel of a homomorphism into a finite group, this implies that $G$ has at most countably many normal subgroups of finite index. Let $\set{N_n}_{n\in\IN}$ be an enumeration of this set, and define $(G_n)_n\in\Lambda(G)$ recursively via $G_1=N_1$ and $G_{n+1}=G_n\cap N_{n+1}$ for all $n\in\IN$. By construction, any normal subgroup must contain a member of this sequence as a subgroup, so this sequence is indeed dominating. 
\end{proof}

\begin{rem} \label{Lambda(G)}
Let $(G_n)_n, (H_n)_n\in\Lambda(G)$ be two sequences with $(G_n)_n\precsim (H_n)_n$.
Then it follows from~\ref{decay functions}(2) that $\asdim(\square_{\set{H_n}} G)\leq\asdim(\square_{\set{G_n}} G)$. In other words, the map
\[
\Lambda(G)\longrightarrow\IN,\quad \sigma\longmapsto \asdim(\square_\sigma G)
\]
is order-reversing.

This implies that the values of asymptotic dimension for all box spaces associated to dominating sequences are the same, and they take the lowest value among all possible box spaces associated to decreasing, separating sequences of finite-index subgroups. 

If $(H_n)_n\in\Lambda(G)$ is any dominating sequence, we will call the associated box space $\square_{\set{H_n}} G$ a \emph{standard} box space, and will sometimes denote it by $\square_s G$ when $G$ is finitely generated. While there might a priori be some ambiguity, we will exclusively be interested in the asymptotic dimension of these spaces, so the above argument shows that there is no ambiguity concerning the asymptotic dimension of a standard box space. 
\end{rem}


\section{Box spaces of nilpotent groups}\label{sec:nilpotent-box}

In this section, we show that the standard box space of a finitely generated, virtually nilpotent groups has finite asymptotic dimension. Recall that a group $G$ is called nilpotent if it has a central series of finite length, i.e.\ there is a sequence of subgroups $ \set{1_G} = G_0 \trianglelefteq G_1 \trianglelefteq \dots\trianglelefteq G_r = G $ such that $[G, G_{i}] \leq G_{i-1}$ for $i = 1, \cdots, r$. Furthermore, a group is called virtually nilpotent if it contains a nilpotent subgroup of finite index. 

We begin this section by recording a few technical observations:

\begin{lemma} \label{right-invariant cover}
Let $G$ be a locally compact group with a proper, right-invariant metric $d$. Let $H\subset G$ be a discrete, cocompact subgroup. Let $G$ have finite covering dimension $m\in\IN$. Then there exists a uniformly bounded, open cover $\CW$ of $G$ with positive Lebesgue number and a decomposition $\CW=\CW^{(0)}\cup\dots\cup\CW^{(m)}$, such that for each $l=0,\dots,m$, the collection $\CW^{(l)}$ consists of mutually disjoint members and is $H$-invariant with respect to multiplication from the right. 
\end{lemma}
\begin{proof}
Consider the quotient map $\pi: G\to G/H$, which is a local homeomorphism. Then $m=\dim(G/H)$. Since $H\subset G$ is discrete, there exists $\eta>0$ with $d(h,1_G)\geq 3\eta$ for all $h\in H\setminus\set{1_G}$. By right-invariance of $d$, this implies $d(h_1,h_2)\geq 3\eta$ for all $h_1\neq h_2$ in $H$.

For every $x\in G/H$, choose $g\in G$ with $x=\pi(g)$, and then choose an open neighbourhood $U_x$ of $g$ with diameter at most $\eta$. Now the collection of the images of these sets $\set{V_x}_{x\in G/H}$ is an open cover of $G/H$. By compactness, there exists some finite subcover. Using $m=\dim(G/H)$, we can choose an open cover $\CU$ of $G/H$ refining $\set{V_x}_{x\in G/H}$ with the following properties:
\begin{itemize}
\item There is a decomposition $\CU=\CU^{(0)}\cup\dots\cup\CU^{(m)}$ such that for each $l=0,\dots,m$, one has $U\cap V=\emptyset$ for all $U\neq V$ in $\CU^{(l)}$.
\item For all $V\in\CU$, there exists an open set $U_V\subset G$ with $\pi(U_V)=V$ and $\diam(U_V)\leq\eta$.
\end{itemize}
Since $G/H$ is compact, we can apply the Lebesgue number theorem (see~\cite[7.2.12]{Sutherland}) to deduce that this cover has some positive Lebesgue number $\mu>0$ with respect to the push-forward metric $\pi_*(d)$. Set $r=\min(\mu,\eta)$.

Now consider the uniformly bounded, open cover $\CW=\CW^{(0)}\cup\dots\cup\CW^{(m)}$ of $G$ given by
\[
\CW^{(l)} = \set{ U_V\cdot h ~|~ V\in\CU^{(l)}, h\in H }
\]
for all $l=0,\dots,m$.

\paragraph{\bf Claim 1:} For each $l=0,\dots,s$, the members in $\CW^{(l)}$ are mutually disjoint.

Indeed, let $W_1\neq W_2$ in $\CW^{(l)}$. Let $V_1,V_2\in\CU^{(l)}$ and $h_{1}, h_{2}\in H$ with $W_1=U_{V_1}h_{1}$ and $W_2=U_{V_2}h_{2}$.

\begin{itemize}
\item[Case 1:] $V_1\neq V_2$. In this case, we get $\pi(W_1)=V_1$ and $\pi(W_2)=V_2$, which are disjoint, so $W_1$ and $W_2$ have to be disjoint.
\item[Case 2:] $h_1\neq h_2$. Then $d(h_1,h_2)\geq 3\eta$. Now $W_1$ is an open neighbourhood of $h_1$ with diameter at most $\eta$, and likewise $W_2$ is an open neighbourhood of $h_2$ with diameter at most $\eta$. It follows by the triangle inequality that $W_1$ and $W_2$ have distance at least $\eta$, so they must be disjoint.
\end{itemize}

\paragraph{\bf Claim 2:} The Lebesgue number of $\CW$ is at least $r$.

Let $X\subset G$ be some set of diameter less than $r$ with respect to $d$. Then $\pi(X)$ has diameter less than $r$ with respect to $\pi_*(d)$. Since $r\leq\mu$, there is some $V\in\CU$ with $\pi(X)\subset V$. But then we have
\[
X\subset\pi^{-1}(\pi(X))\subset\pi^{-1}(V) = \bigcup_{h\in H} U_V\cdot h.
\]
As we have observed earlier, for all $h_1\neq h_2$, the distance between $U_Vh_1$ and $U_Vh_2$ is at least $\eta$ with respect to $d$. Since $r\leq\eta$, it follows that $X$ must be contained in exactly one set of the form $U_Vh$, which is a member of $\CW$.
\end{proof}

The following is the key technical Lemma that will allow us to prove that finitely generated virtually, nilpotent groups have  finite-dimensional box spaces:

\begin{lemma} \label{sufficient crit}
Let $G$ be a locally compact group $G$ with a proper, right-invariant metric $d$. Let $H\subset G$ be a discrete and cocompact subgroup. Suppose that there exists a sequence of continuous automorphisms $\sigma_n\in\Aut(G)$ satisfying
{
\renewcommand{\theenumi}{\alph{enumi}}
\begin{enumerate}
 \item for all $n\in\IN$, the map $\sigma_n$ restricts to an endomorphism on $H$;
 \item for every compact set $K$ and open neighbourhood $U$ of $1_G$, there exists $n\in\IN$ with $K\subset (\sigma_n \circ \sigma_{n-1} \circ \ldots \circ \sigma_0) (U)$.
\end{enumerate}
}
Then $H$ is residually finite and admits a regular approximation $\sigma$ with $\asdim(\square_\sigma H)\leq \dim(G)$. In particular, if $H$ is finitely generated, then $\asdim(\square_s H)\leq\dim(G)$.
\end{lemma}
\begin{proof}
Observe that each subgroup $H_n= (\sigma_n \circ \sigma_{n-1} \circ \ldots \circ \sigma_0) (H) \subset H$ must have finite index in $H$. Indeed, since $H_n$ is the image of $H$ under an automorphism on $G$, the subgroup $H_n$ is cocompact in $G$. Let $K\subset G$ be a compact subset with $G=\bigcup_{h\in H_n} K\cdot h$. Let $\pi_n: G\to G/H_n$ be the quotient map. Then the restriction of $\pi_n|_K$ is surjective. Now we may identify $H/H_n$ as a subset of $G/H_n$ in the obvious way, and then the restriction of $\pi_n$ yields a surjective map from $K\cap H$ onto $H/H_n$. But since $K\cap H$ is the intersection of a compact set and a discrete set, it is finite, and thus $H/H_n$ is finite. Since $H$ is discrete, there exists an open neighbourhood $U\subset G$ of the unit with $U\cap H=\set{1_G}$. Then given any predescribed finite set $F\fin H$ containing the unit, there exists $n$ with $F\subset (\sigma_n \circ \ldots \circ \sigma_0) (U)$, from which it follows that $F\cap H_n = \set{1}$. As $F$ is arbitrary, it follows that $\bigcap_{n\in\IN} H_n=\set{1}$, and in particular $H$ is residually finite.

Without loss of generality, assume that $G$ has finite covering dimension. Set $m=\dim(G)$. Apply~\ref{right-invariant cover} to choose
a uniformly bounded, open cover $\CW$ of $G$ with positive Lebesgue number $\delta>0$ and a decomposition $\CW=\CW^{(0)}\cup\dots\cup\CW^{(m)}$, such that for each $l=0,\dots,s$, the collection $\CW^{(l)}$ consists of mutually disjoint members and is $H$-invariant with respect to multiplication from the right. 

For each $n$, consider the uniformly bounded cover $\CW_n=\CW_n^{(0)}\cup\dots\cup\CW_n^{(m)}$ of $G$ given by $\CW_n^{(l)} = (\sigma_n \circ \ldots \circ \sigma_0) (\CW^{(l)}) =\set{ (\sigma_n \circ \ldots \circ \sigma_0) (U)~|~U\in\CW^{(l)}}$ for all $l=0,\dots,m$. Then any two distinct members of $\CW_n^{(l)}$ are disjoint. The first property of $(\sigma_n)_n$ ensures that for all $n$ and $l=0,\dots,m$, the collection $\CW_n^{(l)}$ is right-invariant with respect to $H_n=(\sigma_n \circ \ldots \circ \sigma_0)(H)\subset H$. Now let $R>0$. Since the metric $d$ is proper, the second property of the sequence $(\sigma_n)_n$ ensures that we find some $n$ such that 
\[
B_R(1_G)\subset (\sigma_n \circ \ldots \circ \sigma_0) (B_{\delta/3}(1_G)).
\]
Since $\CW$ has Lebesgue number $\delta$, any ball of radius $\delta/3$ is contained in a member of $\CW$. By our choice of $n$, any ball of radius $R$ is therefore contained in a member of $\CW_n$, which shows that the Lebesgue number of $\CW_n$ is at least $R$.

If we now restrict the cover $\CW_n=\CW_n^{(0)}\cup\dots\cup\CW_n^{(m)}$ to a cover on $H$, we see that condition~\ref{decay functions}(2) is met for $R$. Since $R>0$ was arbitrary, we see that $\asdim(\square_{\{H_n\}} H)\leq \dim(G)$. If $H$ is finitely generated, the last part of the claim then follows from~\ref{Lambda(G)}.
\end{proof}

\begin{cor} \label{asdim box Z^m}
For all $m\in\IN$, we have $\asdim(\square_s\IZ^m)=m$.
\end{cor}
\begin{proof}
Apply~\ref{sufficient crit} for $G=\IR^m, H=\IZ^m$ and the sequence of automorphisms given by $\sigma_n(x)=n\cdot x$. This shows $\asdim(\square_s\IZ^m)\leq m$. On the other hand, we have $\asdim(\square_s\IZ^m)\geq\asdim(\IZ^m)=m$.
\end{proof}

\begin{rem} \label{group properties}
Let us now list a few well-known facts about (virtually) nilpotent groups, which in particular make them interesting from the point of view of the previous section:

\begin{enumerate}
\item Finitely generated, virtually nilpotent groups are residually finite. (see~\cite[2.10, 2.13]{Wehrfritz})
\item A subgroup or quotient of a finitely generated, (virtually) nilpotent group is finitely generated and (virtually) nilpotent. (see~\cite[5.35 and 5.36]{Rotman})
\item Any finitely generated, virtually nilpotent group contains a torsion-free nilpotent subgroup of finite index. (see~\cite[10.2.4]{Passman} or~\cite[2.6, p.~22]{Wehrfritz})
\item The center of an infinite, finitely generated, nilpotent group is infinite. (see~\cite[11.4.3(ii)]{Passman})
\end{enumerate}
By~\ref{dominating fg}, it follows that a finitely generated, virtually nilpotent group admits a dominating regular approximation, and thus has a standard box space. 
\end{rem}

Let us now consider a prominent class of examples for nilpotent groups:

\begin{defi}
Let $R$ be a commutative, unital ring and $d\in\IN$. The unitriangular matrix group of size $d$ over $R$ is defined by
\[
U_d(R) = \set{ x=(x_{i,j})_{1\leq i,j\leq d} ~|~ x_{i,j}\in R, x_{i,j} =0 \textup{ for } i>j \textup{ and } x_{i,i}=1 }.
\]
The multiplication is given by 
\[
( x\cdot y )_{i,j} = \sum_{m=i}^j x_{i,m}y_{m,j}\quad\text{for all}~1\leq i<j\leq d,
\]
as the usual matrix product. By slight abuse of notation we will write $ U_d(R) = \set{ x=(x_{i,j})_{1\leq i<j\leq d} ~|~ x_{i,j}\in R}$.
\end{defi}

These groups are known to be finitely generated and nilpotent, with nilpotency class equal to $d-1$.
Some of these groups will play an important role for the remainder of this section. This is due to the following embedding theorem:

{
\theorem[see {\cite[5.2]{Jennings55}} and~\cite{Swan67}] \label{Jennings}
Let $G$ be a finitely generated, torsion-free nilpotent group. Then $G$ embeds as a subgroup into $U_d(\IZ)$ for some $d\geq 2$.
}

\defi \label{scaling action}
Let $r>0$ be a real number. Then the map $\alpha_r: U_d(\IR)\to U_d(\IR)$ defined by
\[
\alpha_r(x)_{i,j} = r^{j-i}x_{i,j}\quad\text{for all}~1\leq i<j\leq d
\]
yields a well-defined endomorphism. Moreover, it is immediate that $\alpha_{rs} = \alpha_r\circ\alpha_s$ for all $r,s>0$ and $\alpha_1=\id$, thus
\[
\alpha: \IR^{>0} \to \Aut(U_d(\IR)),\quad r\mapsto\alpha_r
\]
yields a group action on $U_d(\IR)$.

\begin{rem} \begin{enumerate}
\item Observe that if $r>0$ is an integer, then $\alpha_r$ restricts to an endomorphism on $U_d(\IZ)$.
\item It is a well-known fact that $U_d(\IZ)$ is a discrete and cocompact subgroup in $U_d(\IR)$. In fact, one has $U_d(\IR)=K\cdot U_d(\IZ)$ for the compact set
\[
K = \set{ x=(x_{i,j})_{1\leq i<j\leq d} \in U_d(\IR) ~\Big|~ |x_{i,j}|\leq\frac12 }
\]
in $U_d(\IR)$, see~\cite[8.2.2]{AbbaMosk}.
\end{enumerate}
\end{rem}

\begin{theorem} \label{asdim Ud}
One has $\asdim(\square_s U_d(\IZ))=d(d-1)/2$ for all $d\geq 2$.
\end{theorem}
\begin{proof}
We shall apply~\ref{sufficient crit} for $G=U_d(\IR), m=d(d-1)/2, H=U_d(\IZ)$ and $\sigma_n=\alpha_{n+1}$. Indeed, since $U_d(\IR)$ is homeomorphic to $\IR^{d(d-1)/2}$ by evaluation of the matrix components, it has covering dimension $d(d-1)/2$. Moreover, $U_d(\IZ)$ is a discrete, cocompact subgroup such that each automorphism $\sigma_n$ of $U_d(\IR)$ restricts to an endomorphism on $U_d(\IZ)$. Lastly, let $K\subset U_d(\IR)$ be any compact set and $U\subset U_d(\IR)$ any open neighbourhood of the unit. Then there are positive numbers $R,\delta>0$ such that
\[
K\subset\set{ x=(x_{i,j})_{1\leq i<j\leq d} \in U_d(\IR) ~\Big|~ |x_{i,j}|\leq R }
\]
and
\[
\set{ x=(x_{i,j})_{1\leq i<j\leq d} \in U_d(\IR) ~\Big|~ |x_{i,j}|\leq\delta }\subset U.
\]
Then any natural number $n$ with $(n+1)! \geq\frac{R}{\delta}$ clearly satisfies $K\subset (\sigma_n \circ \ldots \circ \sigma_0) (U)$. Thus, we have verified the conditions in~\ref{sufficient crit}, and thus the claim follows.
\end{proof}

\begin{samepage}
\begin{theorem} \label{asdim nilpotent}
Let $G$ be a finitely generated, virtually nilpotent group. Then $\asdim(\square_s G)<\infty$.
\end{theorem}
\end{samepage}
\begin{proof}
Let $G'$ be a nilpotent subgroup of $G$ with finite index.
By~\ref{group properties}, $G'$ contains a torsion-free group $H$ of finite index, which is then necessarily also finitely generated and nilpotent. 
By combining~\ref{Jennings},~\ref{asdim Ud} and~\ref{box subgroup}, we obtain $\asdim(\square_s H)<\infty$. Since $H$ also has finite index as a subgroup in $G$, it follows from~\ref{box subgroup} that $\asdim(\square_s G)<\infty$.
\end{proof}

{\rem
Using a somewhat more involved approach, one can show by induction that for any elementary amenable group $G$, the dimension $\asdim(\square_s G)$ is bounded above by the Hirsch length of $G$. So for example, it is finite for all virtually polycyclic groups, or the lamplighter group. This result is a considerable improvement of~\ref{asdim nilpotent} and is shown by Finn-Sell and the second author in~\cite{FinnSellWu14}.
}


\section{Rokhlin dimension for residually finite groups}\label{sec:Rokhlin}

\noindent
In this section, we define Rokhlin dimension for cocycle actions of residually finite groups. First, we will recall the definition of a (corrected) relative central sequence algebra in the non-unital case, based on Kirchberg's pioneering work~\cite{Kirchberg04} on central sequences of \cstar-algebras.

\defi[cf.~{\cite[1.1]{Kirchberg04}}] 
Let $A$ be a \cstar-algebra. Let $D\subset A_\infty$ be a \cstar-subalgebra. Consider
\[
A_\infty\cap D' = \set{ x\in A_\infty \mid xd=dx~\text{for all}~d\in D }
\]
and
\[
\ann(D,A_\infty) = \set{ x\in A_\infty \mid xd=dx=0~\text{for all}~d\in D}.
\]
Then $\ann(D,A_\infty)\subset A_\infty\cap D'$ is an ideal, and one defines
\[
F(D,A_\infty) = (A_\infty \cap D')/\ann(D, A_\infty) 
\]
the central sequence algebra of $A$ relative to $D$. One writes $F_\infty(A) = F(A,A_\infty)$ as a shortcut.

\begin{rem} Let $A$ be a \cstar-algebra and $D\subset A_\infty$ a separable \cstar-subalgebra.
\begin{enumerate}
\item Using that $D$ has a countable approximate unit, one can apply a standard reindexation argument and see that there is some positive contraction $e\in A_\infty$ with $ed=d=de$ for all $d\in D$.
In particular, the image of $e$ defines a unit of $F(D,A_\infty)$, hence this is a unital \cstar-algebra.
\item If $(\alpha,w): G\curvearrowright A$ is a cocycle action and $D\subset A$ is $\alpha$-invariant and closed under multiplication with the cocycles $\set{w(g,h)~|~g,h\in G}$, then componentwise application of the family of automorphisms $\set{\alpha_{g}}_{g\in G}$ on representing bounded sequences defines a genuine action $\tilde{\alpha}_\infty: G\curvearrowright F(D,A_\infty)$. To see this, we observe for $x\in A_\infty\cap D'$, $d\in D$ and $g,h\in G$ that
\[
w(g,h)xw(g,h)^*d = w(g,h)x(w(g,h)^*d) = dx = xd.
\]
In particular, $x-w(g,h)xw(g,h)^*\in\ann(D,A_\infty)$, showing that the automorphism on $F(D,A_\infty)$ induced by $\ad(w(g,h))$ is indeed trivial.
\end{enumerate}
\end{rem}

\begin{rem}[cf.~{\cite[1.1]{Kirchberg04}}]
\label{tensoring}
Let $A$ be a \cstar-algebra and $D\subset A$ a separable \cstar-subalgebra. Then there is a canonical $*$-homomorphism
\[
F(D,A_\infty)\otimes_{\max} D \to A_\infty\quad\text{via}\quad \bigl( x+\ann(D,A_\infty) \bigl)\otimes a\mapsto x\cdot a.
\]
The image of $\eins\otimes a$ under this $*$-homomorphism is $a$ for all $a\in D$. 

Now if we assume additionally that $\alpha: G\curvearrowright A$ is an action of a discrete group such that $D$ is $\alpha$-invariant, then the above map becomes an equivariant $*$-homomorphism from $\bigl( F(D,A_\infty)\otimes_{\max} D, \tilde{\alpha}_\infty\otimes\alpha|_D \bigl)$ to $(A_\infty, \alpha_\infty)$.
In particular, it makes sense to multiply elements of $F(D,A_\infty)$ with elements of $D$ to obtain an element in $A_\infty$, and this multiplication is compatible with the naturally induced actions by $\alpha$. We will implicitly make use of this observation throughout this section.
\end{rem}

We shall now turn to the definition of Rokhlin dimension for cocycle actions of residually finite groups. 

\begin{defi}\label{defi:dimrok}
Let $A$ be a \cstar-algebra, $G$ a countable group and $(\alpha,w): G\curvearrowright A$ a cocycle action. Let $H\subset G$ be a finite-index subgroup. Let $d\in\IN$ be a natural number. We say that $\alpha$ has Rokhlin dimension at most $d$ relative to $H$, and write $\dimrok(\alpha, H)\leq d$, if the following holds:

For all separable, $\alpha$-invariant \cstar-subalgebras $D\subset A$ closed under multiplication with $\set{w(g,h)~|~ g,h\in G}$, there exist equivariant c.p.c.~order zero maps
\[
\phi_l: (\CC(G/H), G\text{-shift})\longrightarrow (F(D,A_\infty), \tilde{\alpha}_\infty)\quad(l=0,\dots,d)
\]
with $\phi_0(\eins)+\dots+\phi_d(\eins)=\eins$.
The value $\dimrok(\alpha, H)$ is defined to be the smallest $d \in \IN$ such that $\dimrok(\alpha, H)\leq d$, or $\infty$ if no such $d$ exists.
\end{defi}

Let us consider a few equivalent reformulations of the above definition.

\begin{prop} \label{quantifier dimrok}
Let $A$ be a \cstar-algebra, $G$ a countable group and $\alpha: G\curvearrowright A$ an action. Let $H\subset G$ be a subgroup with finite index and $d\in\IN$ a natural number. Then the following are equivalent:
\begin{enumerate}
\item $\dimrok(\alpha, H)\leq d$.
\item For every separable, $\alpha$-invariant \cstar-subalgebra $D\subset A$, there exist positive contractions $(f_{\bar{g}}^{(l)})_{\bar{g}\in G/H}^{l=0,\dots,d}$ in $A_\infty\cap D'$ satisfying
	\begin{enumerate}
	\item[(2a)] $\left(\sum_{l=0}^d \sum_{\bar{g}\in G/H} f_{\bar{g}}^{(l)} \right)\cdot a = a$ for all $a\in D$;
	\item[(2b)] $f_{\bar{g}}^{(l)}f_{\bar{h}}^{(l)} \in \ann(D,A_\infty)$ for all $l=0,\dots,d$ and $\bar{g}\neq\bar{h}$ in $G/H$;
	\item[(2c)] $\alpha_{\infty,g}(f_{\bar{h}}^{(l)})-f_{ {}_{\quer{gh}} }^{(l)} \in \ann(D,A_\infty)$ for all $l=0,\dots,d,~\bar{h}\in G/H$ and $g\in G$.
	\end{enumerate}
\item For all $\eps>0$ and finite sets $M\fin G$ and $F\fin A$, there are positive contractions $(f_{\bar{g}}^{(l)})_{\bar{g}\in G/H}^{l=0,\dots,d}$ in $A$ satisfying
	\begin{enumerate}
	\item[(3a)] $\left(\sum_{l=0}^d \sum_{\bar{g}\in G/H} f_{\bar{g}}^{(l)} \right)\cdot a =_\eps a$ for all $a\in F$;
	\item[(3b)] $\|f_{\bar{g}}^{(l)}f_{\bar{h}}^{(l)}a\|\leq\eps$ for all $a\in F,~l=0,\dots,d$ and $\bar{g}\neq\bar{h}$ in $G/H$;
	\item[(3c)] $\alpha_g(f_{\bar{h}}^{(l)})\cdot a =_\eps  f_{ {}_{\quer{gh}} }^{(l)}\cdot a$ for all $a\in F,~l=0,\dots,d,~\bar{h}\in G/H$ and $g\in M$;
	\item[(3d)] $\|[f_{\bar{g}}^{(l)},a]\|\leq\eps$ for all $a\in F,~l=0,\dots,d$ and $\bar{g}\in G/H$.
	\end{enumerate}
Moreover, it suffices to check these conditions for $M$ being contained in a given generating set of $G$.
\end{enumerate}
\end{prop}
\begin{proof}
$(1)\implies (2):$ Suppose $\dimrok(\alpha,H)\leq d$. Choose equivariant c.p.c.~order zero maps
\[
\phi_l: (\CC(G/H), G\text{-shift})\longrightarrow (F(D,A_\infty), \tilde{\alpha}_\infty)\quad(l=0,\dots,d)
\]
with $\phi_0(\eins)+\dots+\phi_d(\eins)=\eins$. Then, by definition, each positive contraction of the form $\phi_l(\chi_{\set{\bar{g}}})\in F(D,A_\infty)$ for $\bar{g}\in G/H$ has a representing element $f_{\bar{g}}^{(l)}\in A_\infty\cap D'$ with $f_{\bar{g}}^{(l)}+\ann(D,A_\infty) = \phi_l(\chi_{\set{g}})$. By functional calculus, we may assume that each $f_{\bar{g}}^{(l)}$ is a positive contraction. Then we have 
\[
\phi_0(\eins)+\dots+\phi_d(\eins) = \sum_{l=0}^d\sum_{\bar{g}\in G/H} f_{\bar{g}}^{(l)}+\ann(D,A_\infty),
\] 
and thus the property $\phi_0(\eins)+\dots+\phi_d(\eins)=\eins$ translates to condition (2a) by~\ref{tensoring}. For each $l=0,\dots,d$ and $\bar{g}\neq\bar{h}$ in $G/H$, the characteristic functions $\chi_{\bar{g}},\chi_{\bar{h}}\in\CC(G/H)$ are orthogonal. Since $\phi_l$ is an order zero map, we thus have $f_{\bar{g}}^{(l)}f_{\bar{h}}^{(l)}+\ann(D,A_\infty) = \phi_l(\chi_{\bar{g}})\phi_l(\chi_{\bar{h}}) = 0+\ann(D,A_\infty)$. This implies condition (2b). Since $\phi_l$ is equivariant with respect to the $G$-shift on $G/H$, we get for all $g\in G,~\bar{h}\in G/H$ that
\[
\alpha_{\infty,g}(f_{\bar{h}}^{(l)})+\ann(D,A_\infty) = \tilde{\alpha}_{\infty,g}(\phi_l(\chi_{\set{\bar{h}}})) = \phi_l(\chi_{\set{\quer{gh}}}) =  f_{ {}_{\quer{gh}} }^{(l)}+\ann(D,A_\infty).
\]
This implies condition (2c).

$(2)\implies (1):$ Let $(f_{\bar{g}}^{(l)})_{\bar{g}\in G/H}^{l=0,\dots,d}$ be positive contractions in $A_\infty\cap D'$ satisfying the conditions (2a), (2b) and (2c). Then the same calculations as above, only read in the reverse order, show that conditions (2b) and (2c) imply that for $l=0,\dots,d$, the linear map $\phi_l: \CC(G/H)\to F(D,A_\infty)$ given by $\phi_l(\chi_{\set{\bar{g}}})=f_{\bar{g}}^{(l)}+\ann(D,A_\infty)$ is c.p.c.~order zero and equivariant with respect to the $G$-shift and $\tilde{\alpha}_\infty$. Moreover, condition (2a) implies that $\eins=\phi_0(\eins)+\dots+\phi_d(\eins)$ in $F(D,A_\infty)$.

$(2)\implies (3):$ Let $\eps>0$, $M\fin G$ and $F\fin A$ be given. Let $D\subset A$ be a separable, $\alpha$-invariant \cstar-subalgebra containing $F$.
Choose positive contractions $(f_{\bar{g}}^{(l)})_{\bar{g}\in G/H}^{l=0,\dots,d}$ in $A_\infty\cap D'$ satisfying the conditions (2a), (2b) and (2c). For each $l=0,\dots,d$ and $\bar{g}\in G/H$, the element $f_{\bar{g}}^{(l)}$ has a representing sequence $(f_{\bar{g},n}^{(l)})\in\ell^\infty(\IN,A)$, which we may assume by functional calculus to consist of positive contractions. Then conditions (2a), (2b) and (2c) translate to
\begin{itemize}
	\item $\left(\sum_{l=0}^d \sum_{\bar{g}\in G/H} f_{\bar{g},n}^{(l)} \right)\cdot a \stackrel{n\to\infty}{\longrightarrow} a$ for all $a\in D$;
	\item $f_{\bar{g},n}^{(l)}f_{\bar{h},n}^{(l)}a \stackrel{n\to\infty}{\longrightarrow} 0$ for all $a\in D,~l=0,\dots,d$ and $\bar{g}\neq\bar{h}$ in $G/H$;
	\item $(\alpha_g(f_{\bar{h},n}^{(l)}) - f_{ {}_{\quer{gh},n} }^{(l)})\cdot a \stackrel{n\to\infty}{\longrightarrow} 0$ for all $a\in D,~l=0,\dots,d,~\bar{h}\in G/H$ and $g\in G$.
\end{itemize}
Moreover, the fact that $f_{\bar{g}}^{(l)}\in A_\infty\cap D'$, translates to
\begin{itemize}
	\item $[f_{\bar{g},n}^{(l)},a]\stackrel{n\to\infty}{\longrightarrow} 0$ for all $a\in D,~l=0,\dots,d$ and $\bar{g}\in G/H$.
\end{itemize}
Since $F$ is a subset of $D$, and both $F$ and $M$ are finite, we can choose some large number $n$ such that the elements $(f_{\bar{g},n}^{(l)})_{\bar{g}\in G/H}^{l=0,\dots,d}$ satisfy conditions (3a), (3b), (3c) and (3d).

$(3)\implies (2):$ Let $S\subset G$ a be a generating set of $G$. Assume that (3) holds for all finite subsets $M$ of $S$. Since $S$ is countable, choose an increasing sequence of finite sets $M_n\fin S$ with $S=\bigcup_{n\in\IN} M_n$.

Let $D\subset A$ be a separable, $\alpha$-invariant \cstar-subalgebra. Let $F_n\fin D$ be an increasing sequence of finite subsets whose union $\bigcup_{k\in\IN} F_n\subset D$ is dense. For every $n$, choose positive contractions $(f_{\bar{g},n}^{(l)})_{\bar{g}\in G/H}^{l=0,\dots,d}$ in $A$ satisfying conditions (3a), (3b), (3c) and (3d) for the triple $(\eps,M,F)=(\frac{1}{n},M_n,F_n)$. Set $f_{\bar{g}}^{(l)}=[(f_{\bar{g},n}^{(l)})_n]\in A_\infty$ for each $l=0,\dots,d$ and $\bar{g}\in G/H$. Then by choice of the sequence $(f_{\bar{g},n}^{(l)})_n$, we have:
\begin{itemize}
\item $\left(\sum_{l=0}^d \sum_{\bar{g}\in G/H} f_{\bar{g},n}^{(l)} \right)\cdot a \stackrel{n\to\infty}{\longrightarrow} a$ for all $a\in\bigcup_{k\in\IN} F_k$;
	\item $f_{\bar{g},n}^{(l)}f_{\bar{h},n}^{(l)}a \stackrel{n\to\infty}{\longrightarrow} 0$ for all $a\in\bigcup_{k\in\IN} F_k,~l=0,\dots,d$ and $\bar{g}\neq\bar{h}$ in $G/H$;
	\item $(\alpha_g(f_{\bar{h},n}^{(l)}) - f_{ {}_{\quer{gh},n} }^{(l)})\cdot a \stackrel{n\to\infty}{\longrightarrow} 0$ for all $a\in \bigcup_{k\in\IN} F_k,~l=0,\dots,d,~\bar{h}\in G/H$ and $g\in \bigcup_{k\in\IN} M_k=S$;
	\item $[f_{\bar{g},n}^{(l)},a]\stackrel{n\to\infty}{\longrightarrow} 0$ for all $a\in\bigcup_{k\in\IN} F_k,~l=0,\dots,d$ and $\bar{g}\in G/H$.
\end{itemize}
For the correspond elements in the sequence algebra, we thus have
\begin{enumerate}
	\item[(2a')] $\left(\sum_{l=0}^d \sum_{\bar{g}\in G/H} f_{\bar{g}}^{(l)} \right)\cdot a = a$ for all $a\in\bigcup_{k\in\IN} F_k$;
	\item[(2b')] $f_{\bar{g}}^{(l)}f_{\bar{h}}^{(l)}a = 0$ for all $a\in\bigcup_{k\in\IN} F_k,~l=0,\dots,d$ and $\bar{g}\neq\bar{h}$ in $G/H$;
	\item[(2c')] $\alpha_g(f_{\bar{h}}^{(l)})\cdot a = f_{ {}_{\quer{gh}} }^{(l)}\cdot a$ for all $a\in \bigcup_{k\in\IN} F_k,~l=0,\dots,d,~\bar{h}\in G/H$ and $g\in \bigcup_{k\in\IN} M_k=S$;
	\item[(2d')] $[f_{\bar{g}}^{(l)},a]=0$ for all $a\in\bigcup_{k\in\IN} F_k,~l=0,\dots,d$ and $\bar{g}\in G/H$.
\end{enumerate}
Since the union $\bigcup_{k\in\IN} F_k\subset D$ is dense, it follows by continuity of multiplication that $f_{\bar{g}}^{(l)} \in A_\infty\cap D'$ for all $l=0,\dots,d$ and $\bar{g}\in G/H$, and that the elements $(f_{\bar{g}}^{(l)})_{\bar{g}\in G/H}^{l=0,\dots,d}$ satisfy conditions (2a), (2b) and (2c). 
Since condition (2c') holds for all $g$ in a generating set and this condition passes to products of elements, it follows that condition (2c') even holds for all $g\in G$.
\end{proof}

Applying~\ref{quantifier dimrok}(3), we can now see that our definition of Rokhlin dimension indeed extends the original definitions (cf.~\ref{dimrok finite} and~\ref{dimrok Z^m}) due to Hirshberg, Winter and the third author.

\rem \label{dimrok central sequence}
Let $A,G$ and $(\alpha,w): G\curvearrowright A$ be as above. In the case that $A$ is separable, we make two observations:
\begin{enumerate}
\item The map $g\mapsto \tilde{\alpha}_{g,\infty}\in\Aut(F_\infty(A))$ defines a genuine action, which we denote by $\tilde{\alpha}_\infty$.
\item If $H$ is a finite-index subgroup in $G$, then $\dimrok(\alpha, H)\leq d$ if and only if there exist equivariant c.p.c.~order zero maps 
\[
\phi_l: (\CC(G/H),G\text{-shift})\longrightarrow (F_\infty(A),\tilde{\alpha}_\infty)\quad(l=0,\dots,d)
\]
with $\phi_0(\eins)+\dots+\phi_d(\eins)=\eins$.
\end{enumerate}

\begin{proof}[Proof of (2).]
In this case, $A$ is trivially the maximal separable, $\alpha$-invariant \cstar-subalgebra in itself. The ``only if'' part is obvious. For the ``if'' part, choose positive contractions $(f_{\bar{g}}^{(l)})_{\bar{g}\in G/H}^{l=0,\dots,d}$ in $A_\infty\cap A'$ satisfying the conditions (2a), (2b) and (2c) from~\ref{quantifier dimrok} for $A$ in place of $D$. Given any other separable, $\alpha$-invariant \cstar-subalgebra $D$ of $A$, we then have $f_{\bar{g}}^{(l)}\in A_\infty\cap D'$ and the conditions (2a), (2b) and (2c) hold for this subalgebra. Hence $\dimrok(\alpha,H)\leq d$.
\end{proof}

{
\lemma \label{dimrok monotonie}
Let $A,G$ and $(\alpha,w): G\curvearrowright A$ be as before. If $H_2\subset H_1\subset G$ are two finite-index subgroups, then $\dimrok(\alpha, H_1)\leq\dimrok(\alpha, H_2)$. 
}
\begin{proof}
Since there exists an equivariant and unital $*$-homomorphism
\[
(\CC(G/H_1),G\text{-shift})\into (\CC(G/H_2),G\text{-shift}),
\]
this follows from the definition.
\end{proof}

\defi \label{dimrok allgemein} Let $A$ be a \cstar-algebra, $G$ a countable, residually finite group and $(\alpha,w): G\curvearrowright A$ a cocycle action. Let $\sigma=(G_n)_n\in\Lambda(G)$ be a regular approximation. We define the Rokhlin dimension of $\alpha$ along $\sigma$ as
\[
\dimrok(\alpha, \sigma) = \sup_{n\in\IN}~ \dimrok(\alpha, G_n).
\]
Moreover, we define the (full) Rokhlin dimension of $\alpha$ as
\[
\dimrok(\alpha) = \sup\set{ \dimrok(\alpha, H) ~|~ H\leq G,~ [G:H]<\infty }.
\]

\begin{rem} 
\label{dimrok-domi}
It follows from~\ref{dimrok monotonie} that if $\sigma$ is dominating in the sense of~\ref{domi}, then $\dimrok(\alpha, \sigma)=\dimrok(\alpha)$.
\end{rem}

The next example shows that in general, the value $\dimrok(\alpha, \sigma)$ can indeed depend on $\sigma$. Later in Section~\ref{sec:dimrok-amdim}, however, we will see that this dependence is rather mild for certain groups.

\begin{example} 
Let $p\geq 2$ be a natural number. Consider the UHF algebra 
\[
M_{p^\infty} = \bigotimes_\IN M_p \cong \bigotimes_{n\in\IN} \bigotimes_\IN M_{p^n}.
\]
For all $n\in\IN$, let $u_n\in M_{p^n}$ be the shifting unitary given by
\[
u_n = e_{1,p^n}+\sum_{k=2}^{p^n} e_{k,k-1}.
\]
Consider
\[
\alpha\in\Aut(A)\quad\text{via}\quad\alpha=\bigotimes_{n\in\IN} \bigotimes_\IN \ad(u_n).
\]
This automorphism defines an action of $G=\IZ$. For the sequence of subgroups $G_n=p^n\IZ$, one has $\dimrok(\alpha, G_n)=0$. One can see this by using the equivariant, unital diagonal inclusion $(\CC(\IZ/p^n\IZ),G\text{-shift})\subset (M_{p^n},\ad(u_n))$, the rest follows from the construction of the action.
 However, if $q\geq 2$ is a number that does not divide $p$, it is easy to see that $\dimrok(\alpha, q\IZ)>0$, since otherwise the unit would be divisible by $q$ in the $K_0$-group of $M_{p^\infty}$. But this is not true. In particular, one has $\dimrok(\alpha, H)>0$ for all $H\neq p^n\IZ$.

However, it will be a consequence of~\ref{dimrok amdim} that $\dimrok(\alpha)=1$. We remark that this fact also follows from an early unpublished result by Matui and Sato, see the introduction of~\cite{HirshbergWinterZacharias14}. Moreover, this would also follow from a more recent result due to Liao~\cite{Liao15}.
\end{example}

\section{Nuclear dimension of the crossed product \cstar-algebra}\label{sec:dimnuc}

In this section, we show a permanence property of cocycle actions with finite Rokhlin dimension with respect to \cstar-algebras of finite nuclear dimension, under the assumption that the acting group has a finite dimensional box space. This extends important results from~\cite{HirshbergWinterZacharias14} and~\cite{Szabo14}.
But first, we need a slightly more flexible characterization of nuclear dimension:

{
\prop[cf.\ {\cite[2.5]{TikuisisWinter14} and~\cite[A.4]{HirshbergWinterZacharias14}}] \label{dimnuc sequence} 
Let $A$ be a \cstar-algebra and $r\in\IN$ a natural number. Then $\dimnuc(A)\leq r$, if and only if the following holds:

For all $F\fin A$ and $\delta>0$, there exists a finite dimensional \cstar-algebra $\CF$, a c.p.c. map $\psi: A\to \CF$ and c.p.c. order zero maps $\phi^{(0)},\dots,\phi^{(r)}: \CF\to A_\infty$ such that
\[
a =_\delta \sum_{l=0}^r \phi^{(l)}\circ\psi(a)\quad\text{for all}~ a\in F.
\]
}
\begin{proof}
Since the 'only if' part is trivial, we show the 'if' part.
Let $F\fin A$ and $\delta>0$ be arbitrary. Find $\CF,\psi,\phi^{(0)},\dots,\phi^{(r)}$ as in the assertion. Since the cone over $\CF$ is projective~\cite[10.1.11 and 10.2.1]{Loring}, we can find sequences of order zero maps $\phi_n^{(l)}: \CF\to A$ for $l=0,\dots,r$
with $\phi^{(l)}(x) = [(\phi^{(l)}_n(x))_n]$ for all $x\in\CF$. (This follows from the structure theorem for order zero maps~\cite[2.3 and 3.1]{WinterZacharias09} or~\cite[1.2.4]{Winter09}.) In particular, it follows that
\[
\limsup_{n\to\infty}~\|a-\sum_{l=0}^r\phi^{(l)}_n\circ\psi(a)\|\leq\delta\quad\text{for all}~a\in F.
\]
So there exists $n$ with
\[
a =_{2\delta} \sum_{l=0}^r \phi^{(l)}_n\circ\psi(a)\quad\text{for all}~a\in F.
\]
Since $F$ and $\delta$ were arbitrary, this shows $\dimnuc(A)\leq r$.
\end{proof}

\begin{samepage}
{
\theorem[cf.\ {\cite[4.1]{HirshbergWinterZacharias14} and~\cite[1.10]{Szabo14}}] \label{Hauptsatz}
Let $A$ be a \cstar-algebra, $G$ a countable, residually finite group and $(\alpha,w): G\curvearrowright A$ a cocycle action. Let $\sigma=(G_n)_n$ be a regular approximation of $G$.
Then the following estimate holds:
\[
\dimnuceins(A\rtimes_{\alpha,w} G) \leq \asdimeins(\square_\sigma G)\cdot \dimrokeins(\alpha, \sigma)\cdot\dimnuceins(A).
\]
In particular, if $G$ is finitely generated, we get the estimate
\[
\dimnuceins(A\rtimes_{\alpha,w} G) \leq \asdimeins(\square_s G)\cdot \dimrokeins(\alpha)\cdot\dimnuceins(A).
\]
}
\end{samepage}
\begin{proof}
We may assume that $s=\asdim(\square_\sigma G),~ r=\dimnuc(A)$ and $d=\dimrok(\alpha, \sigma)$ are all finite, or else the statement is trivial. As a first consequence of this assumption, $G$ is seen to be amenable by~\ref{finite_asdim_box_space_implies_amenability}, and thus $A\rtimes_{\alpha,w} G $ is naturally isomorphic to $A\rtimes_{r,\alpha,w} G$.

Let $F\fin A\rtimes_{\alpha,w} G$ and $\delta>0$ be given. In order to show the assertion, we show that there exists a finite dimensional \cstar-algebra $\CF$ with a c.p.~approximation $(\CF,\psi,\phi)$ for $F$ up to $\delta$ on $A\rtimes_{\alpha,w} G$ as in~\ref{dimnuc sequence}.

Since the purely algebraic crossed product is dense in $A\rtimes_{\alpha,w}G$ and is linearly generated by terms of the form $au_g$ for $a\in A$ and $g\in G$, we may assume without loss of generality that $F\subset\set{au_g ~|~ a\in F', g\in M }$ for two finite sets $F'\fin A_1$ and $M\fin G$. 

The square root function $\sqrt{\cdot~}: [0,1]\to [0,1]$ is uniformly continuous. Hence, we may choose $\eps>0$ so small that $|\sqrt{s}-\sqrt{t}|\leq\delta$, whenever $|s-t|\leq\eps$.
Recall that the sequence $(G_n)_n$ admits decay functions in the sense of \ref{decay functions}(3). Hence there exists $n\in\IN$ and functions $\mu^{(j)}: G\to [0,1]$ for $j=0,\dots,s$ such that
\begin{align}
 & \label{eq:D1} \supp(\mu^{(j)})\cap\supp(\mu^{(j)})h=\emptyset && \text{for all}\ j=0,\dots,s \text{\ and\ } h\in G_n\setminus\set{1} ;\\
 & \label{eq:D2} \sum_{j=0}^s \sum_{h\in G_n}  \mu^{(j)}(gh)=1 && \text{for all}\ g\in G ;\\
 & \label{eq:D3} \|\mu^{(j)}-\mu^{(j)}(g^{-1}\cdot\_\!\_)\|_\infty\leq\eps && \text{for all}\ j=0,\dots,s \text{\ and\ } g\in M .
\end{align}

Consider the finite sets
\begin{equation} \label{eq:E1}
 B_n^{(j)} = \supp(\mu^{(j)}),\quad B_n = \bigcup_{j=0}^s B_n^{(j)} \quad\fin G
\end{equation}
and
\[
\tilde{F} = \set{ \alpha_{h^{-1}}(a)w(h^{-1},g)~|~ a\in F',~ g\in M,~ h\in B_n} \fin A.
\]

Let $A$ act faithfully on a Hilbert space $H$ and let $A\rtimes_{\alpha,w} G \cong A\rtimes_{r,\alpha,w} G$ be embedded into $\CB(\ell^2(G)\otimes H)$ as in~\ref{leftregular}. Let $Q_j\in\CB(\ell^2(G)\otimes H)$ be the projection onto the subspace $\ell^2(B_n^{(j)})\otimes H$.
Then $x\mapsto Q_jxQ_j$ defines a c.p.c.~map $\Psi_j: A\rtimes_{r,\alpha,w}G \to M_{|B_n^{(j)}|}(A)$. More specifically, we have for all $a\in A, g\in G$ and $j=0,\dots,s$ that
\begin{equation} \label{eq:E2}
 \begin{array}{lll}
\Psi_j(a u_g)
&=& \displaystyle Q_j\Big[\sum_{h\in G} e_{h,g^{-1}h}\otimes\alpha_{h^{-1}}(a)w(h^{-1},g)\Big]Q_j \\\\
&=&  \displaystyle \sum_{h\in B_n^{(j)}: \atop g^{-1}h\in B_n^{(j)}} e_{h,g^{-1}h}\otimes\alpha_{h^{-1}}(a)w(h^{-1},g) \\\\
&=& \displaystyle \sum_{h\in B_n^{(j)}\cap gB_n^{(j)}} e_{h,g^{-1}h}\otimes\alpha_{h^{-1}}(a)w(h^{-1},g). 
\end{array}
\end{equation}

For $j=0,\dots,s$, define the diagonal matrices $D_j\in M_{|B_n|}(\IC)$ by $(D_j)_{h,h} = \mu^{(j)}(h)$. By our choice of $\eps$ and \eqref{eq:D3}, we have
\begin{equation} \label{eq:E3}
 \| \sqrt{\mu^{(j)}}-\sqrt{\mu^{(j)}(g^{-1}\cdot\_\!\_)}\|_\infty\leq\delta\quad\text{for all}\; g\in M.
\end{equation}
It follows for all $g\in M$ and $a\in A$ that
\[
\def\arraystretch{2}
\begin{array}{lcl}
\multicolumn{3}{l}{\|[\sqrt{D_j},Q_jau_gQ_j]\|} \\
\hspace{2mm}&\stackrel{\eqref{eq:E2}}{=}& \displaystyle \left\| \sum_{h\in B_n\cap gB_n} \left( \sqrt{\mu^{(j)}(h)}-\sqrt{\mu^{(j)}(g^{-1}h)} \right) e_{h,g^{-1}h}\otimes \alpha_{h^{-1}}(a)w(h^{-1},g) \right\| \\
&\leq& \max\set{ ~\left| \sqrt{\mu^{(j)}(h)}-\sqrt{\mu^{(j)}(g^{-1}h)} \right|~\bigl|~ h\in B_n\cap gB_n}\cdot \|a\| \\
&\stackrel{\eqref{eq:E3}}{\leq}& \delta\|a\|.
\end{array}
\]
\phantomsection\label{theta}
Define the c.p.c.~map $\theta_j: A\rtimes_{r,\alpha,w} G \to M_{|B_n^{(j)}|}(A)$ by $\theta_j(x)=\sqrt{D_j}\Psi_j(x)\sqrt{D_j}$. By the previous calculation, we have
\begin{equation} \label{eq:E4}
 \|\theta_j(au_g)-D_jQ_jau_gQ_j\|\leq \delta\cdot\|a\|\quad\text{for all}\quad g\in M\;\text{and}\; a\in A.
\end{equation}

By \eqref{eq:E2}, we have for each $au_g\in F$ that the matrix coefficients of $\theta_j(au_g)\in M_{|B_n^{(j)}|}(A)$ are all in $\tilde{F}$. 
\phantomsection\label{phipsi} Choose an $r$-decomposable c.p.~approximation $(\CF,\psi,\phi)$ for $\tilde{F}$ up to $\delta/[G:G_n]$, i.e.~a finite dimensional \cstar-algebra $\CF$ and c.p.~maps $A\stackrel{\psi}{\to}\CF\stackrel{\phi^{(i)}}{\to} A$ for $i=0,\dots,r$ such that $\psi$ is c.p.c., the maps $\phi^{(i)}$ are c.p.c.~order zero and 
\begin{equation} \label{eq:E5}
 \|x-(\phi\circ\psi)(x)\|\leq\frac{\delta}{[G:G_n]}\quad\text{for all}~x\in\tilde{F},
\end{equation}
where $\phi$ denotes the sum $\phi= \phi^{(0)}+\dots+\phi^{(r)}$.

For all $j=0,\dots,s$ and $i=0,\dots,r$ let 
\begin{equation} \label{eq:E6}
 \psi_{j} = \id_{|B_n^{(j)}|}\otimes\psi: M_{|B_n^{(j)}|}\otimes A\to M_{|B_n^{(j)}|}\otimes\CF
\end{equation}
and 
\begin{equation} \label{eq:E7}
 \phi_{j,i} = \id_{|B_n^{(j)}|}\otimes\phi^{(i)}: M_{|B_n^{(j)}|}\otimes\CF\to M_{|B_n^{(j)}|}\otimes A
\end{equation}
denote the amplifications of $\psi$ and $\phi^{(i)}$. 

Let $D\subset A$ be a separable, $\alpha$-invariant \cstar-algebra containing $\tilde{F}$, the image of $\phi^{(i)}$ for each $i=0,\dots,r$ and such that $D$ is closed under multiplication with $\set{w(g,h) ~|~ g,h\in G}$. 
Since $\dimrok(\alpha, G_n)\leq d$, we can use~\ref{quantifier dimrok}(2) and find positive contractions $(f^{(l)}_{\bar{h}})_{\bar{h}\in G/G_n}^{l=0,\dots,d}$ in $A_\infty\cap D'$ satisfying the relations
\begin{align}
 &  \label{eq:R1} \dst \Big( \sum_{l=0}^d\sum_{\bar{h}\in G/G_n} f^{(l)}_{\bar{h}} \Big)a = a && \text{for\ all\ } a \in D ;\\
 &  \label{eq:R2} f^{(l)}_{\bar{g}}f^{(l)}_{\bar{h}} a =0 && \text{for\ all\ } a \in D ,\ \bar{g}\neq\bar{h} \text{\ and \ } l=0,\dots,d; \\
 &  \label{eq:R3} \alpha_{\infty,g}(f^{(l)}_{\bar{h}}) a = f^{(l)}_{ {}_{\quer{gh}} } a && \text{for\ all\ } a \in D ,\ l=0,\dots,d \text{\ and \ } g,h\in G .
\end{align}

\noindent
Now fix any $j\in\set{0,\dots,s}$ and $l\in\set{0,\dots,d}$. Define the map 
\begin{equation} \label{eq:E8}
 \sigma_{j,l}: M_{|B_n^{(j)}|}(D)\to (A\rtimes_{\alpha,w} G)_\infty \quad\text{by}\quad \sigma_{j,l}(e_{g,h}\otimes a) = u_{g^{-1}}^* f^{(l)}_1 a u_{h^{-1}}. 
\footnote{Keep in mind that in a twisted crossed product, the unitaries $u_g$ and $u_{g^{-1}}^*$ do not necessarily coincide. } 
\end{equation}

Note that it is c.p.~since $\sigma_{j,l}(x)=v_{j,l} x v_{j,l}^*$ for the $1\times |B_n^{(j)}|$-matrix $v_{j,l}= ( u_{h^{-1}}^* f_1^{(l) 1/2})_{h\in B_n^{(j)}}$.
We now show that $\sigma_{j,l}$ is order zero. We denote $f^{(l)} = \sum_{\bar{h}\in G/G_n} f^{(l)}_{\bar{h}}$, which, by virtue of \eqref{eq:R3}, is an element fixed by $\alpha_\infty$ upon multiplying with an element in $D$. Let $h_1,h_2,h_3,h_4\in B_n^{(j)}$ and $a,b\in A$. We have
\begin{longtable}{lcl}
\multicolumn{3}{l}{ $\sigma_{j,l}(e_{h_1,h_2}\otimes a)\sigma_{j,l}(e_{h_3,h_4}\otimes b)$ } \vspace{3mm}\\
\hspace{3mm}
& $=$ & $ u_{h_1^{-1}}^* f^{(l)}_1 a u_{h_2^{-1}} \cdot u_{h_3^{-1}}^* f^{(l)}_1 b u_{h_4^{-1}} $ \vspace{3mm}\\
& $=$ & $ u_{h_1^{-1}}^* u_{h_2^{-1}} (\alpha_{h_2^{-1}})_\infty^{-1}(af_{1}^{(l)}) \cdot(\alpha_{h_3^{-1}})_\infty^{-1}(f^{(l)}_1b) u_{h_3^{-1}}^* u_{h_4^{-1}} $ \vspace{3mm}\\
& $\stackrel{\eqref{eq:R3}}{=}$ & $ u_{h_1^{-1}}^* u_{h_2^{-1}} (\alpha_{h_2^{-1}})^{-1}(a) f^{(l)}_{\bar{h}_2} f^{(l)}_{\bar{h}_3} (\alpha_{h_3^{-1}})^{-1}(b) u_{h_3^{-1}}^* u_{h_4^{-1}} $ \vspace{3mm}\\
& $\stackrel{\eqref{eq:R2}}{=}$ & $ \delta_{\bar{h}_2,\bar{h}_3} \cdot
u_{h_1^{-1}}^* u_{h_2^{-1}} (\alpha_{h_2^{-1}})^{-1}(a)\cdot (f^{(l)}_{\bar{h}_2})^2 \cdot (\alpha_{h_2^{-1}})^{-1}(b) u_{h_2^{-1}}^* u_{h_4^{-1}} $ \vspace{3mm}\\
& $\stackrel{\eqref{eq:R3}}{=}$ & $ \delta_{\bar{h}_2,\bar{h}_3} \cdot
u_{h_1^{-1}}^* u_{h_2^{-1}} (\alpha_{h_2^{-1}})_\infty^{-1}(a (f^{(l)}_{1})^2) \cdot (\alpha_{h_2^{-1}})^{-1}(b) u_{h_2^{-1}}^* u_{h_4^{-1}} $ \vspace{3mm}\\
& $=$ & $ \delta_{\bar{h}_2,\bar{h}_3} \cdot
u_{h_1^{-1}}^* (f^{(l)}_1)^2 a u_{h_2^{-1}} \cdot u_{h_2^{-1}}^* b u_{h_4^{-1}} $ \vspace{3mm}\\
& $\stackrel{\eqref{eq:R3}}{=}$ & $ \delta_{\bar{h}_2,\bar{h}_3} \cdot
(f^{(l)}_{\bar{h}_1})^2 u_{h_1^{-1}}^* a b u_{h_4^{-1}} $ \vspace{3mm}\\
& $\stackrel{\mathrm{\eqref{eq:R2}}}{=}$ & $ \delta_{\bar{h}_2,\bar{h}_3} \cdot
f^{(l)} f^{(l)}_{\bar{h}_1} u_{h_1^{-1}}^* a b u_{h_4^{-1}} $ \vspace{3mm}\\
& $\stackrel{\eqref{eq:R3}}{=}$ & $ \delta_{\bar{h}_2,\bar{h}_3} \cdot
f^{(l)} u_{h_1^{-1}}^* f^{(l)}_1 a b u_{h_4^{-1}} $ \vspace{3mm}\\
& $\stackrel{\eqref{eq:R3}, \eqref{eq:D1}, \eqref{eq:E1}}{=}$ & $f^{(l)} \cdot\sigma_{j,l}\bigl( (e_{h_1,h_2}\otimes a)(e_{h_3,h_4}\otimes b) \bigl).$ 
\end{longtable}
\noindent
Note that for the last step of the above calculation, we have used \eqref{eq:D1} in that the canonical map $G\onto G/G_n$ is injective on each set $B_n^{(j)}$. Since this calculation involves linear generators of the \cstar-algebra $M_{|B_n^{(j)}|}(D)$, we obtain the equation
\[
\sigma_{j,l}(x)\sigma_{j,l}(y)=f^{(l)}\cdot\sigma_{j,l}(xy)\quad\text{for all}~x,y\in M_{|B_n^{(j)}|}(D).
\]
Hence each $\sigma_{j,l}$ is order zero and so is $\sigma_{j,l}\circ\phi_{j,i}$ for all $i=0,\dots, r,\; j=0,\dots,s$ and $l=0,\dots,d$. Note that this composition makes sense because of our assumption that $D$ contains the images of each $\phi^{(i)}$, and so the image of each $\phi_{j,i}$ is contained in the domain of $\sigma_{j,l}$.

As the next step, we would like to show that the maps in the diagram
\begin{equation} \label{eq:E9}
      \xymatrix{
	A\rtimes_{\alpha,w} G \ar[rr] \ar[rd]^{\bigoplus_j \psi_{j}\circ\theta_j =: \Theta} & & (A\rtimes_{\alpha,w} G)_\infty \\
	 & \bigoplus_{j=0}^s M_{|B_n^{(j)}|}(\CF) \ar[ru]_{\hspace{6mm} \sum_{j,l,i}  \sigma_{j,l}\circ\phi_{j,i} =: \Phi} &
	}
\end{equation}
give rise to a good c.p.~ approximation of $F$.
For this, we first calculate that for every contraction $x\in D$ and every $g \in M$ that
\begin{flalign} \label{eq:E10}
  \Big\|\sum_{j=0}^s\sum_{h\in B_n^{(j)}\setminus gB_n^{(j)}} \mu^{(j)}(h) f^{(l)}_{\bar{h}}x \Big\| &&
\end{flalign}
\begin{longtable}{cl} 
 \vspace{2mm}
 $\stackrel{\mathrm{\eqref{eq:E1}}}{\leq}$ & $\displaystyle (s+1)\cdot \max_{0\leq j\leq s}\; \Big\| \sum_{h\in B_n^{(j)}\setminus gB_n^{(j)}} \mu^{(j)}(h) f^{(l)}_{\bar{h}} x \Big\|$ \\
 $\stackrel{\mathrm{\eqref{eq:R2}}}{\leq}$ & $(s+1)\cdot\max\set{ \mu^{(j)}(h) ~\big|~ h\in B_n^{(j)}\setminus gB_n^{(j)}, j=0,\dots,s}$ \vspace{2mm}\\
 $\stackrel{\mathrm{\eqref{eq:E1}}}{=}$ & $(s+1)\cdot\max\set{ \| \mu^{(j)}-\mu^{(j)}(g^{-1} \cdot\_\!\_)\|_\infty ~\big|~ j=0,\dots,s }$ \vspace{2mm}\\
 $\stackrel{\mathrm{\eqref{eq:D3}}}{\leq}$ & $(s+1)\eps ~ \leq ~ (s+1)\delta$. 
\end{longtable}
\noindent
Observe that by the properties of the functions $\mu^{(j)}$, we have for all $l$ that 
\begin{align} \label{eq:E11}
f^{(l)} = \sum_{\bar{g}\in G/G_n} f_{\bar{g}}^{(l)} &\stackrel{\eqref{eq:D2}}{=} \sum_{\bar{g}\in G/G_n} \sum_{j=0}^s \sum_{h\in\bar{g}}  \mu^{(j)}(h) f_{\bar{g}}^{(l)} 
&\stackrel{\eqref{eq:E1}}{=} \sum_{j=0}^s \sum_{h\in B_n^{(j)}} \mu^{(j)}(h) f^{(l)}_{\bar{h}}. 
\end{align}
Note that the last equation follows from the fact that for each $j=0,\dots,s$, the set $B_n^{(j)}$ is contained in a finite set of representatives of $G/G_n$ by \eqref{eq:D1} and \eqref{eq:E1}.
It follows for all $g\in M$ and $x\in D$ that
\begin{equation} \label{eq:E12}
 \Big\| x - \Big( \sum_{l=0}^d\sum_{j=0}^s \sum_{h\in B_n^{(j)}\cap gB_n^{(j)}} \mu^{(j)}(h) f^{(l)}_{\bar{h}} \Big) \cdot x \Big\|
\end{equation}
\[
  \begin{array}{cl} 
    \stackrel{\eqref{eq:E11},\eqref{eq:R1}}{=}& \displaystyle \Big\| \Big( \sum_{l=0}^d \sum_{j=0}^s \sum_{h\in B_n^{(j)}\setminus gB_n^{(j)}} \mu^{(j)}(h) f^{(l)}_{\bar{h}} \Big)\cdot x \Big\| \\
    \stackrel{\eqref{eq:E10}}{\leq} & {(s+1)(d+1)\delta} \; .
  \end{array}
\]
Now let $au_g\in F$ and recall the definition of the maps $\Theta$ and $\Phi$ from the approximation diagram \eqref{eq:E9}. To simplify notation in the following calculation, the sums over the indices $l$, $j$, and $i$ always stand for $\displaystyle \sum_{l=0}^d$, $\displaystyle \sum_{j=0}^s$ and $\displaystyle \sum_{i=0}^r$, while the symbol $\displaystyle \sum_{h}$ stands for $\displaystyle \sum_{h\in B_n^{(j)}\cap gB_n^{(j)}}$. With such conventions, we have
\begin{longtable}{ll}
\multicolumn{2}{l}{$\Phi\circ\Theta(au_g)$} \\\\

$\stackrel{\eqref{eq:E9}}{=}$ & \hspace{-30mm}$\displaystyle \sum_{l,j,i} (\sigma_{j,l}\circ\phi_{j,i}\circ\psi_{j}\circ\theta_j)(au_g)$ \\\\

$\stackrel{\eqref{eq:E4}}{=}_{(s+1)(d+1)(r+1)\delta}$ & \hspace{-10mm}$\displaystyle \sum_{l,j,i} (\sigma_{j,l}\circ\phi_{j,i}\circ\psi_{j})(D_jQ_jau_gQ_j)$ \\\\

$\stackrel{\eqref{eq:E2}}{=}$ & \hspace{-30mm}$\displaystyle \sum_{l,j,i} (\sigma_{j,l}\circ\phi_{j,i}\circ\psi_{j}) \left( \sum_{h} \mu^{(j)}(h)\cdot e_{h,g^{-1}h}\otimes\alpha_{h^{-1}}(a)w(h^{-1},g) \right)$ \\\\

$\hspace{-2.8mm}\stackrel{\eqref{eq:E6},\eqref{eq:E7}}{=}$ & \hspace{-30mm}$\displaystyle \sum_{l,j,i} \sigma_{j,l} \left( \sum_{h} \mu^{(j)}(h)\cdot e_{h,g^{-1}h}\otimes (\phi^{(i)}\circ\psi)\Big[ \alpha_{h^{-1}}(a)w(h^{-1},g) \Big] \right)$ \\\\

$\stackrel{\eqref{eq:E8}}{=}$ & \hspace{-30mm}$\displaystyle \sum_{l,j,i} \sum_{h} \mu^{(j)}(h)\cdot u_{h^{-1}}^* f^{(l)}_1 \Big[ (\phi^{(i)}\circ\psi)\bigl( \alpha_{h^{-1}}(a)w(h^{-1},g) \bigl) \Big] u_{h^{-1}g}$ \\\\

$\stackrel{\eqref{eq:R3}}{=}$ & \hspace{-30mm}$\displaystyle \sum_{l,j} \sum_{h} \mu^{(j)}(h)f_{\bar{h}}^{(l)}\cdot u_{h^{-1}}^* \Big[ (\phi\circ\psi)\bigl(\alpha_{h^{-1}}(a)w(h^{-1},g) \bigl) \Big] u_{h^{-1}g}$ \\\\

$\stackrel{\eqref{eq:E5}}{=}_{\big( (s+1)(d+1)|B_n^{(j)}|\cdot \frac{\delta}{[G:G_n]} \big)}$ & \hspace{-1mm}$\displaystyle \sum_{l,j} \sum_{h} \mu^{(j)}(h)f_{\bar{h}}^{(l)}\cdot \underbrace{u_{h^{-1}}^* \alpha_{h^{-1}}(a)w(h^{-1},g) u_{h^{-1}g}}_{=au_g}$ \\\\

$=$ & \hspace{-30mm}$\displaystyle \left(\sum_{l,j} \sum_{h} \mu^{(j)}(h)f_{\bar{h}}^{(l)}\right)\cdot a u_g$ \\\\

$\hspace{-1.8mm}\stackrel{\mathrm{\eqref{eq:E12}}}{=}_{(s+1)(d+1)\delta}$ & \hspace{-15mm}$au_g$.
\end{longtable}
\noindent
Summing up these approximation steps and using \eqref{eq:D1} in the form of the inequality $|B_n^{(j)}|\leq [G:G_n]$, it follows for all $au_g\in F$ that
\[
au_g =_{3(s+1)(d+1)(r+1)\delta} \Phi\circ\Theta(au_g).
\]
Now let us summarize what we have got. We have constructed a c.p.~approximation 
\[ 
\left( \bigoplus_{j=0}^s M_{|B_n^{(j)}|}(\CF),~ \Theta,~ \Phi \right)
\]
of tolerance $3(s+1)(d+1)(r+1)\delta$ on $F$, where the map
\[
\Phi = \sum_{i=0}^r\sum_{l=0}^d \sum_{j=0}^s \sigma_{j,l}\circ\phi_{j,i} : \bigoplus_{j=0}^s M_{|B_n^{(j)}|}(\CF) \to (A\rtimes_{\alpha,w} G)_\infty
\]
is a sum of $(s+1)(d+1)(r+1)$ c.p.c.~maps of order zero.
Since $d,s,r$ are constants and $F\fin A\rtimes_{\alpha,w} G$ and $\delta>0$ were arbitrary, it follows from~\ref{dimnuc sequence} that
\[
\dimnuceins(A\rtimes_{\alpha,w} G)\leq (s+1)(d+1)(r+1),
\]
which is what we wanted to show.
\end{proof}

The following is a nice test case for the above theorem:

\rem Let $G$ be a countably infinite, residually finite and amenable group. Let $\sigma=(G_n)_n$ be a regular approximation of $G$ consisting of normal subgroups. Then the generalized Bunce-Deddens algebra associated to $G$ and $\sigma$ (see~\cite{Orfanos10}) has nuclear dimension at most $\asdim(\square_\sigma G)$.

This is because the associated generalized Bunce-Deddens algebra arises as the crossed product of the dynamical system
\[
(\CC(X),\alpha) = \lim_{\longrightarrow}~ (\CC(G/G_n), G\text{-shift}),
\]
where the connecting maps are induced by the natural maps $G/G_n\onto G/G_{n+1}$ and $X$ turns out to be homeomorphic to the Cantor set. From the definition, it is trivial that the resulting action $\alpha$ has Rokhlin dimension $0$ along $\sigma$. Hence it follows from~\ref{Hauptsatz} that
\[
\dimnuc(\CC(X)\rtimes_\alpha G) \leq \asdimeins(\square_\sigma G)\cdot 1 \cdot 1-1 = \asdim(\square_\sigma G).
\]
In particular, this is independent of the other results within~\cite{Orfanos10}. On the other hand, it is easy to see that generalized Bunce-Deddens algebras are separable, unital, nuclear, quasidiagonal, have a unique tracial state and satisfy the UCT.
A combination of these facts with~\cite[6.2]{MatuiSato14UHF} gives an alternative proof that Bunce-Deddens algebras, associated to regular approximations yielding a box space of finite asymptotic dimension, are classifiable. 

However, we should mention that it is known by now that \emph{all} generalized Bunce-Deddens algebras are classifiable. It follows from the detailed study of the transformation groupoid, which was carried out in~\cite{Orfanos10}, that all generalized Bunce-Deddens algebras have strict comparison of positive elements. This, in turn, is sufficient to deduce from~\cite[6.2]{MatuiSato14UHF} that they are classifiable.


\section{Dependence on the approximation sequence}\label{sec:dimrok-amdim}

In this section, we discuss the dependence of the Rokhlin dimension on the approximation sequence. In order to do so, we introduce a related invariant $\amdim(\alpha)$ which does not depend on any residual finite approximation. Later in Section~\ref{sec:dimBLR}, we shall see that for topological actions, this invariant corresponds to the notion of amenability dimension that appeared in the work of Guentner, Willett and Yu~\cite{GuentnerWillettYu15}. 

\begin{defi} \label{amdim}
Let $A$ be a \cstar-algebra, $G$ a countable discrete group and $(\alpha,w): G\curvearrowright A$ a cocycle action. We define $\amdim(\alpha)$ to be the smallest natural number $d$ with the following property: 

For any $\eps>0$ and any finite sets $M\fin G$ and $F\fin A$, there exist finitely supported maps $\mu^{(l)}: G\to A_{1,+},~ g \mapsto \mu_g^{(l)}$ for $l=0,\dots, d$ satisfying:

{\renewcommand{\theenumi}{\alph{enumi}}
\begin{enumerate}
 \item $\displaystyle \Big( \sum_{l=0}^d \sum_{g\in G} \mu^{(l)}_g \Big)\cdot a =_\eps a$ for any $a\in F$;
 \item $\mu_g^{(l)} \mu_{h}^{(l)} = 0$ for all $l=0,\dots,d$ and ${g}\neq {h}$ in $G$;
 \item $\dst \Big\| \sum_{h \in E} \big( \alpha_g(\mu_{h}^{(l)}) -  \mu_{gh}^{(l)} \big) \cdot a \Big\| \leq \eps$ for all $a\in F,~l=0,\dots,d$, $g\in M$ and every finite subset $E \fin G$;
 \item $\dst \Big\| \Big[ \sum_{g \in E} \mu_g^{(l)},a \Big] \Big\| \leq\eps$ for all $a\in F,~l=0,\dots,d$ and every finite subset $E \fin G$.
\end{enumerate}
}
If no such $d$ exists, then we set $\amdim(\alpha) = \infty$.
\end{defi}

Note that the sum in the first condition is a finite sum because each function $\mu^{(l)}$ is finitely supported. Meanwhile, the third condition expresses a kind of almost equivariance for each $\mu^{(l)}$, and in particular has the consequence that the function $\mu^{(l)}$ tapers off near the edge of its support. As in~\ref{dimrok allgemein}, we drop the cocycle $w$ in the notation because the definition of $\amdim$ only depends on $\alpha$.

Note also that the definition of $\amdim(\alpha)$ does not require residual finiteness of the group. If the group is residually finite, this notion allows us to compare the Rokhlin dimension associated to various regular approximation sequences, via the following intertwining inequalities. 

\begin{theorem} \label{dimrok amdim}
Let $A$ be a \cstar-algebra, $G$ a residually finite group and $(\alpha,w): G\curvearrowright A$ a cocycle action. Let $\sigma=(G_n)_n\in\Lambda(G)$ be a regular approximation of $G$. Then one has the following estimates:
\[
\dimrokeins({\alpha}) \leq \amdimeins(\alpha) \leq \asdimeins(\square_\sigma G) \cdot \dimrokeins({\alpha}, \sigma).
\]
In particular, if the right-hand side is finite for some regular approximation $\sigma$, then it follows that $\dimrok(\alpha) < \infty$.
\end{theorem}
\begin{proof}
Denote $d=\amdim(\alpha), s=\asdim(\square_\sigma G)$ and $r=\dimrok(\alpha, \sigma)$. Let us prove the first inequality, assuming that $d<\infty$. Let $M\fin G$, $F\fin A$ and $\eps>0$ be given. By Definition~\ref{amdim}, we can find finitely supported maps $\mu^{(l)}: G\to A_{1,+}$ for $l=0,\dots,d$ with the given properties depending on $(M,F,\eps)$. Let $H\subset G$ be a subgroup of finite index. For $l=0,\dots,d$ and $\bar{g}\in G/H$, we define $f_{\bar{g}}^{(l)} = \sum_{h\in\bar{g}} \mu_h^{(l)}$, where the sum is in fact finite. Then one easily verifies that~\ref{quantifier dimrok}(3) is satisfied for the parameters $(M,F,\eps)$: for any $a\in F$, we have
\begin{align}
 \Big( \sum_{l=0}^d \sum_{\bar{g}\in G/H} f_{\bar{g}}^{(l)} \Big) \cdot a =_\eps a  \; ; & \hspace{2in} 
\end{align}
and for all $l=0,\dots,d$, we have
\begin{align}
 & f_{\bar{g}}^{(l)}f_{\bar{h}}^{(l)} = 0 && \text{for~all~}  \bar{g}\neq\bar{h} ~\mathrm{in}~ G/H; \\
 & \big\| \big( \alpha_g(f_{\bar{h}}^{(l)}) - f_{ {}_{\quer{gh}} }^{(l)} \big) \cdot a \big\| \leq \eps  && \text{for~all~} \bar{h}\in G/H ~\text{and~} g\in M;\\
 & \| [ f_{\bar{g}}^{(l)},a ] \| \leq\eps  && \text{for~all~}  \bar{g}\in G/H.
\end{align}
As $H, M$ and $\eps$ were arbitary, we can deduce $\dimrok(\alpha)\leq d$ by~\ref{quantifier dimrok}.

Let us prove the second inequality, assuming that $s,r<\infty$. Let $M\fin G$ and $\eps>0$ be given. By~\ref{decay functions}(3), there is some $n$ such that there exist finitely supported functions $\nu^{(j)}: G\to [0,1]$, for $j=0,\dots,d$, that satisfy the three conditions in~\ref{decay functions}(3) in place of $\mu^{(j)}$, with regard to the pair $(M,\eps)$. Moreover, given a finite subset $F \subset A_1$ and setting $N = [G : G_n]$, we apply~\ref{quantifier dimrok}(3) and choose positive contractions $f^{(l)}_{\bar{g}}$ in $A_{1,+}$ for $l=0,\dots,r$ and $\bar{g}\in G/G_n$ such that for every $a\in F$, we have
\begin{align}
 & \label{eq:almost-amdim-1} \sum_{l=0}^d\sum_{\bar{g}\in G/G_n} f_{\bar{g}}^{(l)} \cdot a =_\eps a \; ; & \hspace{2in} 
\end{align}
and for all $l=0,\dots,d$, we have
\begin{align}
 & \label{eq:almost-amdim-2} \|f_{\bar{g}}^{(l)}f_{\bar{h}}^{(l)}\cdot a \|\leq \frac{\eps}{N}  && \text{for~all~} \bar{g}\neq\bar{h} \text{~in~} G/G_n ; \\
 & \label{eq:almost-amdim-3} \bar{\alpha}_g(f_{\bar{h}}^{(l)}) \cdot a =_{\frac{\eps}{N}} f_{ {}_{\quer{gh}} }^{(l)} \cdot a  && \text{for~all~} \bar{h}\in G/G_n ; \\
 & \label{eq:almost-amdim-4} \| [ f_{\bar{g}}^{(l)},a ] \|\leq \frac{\eps}{N}  && \text{for~all~}   \bar{g}\in G/H .
\end{align}
Now the cone over $\IC^N$ is projective by~\cite[10.1.7]{Loring}. Since in the relations \eqref{eq:almost-amdim-1} - \eqref{eq:almost-amdim-4}, $\eps$ may be arbitrarily small, we can use a projectivity argument and strengthen \eqref{eq:almost-amdim-2} to
\begin{equation}\tag{\ref{eq:almost-amdim-2}'}
 f_{\bar{g}}^{(l)}f_{\bar{h}}^{(l)}=0\quad\text{for all}~l=0,\dots,d~\text{and}~\bar{g}\neq\bar{h}~\text{in}~G/G_n.\footnote{Fix the index $l$. For a shrinking sequence of $\eps$, the elements $f^{(l)}_{\bar{g}}$ give rise to a $*$-homomorphism from $\CC_0(0,1]^{\oplus N}$ to the sequence algebra $A_\infty$. Due to projectivity, this lifts to a sequence of $*$-homomorphisms to $A$, meaning that we may find honestly orthogonal elements close to the original ones.}
\end{equation}
Now consider the functions $\mu^{(l,j)}$ for $l=0,\dots,r$ and $j=0,\dots,s$ such that
\[
\mu^{(l,j)}_g = \nu^{(j)}_g\cdot f_{\bar{g}}^{(l)} \quad\text{for\ all}~g \in G.
\]
Notice that the support of each $\mu^{(l,j)}$ is contained in the support of $\nu^{(j)}$, which consists of no more than $N$ elements. Thus one directly computes that the conditions of~\ref{amdim} are satisfied for $M$, $F$ and $3 \eps$. For example, to check (3), we compute
\begin{align*}
 & \ \Big\| \sum_{h \in E} \big( \alpha_g(\mu_{h}^{(l,j)}) -  \mu_{gh}^{(l,j)} \big) \cdot a \Big\| \\
 = & \ \Big\| \sum_{h \in E} \Big( \alpha_g( \nu^{(j)}_h\cdot f_{\bar{h}}^{(l)}) - \alpha_g ( \nu^{(j)}_h) \cdot f_{\bar{gh}}^{(l)} + \alpha_g( \nu^{(j)}_h) \cdot f_{\bar{gh}}^{(l)} - \nu^{(j)}_{gh}\cdot f_{\bar{gh}}^{(l)} \Big) \cdot a \Big\| \\
 \leq & \ \sum_{h \in E \cap \supp(\nu^{(j)}) } \| \alpha_g ( \nu^{(j)}_h ) \| \big\| \big( \alpha_g(  f_{\bar{h}}^{(l)} ) - f_{\bar{gh}}^{(l)}  \big) \cdot a \big\| \\ 
 & + \sum_{h \in E \cap \big( \supp(\nu^{(j)})\cup g \cdot \supp(\nu^{(j)}) \big) } \|  \alpha_g ( \nu^{(j)}_h )  - \nu^{(j)}_{gh} \| \| f_{\bar{gh}}^{(l)} \cdot a \| \\
 \leq & \  N \cdot \frac{\eps}{N} + 2N \cdot \frac{\eps}{N} = 3 \eps \; .
\end{align*}
The other conditions are checked similarly, and in fact these condition hold with respect to the tolerance $\eps$ instead of $3 \eps$. Since $M$, $F$ and $\eps$ were arbitrary, we get $d+1 \leq (s+1)(r+1)$.
\end{proof}

This shows that the mere finiteness of the Rokhlin dimension is, in a sense, independent of the regular approximation. In particular, for a system $(A,\alpha,w)$, whenever Theorem~\ref{Hauptsatz} applies nontrivially (i.e., $G$ admits a regular approximation $\sigma$ such that the right-hand side of the first inequality, $\asdimeins(\square_\sigma G)\cdot \dimrokeins(\alpha, \sigma)\cdot\dimnuceins(A)$, is finite), all versions of Rokhlin dimension are automatically finite.


\section{Free topological actions of nilpotent groups}\label{sec:nilpotent-dimnuc}

\noindent
In this section, we study free actions of finitely generated, nilpotent groups on finite dimensional spaces, and show that their Rokhlin dimension is finite. First, we have to establish an important technical property of nilpotent groups: 

\begin{lemma} \label{nilpotent growth}
Let $G$ be a finitely generated, nilpotent group. Let $\ell$ be the Hirsch-length of $G$ and set $m=3^\ell$. Then for any finite subset $M\fin G$, there is a finite subset $F\fin G$ containing the unit and $g_1,\dots, g_m \in G$, such that for any $g\in F^{-1}F$, we have $Mg\subset g_jF $ for some $j\in\set{1,\dots,m}$.
\end{lemma}
\begin{proof}
Before we delve into the somewhat technical proof, let us summarize the approach. It is easy to see that this is true for $G=\IZ$. If we pretend for a moment that $G$ is torsion-free and nilpotent, then the idea boils down to representing $G$ as an iterated central extension by copies of $\IZ$, and inductively define $F$ in such a way that, in the direction of a newly-added generator, each \emph{new} $F$ is long enough that commutators of the old generators have relatively minute impact. This process describes an induction argument by the Hirsch-length of $G$.

So let us apply induction by $\ell$. If $\ell=0$, then $G$ is finite, and we may simply take $F=G$ and $m=1$.

Now suppose $G$ has Hirsch-length $\ell+1$ for some $\ell\geq 0$, and the statement has been proved for all nilpotent groups with Hirsch length at most $\ell$. By~\ref{group properties}, $G$ has a central element $t\in G$ of infinite order. Set $H = G / \langle t \rangle$. This yields a finitely generated, nilpotent group of Hirsch-length $\ell$. Denote by $\pi: G\to H$ the quotient map and set $m=3^\ell$.

Now let $M\fin G$ be a finite subset. We apply our induction hypothesis on $H$ to get a finite set $F_0\fin H$ containing the identity and $h_1,\dots, h_m\in H$ such that for each $h\in F_0^{-1}F_0$, we have $\pi(M)h\subset h_jF_0$ for some $j\in\set{1,\dots,m}$. 

Pick some cross section $\sigma: H\to G$ such that $\sigma(1_H)=1_G$. 
Decompose $M$ into 
\begin{equation}  \label{eq:F6}
M = \bigsqcup_{k\in\pi(M)} M_k\cdot\sigma(k)
\end{equation} 
for some finite subsets $M_k \fin \langle t\rangle$. Define the finite sets 
\begin{equation}  \label{eq:F7}
T = \bigcup_{k \in \pi(M)} \bigcup_{k_1,k_2\in F_0} \sigma(kk_1^{-1}k_2)^{-1}\sigma(k) \sigma(k_1)^{-1} \sigma(k_2) M_k
\end{equation}
and
\begin{equation}  \label{eq:F8}
S= \bigcup_{j=1}^m \bigcup_{k\in\pi(M)} \bigcup_{k_1,k_2\in F_0} \sigma(h_{j}^{-1}kk_1^{-1}k_2)^{-1}\sigma(h_{j})^{-1}\sigma(kk_1^{-1}k_2).
\end{equation}
Since $\sigma$ is a cross section, we have $S, T\subset\langle t\rangle$. Moreover, $ST$ is finite, so there exists $n\in\IN$ such that 

\begin{equation}  \label{eq:F9}
ST\subset\set{t^i ~|~ -n\leq i\leq n}.
\end{equation} 
We set 
\begin{equation}  \label{eq:F10}
F_1 = \set{t^i ~|~ -3n\leq i\leq 3n},\quad F=\sigma(F_0)\cdot F_1
\end{equation} 
and 
\begin{equation}  \label{eq:F11}
g_{i,j}=t^{4ni}\sigma(h_j)
\end{equation} 
for $i=-1,0,1$ and $j=1,\dots,m$. We will show that the set $F$ and the elements $g_{i,j}$ satisfy the desired property of the assertion.

Let us choose an element $x\in F^{-1}F$. By \eqref{eq:F10}, this element has the form 
\begin{equation}  \label{eq:F12}
x=t^{l_1-l_2}\sigma(k_1)^{-1}\sigma(k_2)
\end{equation} 
for certain $k_1,k_2\in F_0$ and $-3n\leq l_1,l_2\leq 3n$. By assumption, we have 
\begin{equation}  \label{eq:F13}
\pi(M)k_1^{-1}k_2\subset h_{j_0} F_0
\end{equation} 
for some $j_0\in\set{1,\dots,m}$.
For the number 
\[
i_0=\begin{cases} -1 &,\quad\text{if}~~ -6n \leq l_2-l_1 < -2n \\
0 &,\quad\text{if}~~ -2n\leq l_2-l_1\leq 2n \\
1 &,\quad\text{if}~~ 2n< l_2-l_1\leq 6n
\end{cases}
\]
we have
\begin{equation}  \label{eq:F14}
l_2-l_1+\set{-n,\dots,n} \subset 4ni_0 + \set{-3n,\dots,3n}.
\end{equation}
Combining all this, we observe that
\begin{longtable}{cl}
\multicolumn{2}{l}{\hspace{7mm} $Mx$} \\\\ 
$\stackrel{\eqref{eq:F6}, \eqref{eq:F12}}{=}$ & \hspace{-5mm}$\dst \bigsqcup_{k\in\pi(M)} M_k\cdot\sigma(k) t^{l_1-l_2}\sigma(k_1)^{-1}\sigma(k_2)$ \\\\
$=$ & \hspace{-5mm}$\dst \bigsqcup_{k\in\pi(M)} \sigma(k)\sigma(k_1)^{-1}\sigma(k_2) M_k t^{l_1-l_2}$ \\\\
$=$ & \hspace{-5mm}$\dst \bigsqcup_{k\in\pi(M)} \sigma(kk_1^{-1}k_2)\sigma(kk_1^{-1}k_2)^{-1}\sigma(k)\sigma(k_1)^{-1}\sigma(k_2) M_k t^{l_1-l_2}$ \\\\
$\stackrel{\eqref{eq:F7}}{\subset}$ & \hspace{-5mm}$\dst \bigsqcup_{k\in\pi(M)} \sigma(kk_1^{-1}k_2) T t^{l_1-l_2}$ \\\\
$=$ & \hspace{-5mm}$\dst \bigsqcup_{k\in\pi(M)} \sigma(h_{j_0})\sigma(h_{j_0}^{-1}kk_1^{-1}k_2)\sigma(h_{j_0}^{-1}kk_1^{-1}k_2)^{-1}\sigma(h_{j_0})^{-1}\sigma(kk_1^{-1}k_2) T t^{l_1-l_2}$ \\\\
$\stackrel{\eqref{eq:F8}}{\subset}$ & \hspace{-5mm}$\dst \bigsqcup_{k\in\pi(M)} \sigma(h_{j_0})\sigma(h_{j_0}^{-1}kk_1^{-1}k_2) ST t^{l_1-l_2}$ \\\\
$\stackrel{\eqref{eq:F13}}{\subset}$ & $\sigma(h_{j_0}) \sigma(F_0) ST t^{l_1-l_2}$ \\\\
$\stackrel{\eqref{eq:F9}, \eqref{eq:F14}, \eqref{eq:F10}}{\subset}$ & $\sigma(h_{j_0})\sigma(F_0)F_1t^{4ni_0}$ \\\\
$\stackrel{\eqref{eq:F10}, \eqref{eq:F11}}{=}$ & $g_{i_0,j_0}\cdot F$.
\end{longtable}
\noindent
By the definition of the elements $g_{i,j}$ in \eqref{eq:F11}, we see that there are $3m=3^{\ell+1}$ lower indices.
This finishes the induction step and the proof.
\end{proof}

Recall the following marker property Lemma:

\begin{lemma}[see {\cite[3.8, 4.4]{Szabo14}}] \label{marker property}
Let $G$ be a countable, discrete group and $X$ a compact metric space of finite covering dimension $d\in\IN$. Let $\alpha: G\curvearrowright X$ be a free action. Let $F\fin G$ be a finite set, and let $g_1,\dots,g_d\in G$ be elements such that the sets
\[
F^{-1}F,~ g_1F^{-1}F,~ \dots,~ g_dF^{-1}F ~\fin~ G
\]
are pairwise disjoint. Set $g_0=1_G$. Then there exists an open set $Z\subset X$ such that
\[
X = \bigcup_{l=0}^d \bigcup_{g\in F^{-1}F} \alpha_{g_lg}(Z)
\]
and
\[
\alpha_g(\quer{Z})\cap\alpha_h(\quer{Z})=\emptyset\quad\text{for all}~g\neq h~\text{in}~F.
\]
\qed
\end{lemma}

We need one more Lemma characterizing the value $\amdim(-)$ in the case of a topological action.

\begin{lemma}\label{lem:amdim-abelian}
 Let $G$ be a countable, discrete group, $X$ a locally compact Hausdorff space, and $d\in\IN$ a natural number. Let $\alpha: G\curvearrowright X$ be a free action and $\bar{\alpha} : G \curvearrowright \CC_0(X)$ the induced action on the $C^*$-algebra. Then $\amdim(\bar{\alpha}) \leq d$ if and only if for every $\eps>0$, every finite subset $M\fin G$ and every compact subset $K \subset X$, there exist finitely supported maps $\mu^{(l)}: G\to \CC_c(X)_{1,+}$ for $l=0,\dots, d$ satisfying:
{
\renewcommand{\theenumi}{\alph{enumi}}
\begin{enumerate}
 \item $\dst \sum_{l=0}^d \sum_{g\in G} \mu^{(l)}_g\leq\eins$ ~and~ $\dst \sum_{l=0}^d \sum_{g\in G} \mu^{(l)}_g|_K = 1$;
 \item $\mu_g^{(l)} \mu_{h}^{(l)} = 0 $ for all $l = 0,\cdots,d$ and $g \neq h$ in $G$;
 \item $\mu_h^{(l)}\circ\alpha_{g^{-1}} =_\eps \mu_{gh}^{(l)}$ for all $l = 0,\cdots,d$, $g\in M$ and $h\in G$.
\end{enumerate}
}
When $X$ is compact, it suffices to check the conditions for $K= X$.
\end{lemma}

\begin{proof}
Suppose $\amdim(\bar{\alpha}) \leq d$. Then given any $\eps > 0$, any finite subset $M\fin G$ and any compact subset $K \subset X$, there exists a quasicentral approximate unit for $\CC_0(X) \rtimes G$ inside $\CC_c(X)$, so there is $a \in \CC_c(X)_{1,+}$ such that $a|_K = 1$ and $a \circ\alpha_{g^{-1}} =_\eps a$ for all $g\in M$. Let us assume without loss of generality that $\eps\leq 1/4$. Setting $F = \{a\}$, we obtain finitely supported maps $\nu^{(l)}: G\to \CC_c(X)_{1,+},~ g \mapsto \nu_g^{(l)}$ for $l=0,\dots, d$ satisfying the conditions in~\ref{amdim} for $M$, $\eps$ and $F$. In particular, we have $\sum_{l=0}^d \sum_{g\in G} \mu^{(l,j)}_g|_K =_\eps 1$. Now we define, for $l = 0,\dots, d$, the functions
 \[
  \mu^{(l)} = \frac{\nu^{(l)} \cdot a }{ \max\left\{ \frac{1}{2} , \sum_{j=0}^d \nu^{(j)} \cdot a \right\} } \; .
 \]
 A direct computation shows that the maps $\mu^{(l)}$ satisfy (a) and (b) in the statement of the lemma. As for (c), we see that 
 \[
  \big( \nu^{(l)} \cdot a \big) \circ\alpha_{g^{-1}} =_\eps  \nu^{(l)}  \circ\alpha_{g^{-1}} \cdot a =_\eps \nu^{(l)} \cdot a
 \]
 and similarly 
 \[
  \max\left\{ \frac{1}{2} , \sum_{j=0}^d \nu^{(j)} \cdot a \right\} \circ\alpha_{g^{-1}} =_{(d+2)\eps} \max\left\{ \frac{1}{2} , \sum_{j=0}^d \nu^{(j)} \cdot a \right\} \; ,
 \]
 which implies that 
\[
\begin{array}{cll} 
\mu_h^{(l,j)} \circ\alpha_{g^{-1}} &=& \dst\frac{ \left(\nu^{(l)} \cdot a \right) \circ\alpha_{g^{-1}} }{ \max\left\{ \frac{1}{2} , \sum_{j=0}^d \nu^{(j)} \cdot a \right\} \circ\alpha_{g^{-1}} } \\\\ 
  &=_{4\eps} &  \dst\frac{\nu^{(l)} \cdot a }{ \max\left\{ \frac{1}{2} , \sum_{j=0}^d \nu^{(j)} \cdot a \right\} \circ\alpha_{g^{-1}} } \\\\
  &=_{4(d+2)\eps} & \dst\frac{\nu^{(l)} \cdot a }{ \max\left\{ \frac{1}{2} , \sum_{j=0}^d \nu^{(j)} \cdot a \right\} } \\\\
  & =& \mu_h^{(l,j)}.
\end{array}
\]
Note that in the second estimate, we used the fact that the denominator is at least $\frac{1}{2}$. Since $\eps$ was arbitary, we see that condition (c) is also satisfied. 
 
 To prove the converse, let again $\eps > 0$ and finite subsets $M\fin G$ and $F \fin \CC_0(X)$ be given. Without loss of generality, we may assume $F \subset \CC_0(X)_1$. Choose a compact set $K \subset X$ such that for any $a \in F$ and $x \in X \setminus K$, we have $|a(x)| < \eps$. By our assumption, we may find finitely supported maps $\mu^{(l)}: G\to \CC_c(X)_{1,+}$ for $l=0,\dots, d$ satisfying the conditions (a), (b), and (c) in the statement of the lemma. It is clear that the functions $\mu^{(l)}$, for $l=0,\dots, d$, satisfy conditions (a), (b) and (d) in~\ref{amdim} (the last condition is automatic). For condition (c), we see that for any $l \in \{0,\dots,d \}$, any finite subset $E \fin G$ and any $x \in X$, the set 
 \[
  \left\{ h \in E ~|~ \mu_{gh}^{(l)}(x) \neq 0 \right\}
 \]
 consists of at most one element because of orthogonality. Thus
 \[
  \Big| \sum_{h \in E} \big( \bar{\alpha}_g (\mu_{h}^{(l)}) -  \mu_{gh}^{(l)} \big) (x) \Big| \leq  \sum_{h \in E_x \cup E_{\alpha_{g^{-1}}(x)}} \big| \big( \bar{\alpha}_g (\mu_{h}^{(l)}) -  \mu_{gh}^{(l)} \big) (x) \big| \leq 2 \eps \; .
 \]
 Since $\eps$ was arbitary, we conclude that $\amdim(\bar{\alpha}) \leq d$. 
\end{proof}

Combining these Lemmas, we can obtain the following result:

\begin{theorem} \label{free nilpotent amdim}
Let $G$ be a finitely generated, infinite, nilpotent group and $X$ a compact metric space of finite covering dimension. Let $\alpha: G\curvearrowright X$ be a free action. Then the induced \cstar-algebraic action $\bar{\alpha} : G \curvearrowright \CC(X)$ satisfies
\[
 \amdimeins(\bar{\alpha}) \leq 3^{\ell_{\mathrm{Hir}}(G)}\cdot\dimeins(X).
\]
\end{theorem}

\begin{proof}
Let $d$ be the covering dimension of $X$ and $\ell$ the Hirsch-length of $G$, and denote $m=3^\ell$. Suppose we are given $\eps>0$ and finite sets $M\fin G$ and $F\fin \CC(X)$. Without loss of generality, we may assume that $F$ consists of contractions. We are going to show that there exist finitely supported maps $\mu^{(l,j)}: G\to \CC(X)_{1,+}$ for $l=0,\dots, d$ and $j=1,\dots,m$ satisfying:
{\renewcommand{\theenumi}{\alph{enumi}}
\begin{enumerate}
 \item $\dst \sum_{l=0}^d \sum_{j=1}^{m} \sum_{g\in G} \mu^{(l,j)}_g=\eins$;
 \item $\mu_g^{(l,j)} \mu_{h}^{(l,j)} = 0 $ for all $l = 0,\dots,d,~j=1,\dots,m$ and $g \neq h$ in $G$;
 \item $\mu_h^{(l,j)}\circ\alpha_{g^{-1}} =_\eps \mu_{gh}^{(l,j)}$ for all $l = 0,\dots,d,~j=1,\dots,m$, $g\in M$ and $h\in G$.
\end{enumerate}
}
By~\ref{lem:amdim-abelian}, these conditions are sufficient to imply $\amdimeins(\bar{\alpha}) \leq m (d+1)$.
Since $G$ is amenable, we can choose a F{\o}lner set
\begin{equation}  \label{eq:F15}
J\fin G\quad\text{with}\quad |J\Delta gJ| ~\leq~ \eps|J|\quad\text{for all}~g\in M.
\end{equation}
By~\ref{nilpotent growth}, we can find a finite subset $F\fin G$ containing the unit and $h_1,\dots,h_m\in G$ such that for every $x\in F^{-1}F$, we have $Jx\subset h_jF$ for some $j\in\set{1,\dots,m}$. Since $G$ is infinite and nilpotent, so is its center by~\ref{group properties}. Thus we can find $g_1,\dots, g_d$ in the center of $G$ such that the sets 
\[
F^{-1}F ~,~ g_1F^{-1}F ~,~ \dots ~,~ g_dF^{-1}F \ \fin \ G
\] 
are pairwise disjoint.
By~\ref{marker property}, there exists an open set $Z\subset X$ with
\begin{equation} \label{eq:F16}
\alpha_g(\quer{Z})\cap\alpha_h(\quer{Z})=\emptyset\quad\text{for all}~g\neq h~\text{in}~F
\end{equation}
and
\begin{equation}  \label{eq:F17}
X = \bigcup_{l=0}^d \bigcup_{g\in F^{-1}F} \alpha_{g_lg}(Z).
\end{equation}
For $l=0,\dots,d$ and $j=1,\dots,m$, define the finite sets
\[
 N^{(l,j)} = \set{g\in G ~|~ J g \subset g_l h_j F }  \fin G.
\]
For each $g \in N^{(l,j)}$, define the open set
\[
 Z^{(l,j,g)} = \alpha_{g}(Z) \subset X.
\]
We now claim that 
\[
X= \bigcup_{l=0}^d \bigcup_{j=1}^m \bigcup_{g\in N^{(l,j)} } Z^{(l,j,g)}.
\]
For any $x \in X$, we can apply \eqref{eq:F17} and find some $l\in\set{0,\dots,d}$ and $g\in F^{-1}F$ with $x\in\alpha_{g_{l}g}(Z)$. By our choice of $F$ and the elements $h_1,\dots,h_m$, we can find $j\in\set{1,\dots,m}$ with $Jg\subset h_{j}F$. But then $Jg_lg \subset g_lh_j F$, so we see that $g_lg\in N^{(l,j)}$ and $x\in Z^{(l,j,g_lg)}$.

We may find a partition of unity 
\[
\set{ \nu^{(l,j,g)} ~|~ g\in N^{(l,j)},~ l = 0,\dots,d,~ j = 1,2,\dots,m }
\] 
of $X$ subordinate to the open cover 
\[
\set{ Z^{(l,j,g)} ~|~ g\in N^{(l,j)},~ l = 0,\dots,d,~ j = 1,2,\dots,m }. \footnote{Note that we allow repetions with this notation.}
\] 
 For the sake of convenience, let us set $\nu^{(l,j,g)} = 0$ for all $l = 0,\dots,d$ and $j = 1,\dots,m$, whenever $g \notin N^{(l,j)}$. 

For every $l=0,\dots,d$ and $j=1,\dots,m$, we now define the finitely supported maps $\mu^{(l,j)}: G\to\CC(X)_{1,+}$ via
\[
\mu^{(l,j)}_g = \frac{1}{|J|} \sum_{h\in J} \nu^{(l,j,h^{-1}g)}\circ\alpha_{h^{-1}}
\]
We claim that these satisfy the desired properties. Firstly, the support of each map $\mu^{(l,j)}$ is contained in $J\cdot N^{(l,j)}$, and is hence finite. Secondly, since each map $\nu^{(l,j,g)}$ has values in $[0,1]$, so does each $\mu^{(l,j)}_g$ by triangle inequality. We have
\[
\begin{array}{ccl}
 \dst\sum_{l=0}^d \sum_{j=1}^m \sum_{g\in G} \mu^{(l,j)}_g
 &=& \dst\sum_{l=0}^d \sum_{j=1}^m \sum_{g\in G} \frac{1}{|J|} \sum_{h \in J}  \nu^{(l,j, h^{-1} g)} \circ \alpha_{h^{-1}} \\
 &=& \dst\frac{1}{|J|} \sum_{h \in J} \Big( \underbrace{ \sum_{l=0}^d \sum_{j=1}^m \sum_{g\in G} \nu^{(l,j, h^{-1} g)} }_{=\eins} \Big) \circ \alpha_{h^{-1}} \\
 &=& \eins.
\end{array}
\]
It follows that condition (a) holds.

Next we observe that for all $l = 0,\dots,d$, $j = 1,\dots,m$ and $g \in G$, the open support of $\mu^{(l,j)}_g$ satisfies
\begin{align*}
 \supp(\mu^{(l,j)}_g) & \subset \bigcup_{h \in J} \alpha_{h}\bigl( \supp(\nu^{(l,j, h^{-1} g)}) \bigl)  \\
 & \subset
 \begin{cases}
  \bigcup_{h \in J} \alpha_{h} ( \alpha_{h^{-1} g}(Z)) = \alpha_{g}(Z) &,~\text{if}~ g \in J\cdot N^{(l,j)} \\
  \emptyset  &,~\text{if}~ g \notin J\cdot N^{(l,j)}. 
 \end{cases}
\end{align*}
By definition of the set $N^{(l,j)}$, we have $J\cdot N^{(l,j)}\subset g_lh_jF$.
So let $g'\in G$ be an element different from $g$. If $g\notin J\cdot N^{(l,j)}$ or $g'\notin J\cdot N^{(l,j)}$, then we trivially have $ \mu_g^{(l,j)} \mu_{g'}^{(l,j)} = 0 $ because one of the functions is zero already. If $g,g'\in  J\cdot N^{(l,j)}\subset g_lh_j F$, then we can write $g=g_lh_jf_1$ and $g'=g_lh_jf_2$ for some $f_1\neq f_2$ in $F$. By \eqref{eq:F16}, it follows that $\alpha_{f_1}(Z)\cap\alpha_{f_2}(Z)=\emptyset$, and thus $\alpha_g(Z)\cap\alpha_{g'}(Z)=\emptyset$. In particular, the two functions $\mu^{(l,j)}_g$ and $\mu^{(l,j)}_{g'}$ have disjoint open supports and are therefore orthogonal. This verifies condition (b).

Lastly, for all $l = 0,\dots,d$,~ $j = 1,2,\dots,m$,~ $g\in M$ and $g' \in G$, we calculate
\[
\def\arraystretch{2}
\begin{array}{cll}
 \multicolumn{3}{l}{ \|\mu^{(l,j)}_{g'}\circ\alpha_{g^{-1}} - \mu_{gg'}^{(l,j)} \| } \\
 &=& \dst\frac{1}{|J|} \Big\| \sum_{h \in J}  \nu^{(l,j, h^{-1} g')} \circ \alpha_{h^{-1}} \circ \alpha_{g^{-1}} -  \sum_{h \in J}  \nu^{(l,j, h^{-1} g g')} \circ \alpha_{h^{-1}}  \Big\| \\
 &=& \dst\frac{1}{|J|} \Big\| \sum_{h \in J}  \nu^{(l,j, (gh)^{-1} g g')} \circ \alpha_{(gh)^{-1}}  -  \sum_{h \in J}  \nu^{(l,j, h^{-1} g g')} \circ \alpha_{h^{-1}} \Big\| \\
 &=& \dst\frac{1}{|J|} \Big\| \sum_{h \in gJ}  \nu^{(l,j, h^{-1} g g')} \circ \alpha_{h^{-1}}  -  \sum_{h \in J}  \nu^{(l,j, h^{-1} g g')} \circ \alpha_{h^{-1}} \Big\| \\
 &=& \dst\frac{1}{|J|} \Big\| \sum_{h \in gJ \setminus J}  \nu^{(l,j, h^{-1} g g')} \circ \alpha_{h^{-1}}  -  \sum_{h \in J \setminus gJ}  \nu^{(l,j, h^{-1} g g')} \circ \alpha_{h^{-1}} \Big\| \\
 &\leq& \dst\frac{1}{|J|} \Big( \sum_{h \in gJ \setminus J} \| \nu^{(l,j, h^{-1} g g')} \circ \alpha_{h^{-1}} \| + \sum_{h \in J\setminus gJ} \| \nu^{(l,j, h^{-1} g g')} \circ \alpha_{h^{-1}} \| \Big) \\
 &\leq& \dst\frac{|J \Delta g J|}{|J|} ~\stackrel{\eqref{eq:F15}}{\leq}~ \eps.
\end{array}
\]
This verifies condition (c) and finishes the proof.
\end{proof}

\begin{cor} \label{free nilpotent dimrok}
Let $G$ be a finitely generated, infinite, nilpotent group and $X$ a compact metric space of finite covering dimension. Let $\alpha: G\curvearrowright X$ be a free action. Then the \cstar-algebraic action $\bar{\alpha}: G\curvearrowright\CC(X)$ has finite Rokhlin dimension, and in fact
\[
\dimrokeins(\bar{\alpha}) \leq 3^{\ell_{\mathrm{Hir}}(G)}\cdot\dimeins(X).
\]
\end{cor}
\begin{proof}
This follows directly from~\ref{free nilpotent amdim} and~\ref{dimrok amdim}.
\end{proof}

\begin{cor} \label{free nilpotent dimnuc}
Let $G$ be a finitely generated, infinite, nilpotent group and $X$ a compact metric space of finite covering dimension. Let $\alpha: G\curvearrowright X$ be a free action. Then the transformation group \cstar-algebra $\CC(X)\rtimes_\alpha G$ has finite nuclear dimension, and in fact
\[
\dimnuceins(\CC(X)\rtimes_\alpha G) \leq 3^{\ell_{\mathrm{Hir}}(G)} \cdot \asdimeins(\square_s G) \cdot \dimeins(X)^2.
\]
\end{cor}
\begin{proof}
This follows directly from~\ref{free nilpotent dimrok},~\ref{asdim nilpotent} and~\ref{Hauptsatz}.
\end{proof}

\begin{rem}
We remark that one can generalize the above results~\ref{free nilpotent dimrok} and~\ref{free nilpotent dimnuc} to actions on locally compact spaces. A rigorous treatment of this can be found in the first author's dissertation~\cite{Szabo_Diss}, which follows a very similar approach, using a generalized version of~\ref{marker property}. (Note that a variant of this can be found in the appendix of~\cite{HirshbergWu15}.)
\end{rem}

\begin{rem}
Through recent groundbreaking progress in the Elliott programme by many hands, the combined main results of~\cite{GongLinNiu14, ElliottNiu15, ElliottGongLinNiu15, TikuisisWhiteWinter15} have completed the classification of separable, unital, simple \cstar-algebras that satisfy the UCT and have finite nuclear dimension. Now transformation group \cstar-algebras of amenable groups are well-known to satisfy the UCT due to~\cite{Tu99}, and they are further known to be simple if the action is free and minimal~\cite{KawamuraTomiyama90}. Therefore, the main result of this section implies the classifiability of a large class of simple transformation group \cstar-algebras:
\end{rem}

\begin{theorem} \label{crossed product classifiable}
Let $G$ be a finitely generated, infinite, nilpotent group and $X$ a compact metric space of finite covering dimension. Let $\alpha: G\curvearrowright X$ be a free and minimal action. Then the transformation group \cstar-algebra $\CC(X)\rtimes_\alpha G$ is a simple ASH algebra of topological dimension at most 2.
\end{theorem}


\section{Amenability dimension for topological actions} \label{sec:dimBLR}

In this section, we show that for topological actions, the invariant $\amdim(\bar{\alpha})$ defined in~\ref{amdim} coincides with a dimension concept studied by Guentner, Willett and Yu in~\cite{GuentnerWillettYu15}\footnote{Note that instead of writing ``$\amdim(\bar{\alpha})\leq d$'' they write ``$\alpha$ is $d$-BLR''. Our terminology is inspired by an earlier draft of their paper. Another closely related notion they introduced in the same paper is termed dynamic asymptotic dimension.}. We remark that before this work, a similar concept had been defined, at least implicitly, in the pioneering work~\cite{BartelsLuckReich08} of Bartels, L\"{u}ck and Reich on the Farrell-Jones conjecture for hyperbolic groups.

\defi Let $G$ be a countable, discrete group and let $X$ be a compact metric space. Let $\alpha: G\curvearrowright X$ be an action. Then we denote by $\Delta: G\curvearrowright G\times X$ the diagonal action defined by $\Delta_g(h,x) = (hg^{-1}, \alpha_g(x))$.

{
\defi[see {\cite[4.10, 4.14]{GuentnerWillettYu15}}] \label{top-amdim}
Let $G$ be a countable, discrete group and let $X$ be a compact metric space. Let $\alpha: G\curvearrowright X$ be a free action.
The \emph{amenability dimension} of $\alpha$, denoted $\amdim(\alpha)$, is defined to be the smallest natural number $d$ with the following property: 
For any finite subset $M \fin G$, there exists an open cover $\CU$ of $G \times X$ satisfying:
 {\renewcommand{\theenumi}{\alph{enumi}}
 \begin{enumerate}
  \item $\CU$ is a so-called \emph{free $G$-cover} with regard to the diagonal action $\Delta$; that is, for any $U \in \CU$ and $1\neq g \in G$, we have $\Delta_g(U) \in \CU$ and $U \cap \Delta_g(U) = \emptyset$.
  \item The multiplicity of $\CU$ is at most $d+1$.
  \item For any $x\in X$, the set $M \times\set{x}$ is contained in a member of $\CU$.
 \end{enumerate}
 }
If no such $d$ exists, we write $\amdim(\alpha)=\infty$. \\
}

Next we shall see that this definition is a special case of~\ref{amdim}, thus reconciling the potential conflict of notation. To this end, let us fix some terminology: 

\defi For us, an \emph{abstract $G$-simplicial complex} $Z$ consists of:
\begin{itemize}
 \item a set $Z_0$, called the set of \emph{vertices}, which comes equipped with a $G$-action, and
 \item a $G$-invariant collection of its finite subsets closed under taking subsets, called the collection of \emph{simplices}. 
\end{itemize}
We often write $\sigma \in Z$ to denote that $\sigma$ is a simplex of $Z$. The dimension of a simplex is the cardinality of the corresponding finite subset minus $1$, and the dimension of the abstract $G$-simplicial complex is the supremum of the dimensions of its simplices. The \emph{geometric realization} of an abstract $G$-simplicial complex $Z$, denoted as $|Z|$, is the set of formal sums
\[
|Z| = \set{ \sum_{v\in\sigma} \lambda_v v ~\Big|~ \sigma\in Z~\text{and}~ \lambda_v\geq 0~\text{with}~\sum_{v\in\sigma} \lambda_v = 1 }.
\]
By convention, $\lambda_v$ is defined to be zero for $v \in Z_0 \setminus \sigma$. Similarly for a simplex $\sigma$ of $Z$, we define its geometric realization $|\sigma|$ to be 
\[
 \set{ \sum_{v\in\sigma} \lambda_v v ~\Big|~  \lambda_v\geq 0~\text{with}~\sum_{v\in\sigma} \lambda_v = 1 } \subset |Z| .
\]
The space $|Z|$ carries a natural $G$-action, where 
\[
 g \cdot \sum_{v\in\sigma} \lambda_v v = \sum_{v\in\sigma} \lambda_v (g \cdot v) \in |g \cdot \sigma|
\]
for any $g \in G$, $\sigma \in Z$ and $\sum_{v\in\sigma} \lambda_v v \in |\sigma|$. If this action is free, then we call $Z$ an abstract \emph{free} $G$-simplicial complex. Usually $Z$ is equipped with the weak topology, but for our purposes we consider the $\ell^1$-topology, induced by the $\ell^1$-metric $d^1: Z \times Z \to [0, 2]$ given by 
\[ 
d^1\Big( \sum_{v \in \sigma} \lambda_v v , \sum_{v \in \sigma'} \mu_v v \Big) = \sum_{v \in \sigma \cup \sigma'} |\mu_v - \lambda_v | .
\]
This metric is obviously $G$-invariant. 

{
\lemma \label{amdim functions}
Let $G$ be a countable, discrete group and let $X$ be a compact Hausdorff space. Let $\alpha: G\curvearrowright X$ be a free action and $d\in\IN$ a natural number.
Then the following are equivalent:
\begin{enumerate}
  \item The amenability dimension of $\alpha$ is at most $d$.
  \item For every $M \fin G$ and $\eps > 0$, there exists a free $G$-simplicial complex $Z$ of dimension at most $d$ together with a continuous map $\phi: X \to |Z|$ that is $(M, \eps)$-approximately equivariant\footnote{In \cite{GuentnerWillettYu15}, the terminology ``$(M, \eps)$-equivariant'' is used instead.}, i.e. $ d^1( \phi(\alpha_g(x)), g \cdot \phi(x) ) \leq \eps$ for all $x \in X$ and $g \in M$.
  \item For the induced action $\bar{\alpha}: G \curvearrowright \CC(X)$, we have $\amdim(\bar{\alpha}) \leq d$.
\end{enumerate}
}
\begin{proof}$(1)\implies (2):$ We follow the proof of \cite[4.6, (ii)$\implies$(i)]{GuentnerWillettYu15}, though some extra care is necessary as the authors of that paper did not discuss free simplicial complexes, but rather complexes whose \emph{vertex stabilizers} are from a fixed collection $\mathcal{F}$ of subgroups of $G$ (e.g., the collection of all finite subgroups; see \cite[4.3]{GuentnerWillettYu15}). 
	
	Let $M \fin G$ and $\eps > 0$ be given. Replacing $M$ by $M \cup M^{-1} \cup \{1\}$, we may assume that $M$ is symmetric and contains the identity element. Pick $n$ with 
	\[
		\frac{(2d+2)(4d+6)}{n} < \varepsilon \; .
	\]
	Since the amenability dimension of $\alpha$ is at most $d$, we can find a cover $\CU$ of $G\times X$ as in Definition~\ref{top-amdim} with $M$ replaced by $M^n$. Let $Z$ be the nerve of this cover, i.e. $Z_0 = \CU$ and a finite subset $\{U_1, U_2, \ldots, U_k \}$ of $\CU$ is in $Z$ if and only if $\bigcap_{i=1}^{k} U_i \not= \varnothing$. Since $\CU$ is $G$-invariant, we see that $Z$ becomes a $G$-simplical complex under the $G$-action on the vertex set $\CU$. The dimension of $Z$ is simply the multiplicity of $\CU$ minus $1$ and thus is at most $d$. We also observe that the action of $G$ on $|Z|$ is free. Indeed, if this were not the case, we would be able to find a nontrivial element $g \in G$ and a point $z$ such that $g \cdot z = z$. Write $z = \sum_{v \in \sigma} \lambda_v v$ with $\lambda_v > 0$ for each $v \in \sigma$. Then the simplex $\sigma$ is fixed by $g$, too. Thus picking an element $U$ in $\sigma$, viewed as an element of $\CU$, we would have $g \cdot U \cap U \not= \varnothing$, which is a contradiction to the fact that $\CU$ is a free $G$-cover. In summary, $Z$ is a free $G$-simplicial complex of dimension at most $d$. 
	
	The construction of the continuous map $\phi \colon X \to |Z|$ is identical to that in the proof of \cite[4.6, (ii)$\implies$(i)]{GuentnerWillettYu15} and is thus omitted. 

$(2)\implies (3):$ Given any free $G$-simplicial complex $Z$ of dimension no more than $d$, there is a canonical $G$-invariant cover $\CU$ constructed as follows. For all $l=0,\dots,d$ and every $l$-dimensional simplex $\sigma$ of $Z$, we define an open subset 
\[
U_\sigma = \set{ \sum_{v\in Z_0} \lambda_v v ~|~ \lambda_v > \lambda_{v'}~\text{for all}~ v\in\sigma~\text{and}~ v'\in Z_0\setminus\sigma  }.
\]
Then $U_{g\cdot\sigma} = g \cdot U_\sigma$ and $U_\sigma \cap U_{\sigma'} = \emptyset$ for $\sigma \neq \sigma'$ of the same dimension. Now $\CU$ is defined to be the collection $\CU = \set{ U_\sigma ~|~ \sigma\in Z}$, which gives us an open cover of $|Z|$ of multiplicity at most $d+1$. 

This allows us to define a partition of unity $\set{ \nu_\sigma }_{\sigma\in Z}$ subordinate to $\CU$ by setting
\[
\dst\nu_\sigma : |Z| \to [0,1]~,\quad z \mapsto \frac{d^1(z, |Z| \setminus U_\sigma)}{ \sum_{\sigma'\in Z} d^1(z, |Z| \setminus U_{\sigma'}) } .
\]
Observe that $\nu_{g\cdot\sigma} = g \cdot \nu_\sigma$ for all $\sigma\in Z$ and $\nu_\sigma \cdot \nu_{\sigma'} = 0 $, whenever $\sigma \neq \sigma'$ are two simplices of the same dimension. 

Moreover, we claim that each $ \nu_\sigma $ is Lipschitz with respect to the constant $(d+1)(d+2)$. Since $|Z|$ is a geodesic space, with geodesics being piecewise linear paths in $|Z|$, it suffices to show that all the directional derivatives of $ \nu_\sigma $ along linear paths are bounded by this constant. Observe that for any $ z  = \sum_{v \in Z_0} \lambda_v v \in U_\sigma $, the point in $|Z| \setminus U_\sigma$ that is closest to $z$ differs from $z$ only by replacing one of the $(\dim(\sigma) + 1 )$-th greatest coordinates of $ z $ and one of the $(\dim(\sigma) + 2 )$-th greatest coordinates of $ z $ by the average of the two values, and thus we have
\[
	d^1(z, |Z| \setminus U_\sigma) = \lambda^{(\dim(\sigma) + 1 )} - \lambda^{(\dim(\sigma) + 2 )} ~,
\]
where $ \lambda^{(l)}$ is the value of the $l$-th greatest coordinate of $ z $ for all $l\in\IN$. Hence for any $ z  = \sum_{v \in Z_0} \lambda_v v \in |Z| $, we have
\[
\dst\nu_\sigma ( z ) = \chi_{{U_\sigma}} (z) \cdot \frac{ \lambda^{(\dim(\sigma) + 1 )} - \lambda^{(\dim(\sigma) + 2 )} }{ \sum_{l =2}^{\operatorname{card}(Z_0)} (\lambda^{(l -1 )} - \lambda^{(l )}) } = \chi_{{U_\sigma}} (z) \cdot \frac{ \lambda^{(\dim(\sigma) + 1 )} - \lambda^{(\dim(\sigma) + 2 )} }{ \lambda^{(1)} } ~,
\]
where $\chi_{{U_\sigma}} $ is the characteristic function for $ U_\sigma $ and $\operatorname{card}(Z_0)$ is the cardinality of $Z_0$ (in fact, we are only dealing with a finite sum). Since $ z $ has no more than $d+1$ positive terms and they sum up to $1$, it follows that $\lambda^{(1)} \ge \frac{1}{d+1} $. Thus for any finite sum $ w  = \sum_{v \in Z_0} \eta_v v $ with $ \sum_{v \in Z_0} |\eta_v| = 1 $ and $ z + t w \in |Z| $ for any $t$ in a small neighborhood of $0$ in $\mathbb{R}^{\ge 0}$, we have
\[
\dst \Big| \frac{\mathrm{d}\:\nu_\sigma ( z + t w )}{\mathrm{d}t} \Big|_{t=0^+} \Big| ~\le~ \frac{ 1 }{ \lambda^{(1)} } + \frac{ 1 }{ ( \lambda^{(1)} )^2 } \le (d+1)(d+2) ,
\]
which proves what we claimed.\footnote{Another way to give an estimate of the Lipschitz constant is by showing that the cover $ \CU$ has Lebesgue number at least $\frac{2}{(d+1)(d+2)}$ and use~\cite[4.3.5]{NowakYuLSG}.}

Now given any finite subset $M \fin G$ and $\eps > 0$, we apply the assumption to get a free $G$-simplicial complex $Z$ of dimension no more than $d$ and a $(M^{-1}, \frac{\eps}{2(d+1)(d+2)})$-approximately equivariant continuous map $\phi: X \to |Z|$. Since $X$ is compact, by~\cite[4.5]{GuentnerWillettYu15}, we may assume without loss of generality that the image of $\phi$ intersects finitely many simplices in $Z$. Let $\set{ \nu_\sigma }$ be the partition constructed above. For every $l = 0, \dots, d$, pick a set $\Sigma^{(l)}$ of representatives for the $G$-action on the set of $l$-dimensional simplices in $Z$, and define $\nu^{(l)} = \sum_{\sigma \in \Sigma^{(l)}} \nu_\sigma$.
Since there is only one non-zero summand at every point $z \in |Z|$, the sum is well defined and we see that each $\nu^{(l)}$ is also Lipschitz with respect to the constant $(d+1)(d+2)$. Additionally, we have $\nu^{(l)} \cdot (g \cdot \nu^{(l)})=0$ for all $g \in G\setminus\set{1}$ and 
\[
\sum_{l=0}^d \sum_{g \in G} (g \cdot \nu^{(l)})(z) = 1 \quad\text{for all}~ z \in |Z|.
\]
We define the maps $\mu^{(l)}: G\to \CC(X)_+$ as the pullbacks $\mu^{(l)}_g = (g \cdot \nu^{(l)} ) \circ \phi$ for all $l = 0, \dots, d$ and $g\in G$. We claim that these maps satisfy the conditions in~\ref{lem:amdim-abelian}. As assumed above, the image of $\phi$ intersects only finitely many simplices in $Z$. Since, moreover, each simplex intersects only finitely many $U_\sigma$'s, it follows that the image of $\phi$ intersects only finitely many $U_\sigma$'s and thus there are only finitely many $g \in G$ such that $\mu^{(l)}_g \neq 0 $. Conditions (a) and (b) follow directly from the above properties of the collection $\set{ \nu^{(l)} }_{l=0,\dots,d}$. Finally we check that for all $l = 0,\dots,d$,~ $g\in M$,~ $h\in G$ and $x \in X$, we have
\[
\begin{array}{cll}
\bar{\alpha}_g(\mu_h^{(l)})(x) &=& \mu_h^{(l)}(\alpha_{g^{-1}}(x)) \\
&=& (h\cdot\nu^{(l)})\circ\phi\circ\alpha_{g^{-1}}(x) \\
&=_{\eps/2}& (h\cdot\nu^{(l)})(g^{-1}\cdot\phi(x)) \\
&=& (gh\cdot\nu^{(l)})\circ\phi(x) = \mu_{gh}^{(l)}(x). 
\end{array}
\]
Note that for the intermediate approximation step, we have used that the function $h\cdot\nu^{(l)}$ is Lipschitz with respect to the constant $(d+1)(d+2)$, combined with $d^1( g^{-1}\cdot\phi(x), \phi\circ\alpha_{g^{-1}}(x))\leq\eps/2(d+1)(d+2)$. Hence condition (c) follows.

$(3)\implies (1):$ Given any finite subset $M \fin G$, we use the assumption to get finitely supported maps $\mu^{(l)}: G\to \CC(X)_+$ satisfying the conditions in~\ref{lem:amdim-abelian} for $(M, \frac{1}{d+1})$. For $l= 0,\dots,d$, define an open set
\[
U^{(l)} = \set{ (g, x)\in G\times X ~|~ \mu_g^{(l)}(\alpha_{g}(x)) > 0 }
\]
and put $\CU^{(l)} := \set{ \Delta_g(U^{(l)}) ~|~ g\in G }$. We claim that $\CU := \CU^{(0)} \cup \dots \cup \CU^{(d)} $ is an open cover satisfying the conditions in~\ref{top-amdim}.

First we observe that the $G$-invariance of $\CU$ is built into the definition. Now it holds for any $(h, x) \in G\times X$, $l=0,\dots,d$ and $g\in G$ that $(h, x)\in \Delta_g\left( U^{(l)} \right)$ if and only if $\mu_{hg}^{(l)}(\alpha_{h}(x)) > 0$. Since for two distinct $g, g' \in G$, one has $\mu_{hg}^{(l)} \mu_{hg'}^{(l)} = 0$, it follows that $(h, x)$ cannot be in both $\Delta_g(U^{(l)})$ and $ \Delta_{g'}(U^{(l)})$. This proves that for each $l=0,\dots,d$, the collection $\CU^{(l)}$ is $G$-invariant and consists of pairwise disjoint sets.

Finally, we observe that for any $(h, x)\in G\times X$, there is $g_0\in G$ and $l_0\in\set{0,\dots,d}$ such that $\mu^{(l)}_{hg_0}(\alpha_{h}(x)) \geq \frac{1}{d+1}$. This is because we have $1=\sum_{l=0}^d\sum_{g\in G} \mu_g^{(l)}(\alpha_h(x))$ and the collections of functions $\set{\mu^{(l)}_g ~|~ g\in G}$ consist of pairwise orthogonal functions for all $l=0,\dots,d$.
It follows from condition (c) that for any $b \in M$,
\[
\mu^{(l)}_{bhg_0}(\alpha_{bh}(x)) > \mu^{(l)}_{hg_0}(\alpha_h(x))-\frac{1}{d+1} ~\geq~ 0.
\]
But this is equivalent to $(bh, x)\in\Delta_{g_0}(U^{(l)})$. Hence it follows that $Mh\times\set{x}$ is contained in a member of $\CU$. In particular, $\CU$ indeed covers $G\times X$. This yields all the conditions in order to deduce $\amdim(\alpha)\leq d$.
\end{proof}

It turns out that amenability dimension for free actions is a notion that generally behaves well with respect to the nuclear dimension of the transformation group \cstar-algebra. The following result is due to Guentner, Willett and Yu:

\begin{theorem}[see {\cite[8.6, 4.11]{GuentnerWillettYu15}}] \label{amdim dimnuc}
Let $G$ be a countable, discrete group and $X$ a compact metric space. Let $\alpha: G\curvearrowright X$ be a free action. Then
\[
\dimnuceins(\CC(X)\rtimes_\alpha G) \leq \amdimeins(\alpha)\cdot\dimeins(X).
\]
\end{theorem}

This theorem allows us to provide a better estimate for~\ref{free nilpotent dimnuc} when $X$ is compact. In view of Definition~\ref{amdim} and Lemma~\ref{amdim functions}, it is interesting to ask in what generality this formula holds beyond the commutative case, with $\amdim(\alpha)$ replaced by $\amdim(\bar{\alpha})$. 

\begin{cor} \label{free nilpotent dimnuc better}
Let $G$ be a finitely generated, infinite nilpotent group and $X$ a compact metric space of finite covering dimension. Let $\alpha: G\curvearrowright X$ be a free action. Then the transformation group \cstar-algebra $\CC(X)\rtimes_\alpha G$ has finite nuclear dimension, and in fact
\[
\dimnuceins(\CC(X)\rtimes_\alpha G) \leq 3^{\ell_{\mathrm{Hirsch}}(G)} \cdot \dimeins(X)^2.
\]
\end{cor}
\begin{proof}
This follows directly from~\ref{free nilpotent amdim},~\ref{amdim functions} and~\ref{amdim dimnuc}.
\end{proof}

\begin{rem}
Some time after the initial preprint version of this paper was published, Bartels~\cite{Bartels16} has independently proved a more general version of~\ref{free nilpotent amdim} with slightly different methods. In particular, his result~\cite[1.10]{Bartels16} in conjuction with \ref{amdim dimnuc} implies that the finiteness part of~\ref{free nilpotent amdim} holds for virtually nilpotent groups, although no specific estimate is obtained this way. Using his results instead of ours, we thus obtain the following improvement of~\ref{crossed product classifiable}:
\end{rem}

\begin{theorem}\label{thm:Bartels-extension}
Let $G$ be a finitely generated, infinite, virtually nilpotent group and $X$ a compact metric space of finite covering dimension. Let $\alpha: G\curvearrowright X$ be a free and minimal action. Then the \cstar-algebra $\CC(X)\rtimes_\alpha G$ is a simple ASH algebra of topological dimension at most 2.
\end{theorem}


\section{Rokhlin dimension with commuting towers}\label{sec:commuting-towers}

\noindent
In this section, we consider the notion of Rokhlin dimension with commuting towers, in analogy to what has been done in~\cite[Section 5]{HirshbergWinterZacharias14}. Similarly as in~\cite[Section 5]{HirshbergWinterZacharias14}, it turns out that cocycle actions with finite Rokhlin dimension with commuting towers preserve $\CZ$-stability. Moreover, borrowing a technique from~\cite{HirshbergSzaboWinterWu16}, this turns out to be true even for $\CD$-stability, for any strongly self-absorbing \cstar-algebra $\CD$. Let us first recall some relevant notions. 

\begin{defi}[see\ {\cite[1.3]{TomsWinter07}}] 
\label{defi:SSA}
A separable, unital \cstar-algebra $\CD$ is called strongly self-absorbing, if there is an isomorphism $\phi: \CD \to \CD \otimes \CD$ that is approximately unitarily equivalent to $\mathrm{id}_\CD \otimes \eins_\CD : \CD \to \CD \otimes \CD$, i.e.\ there is a sequence of unitaries $u_n \in \CD \otimes \CD$ such that 
\[
\phi(x) = \lim_{n\to \infty} u_n(x \otimes \eins_\CD)u_n^*
\]
for all $x\in\CD$.

A \cstar-algebra $A$ is called $\CD$-stable if $A \cong A \otimes \CD$. 
\end{defi}

Strongly self-absorbing \cstar-algebras are always simple and nuclear. Known examples are the Jiang-Su algebra $\mathcal{Z}$, any UHF algebra of infinite type, the Cuntz algebras $\mathcal{O}_\infty$ and $\mathcal{O}_2$, as well as countable tensor products of these.
They play an important role in the Elliott classification programme in that $\CD$-stability can be considered as an important regularity property. Most notably, $\CZ$-stability appears as a prerequisite for a \cstar-algebra to be inside a classifiable class. We are thus led to the problem of determining cases in which $\CD$-stability is preserved under forming crossed products.

Next we introduce the main definition of the section, which is a strengthening of~\ref{dimrok central sequence} and~\ref{dimrok allgemein}.

\begin{defi} \label{dimrokc} Let $A$ be a separable \cstar-algebra, $G$ a residually finite group and $(\alpha,w): G\curvearrowright A$ a cocycle action. Let $H\subset G$ be a subgroup of finite index. We say that $\alpha$ has Rokhlin dimension $d$ with commuting towers and relative to $H$, and write $\dimrokc(\alpha, H)=d$, if $d$ is the smallest number such that there exist equivariant order zero maps
\[
\phi_l: (\CC(G/H),G\text{-shift})\to (F_\infty(A), \tilde{\alpha}_\infty)\quad (l=0,\dots,d)
\]
with pairwise commuting ranges and $\eins = \phi_0(\eins)+\dots+\phi_d(\eins)$.

Let $\sigma=(G_n)_n\in\Lambda(G)$ be a regular approximation of $G$. We define 
\[
\dimrokc(\alpha,\sigma) = \sup_{n\in\IN}~ \dimrokc(\alpha,G_n)
\]
and
\[
\dimrokc(\alpha) = \sup\set{ \dimrokc(\alpha, H) ~|~ H\subset G,~ [G:H]<\infty }.
\]
\end{defi}

Before we can prove the main theorem of the section, we need two results from the literature. The first may be called a folklore result. Note that the unital case of it was already implicitly at the core of the results of~\cite[4.6]{IzumiMatui10},~\cite[4.8]{MatuiSato12_2} and~\cite[4.9]{MatuiSato14}. The proof essentially goes by a similar method as in~\cite[2.3]{TomsWinter07},~\cite[4.11]{Kirchberg04} or~\cite[2.3.5, 7.2.1, 7.2.2]{Rordam}. A detailed treatment has been given by the first author in~\cite{Szabo15ssa} for more general (cocycle) actions. 

\begin{theorem}[see {\cite[3.8]{Szabo15ssa}}] \label{Matui}
Let $G$ be a countable, discrete group, $A$ a separable \cstar-algebra and let $\CD$ be a strongly self-absorbing \cstar-algebra.
Let $(\alpha, w): G\curvearrowright A$ be a cocycle action. Then $(\alpha,w)$ is cocycle conjugate to $(\alpha\otimes\id_\CD, w\otimes\eins_\CD)$ if and only if there exists a unital $*$-homomorphism from $\CD$ to the fixed point algebra $F_\infty(A)^{\tilde{\alpha}_\infty}$.
\end{theorem}

The second is an important technical lemma for proving the existence of $*$-homomorphisms from strongly self-absorbing \cstar-algebras.

\begin{lemma}[see\ {\cite[5.8]{HirshbergSzaboWinterWu16}}] \label{order zero glueing}
Let $\CD$ be a strongly self-absorbing \cstar-algebra. Let $B$ be a unital \cstar-algebra. Suppose that $\psi_1,\dots,\psi_n: \CD\to B$ are c.p.c.~order zero maps with pairwise commuting ranges such that $\psi_1(\eins_\CD) +\dots + \psi_n(\eins_\CD) = \eins_B$. Then
there exists a unital $*$-homomorphism from $\CD$ to $B$.
\end{lemma}

\begin{rem}
We note that one can use~\cite[7.6]{KirchbergRordam12} to see that~\ref{order zero glueing} holds, if one replaces $\CD$ either by a matrix algebra or a dimension drop algebra. This would be enough to prove~\ref{dimrokc sigma D} for $\CD$ being either $\CZ$ or a UHF algebra of infinite type\footnote{In fact, this was done in a previous preprint version of this paper.}.
\end{rem}

The main theorem of the section is a generalization of~\cite[5.8, 5.9]{HirshbergWinterZacharias14}:

{
\theorem \label{dimrokc sigma D}
Let $\CD$ be a strongly self-absorbing \cstar-algebra and $A$ a separable, $\CD$-stable \cstar-algebra. Let $G$ be a countable, residually finite group and $(\alpha,w): G\curvearrowright A$ a cocycle action. Assume that $\sigma\in\Lambda(G)$ is a regular approximation with $\asdim(\square_\sigma G)<\infty$ and $\dimrokc(\alpha, \sigma)<\infty$. Then $(\alpha,w)$ is cocycle conjugate to $(\alpha\otimes\id_\CD, w\otimes\eins_\CD)$, and in particular the twisted crossed product $A\rtimes_{\alpha,w} G$ is $\CD$-stable.
}
\begin{proof}
Denote $s=\asdim(\square_\sigma G)$ and $d=\dimrokc(\alpha, \sigma)$. Denote $\sigma=(G_n)_n$. 
As $A$ is $\CD$-stable, there exists a unital $*$-homomorphism from $\CD$ to $F_\infty(A)$.
By a standard reindexation trick (such as~\cite[1.13]{Kirchberg04}), we may find unital embeddings $\iota_{l,j}: \CD\to F_\infty(A)$ for $l=0,\dots,d$ and $j=0,\dots,s$ such that
\[
[\iota_{l_1,j_1}(x), \tilde{\alpha}_{\infty,g}(\iota_{l_2,j_2}(y))] = 0\quad\text{for all}~x,y\in \CD,~ g\in G~\text{and}~(l_1,j_1)\neq (l_2,j_2).
\] 
Let $M\fin G$ and $\eps>0$ be arbitrary. By~\ref{decay functions}(3), we can find $n$ and finitely supported functions $\nu^{(j)}: G\to [0,1]$ for $j=0,\dots,s$ satisfying
{
\renewcommand{\theenumi}{\alph{enumi}}
\begin{enumerate}
 \item For every $j=0,\dots,s$, one has
\[
\supp(\nu^{(j)})\cap\supp(\nu^{(j)})h=\emptyset\quad\text{for all}~h\in G_n\setminus\set{1}.
\]
 \item For every $g\in G$, one has
\[
\sum_{j=0}^s \sum_{h\in G_n}  \nu^{(j)}(gh)=1.
\]
 \item For every $j=0,\dots,s$ and $g\in M$, one has 
\[
\|\nu^{(j)}-\nu^{(j)}(g\cdot\_\!\_)\|_\infty\leq\eps.
\]
\end{enumerate}
}
In order to shorten the notation for the rest of the proof, we will write $\nu^{(j)}_h=\nu^{(j)}(h)$ for all $j=0,\dots,s$ and $h\in G$.

By definition, there are equivariant c.p.c.~order zero maps
\[
\phi_l: (\CC(G/G_n),G\text{-shift})\to (F_\infty(A), \tilde{\alpha}_\infty)\quad (l=0,\dots,d)
\]
with pairwise commuting ranges and $\eins = \phi_0(\eins)+\dots+\phi_d(\eins)$. Without loss of generality, we may assume that the image of each map $\phi_l$ commutes with the image of each map of the form $\tilde{\alpha}_{\infty,g}\circ\iota_{l',j}$ for $l,l'=0,\dots,d$, $j=0,\dots,s$ and $g\in G$. 

For each $\bar{g}\in G/G_n$, let $e_{\bar{g}}\in\CC(G/G_n)$ denote the characteristic function of $\set{\bar{g}}$. For each $j=0,\dots,s$ and $l=0,\dots,d$, define the maps $\psi_{l,j}: \CD\to F_\infty(A)$ via
\[
\psi_{l,j} = \sum_{g\in G} \nu^{(j)}_g\cdot \phi_l(e_{\bar{g}})\cdot\tilde{\alpha}_{\infty,g}\circ\iota_{l,j}.
\]
Then each of the maps $\psi_{l,j}$ is c.p.c.~order zero. They also have pairwise commuting ranges because we have
\[
\|[\psi_{l_1,j_1}(x_1),\psi_{l_2,j_2}(x_2)]\| \leq \sup_{g_1,g_2\in G} \big\| \big[ \tilde{\alpha}_{\infty, g_1}(\iota_{l_1,j_1}(x_1)), \tilde{\alpha}_{\infty,g_2}(\iota_{l_2,j_2}(x_2)) \big] \big\| = 0
\]
for all $(l_1,j_1)\neq (l_2, j_2)$ and all $x_1,x_2\in\CD$. We have moreover
\[
\begin{array}{lll}
 \dst\sum_{l=0}^d\sum_{j=0}^s \psi_{l,j}(\eins) &=& \dst \sum_{l=0}^d \sum_{j=0}^s \sum_{g\in G}\nu^{(j)}_g\cdot \phi_l(e_{\bar{g}})\cdot\tilde{\alpha}_{\infty,g}\circ\iota_{l,j}(\eins) \\\\
&=& \dst \sum_{l=0}^d \sum_{j=0}^s \sum_{g\in G}\nu^{(j)}_g\cdot \phi_l(e_{\bar{g}}) \\\\
&=& \dst \sum_{l=0}^d \sum_{\bar{g}\in G/G_n} \phi_l(e_{\bar{g}})\cdot \sum_{j=0}^s \sum_{h\in\bar{g}} \nu^{(j)}_h \\\\
&=& \dst \sum_{l=0}^d \phi_l(\eins) \quad = \quad \eins
.\end{array}
\]
Condition (c) of the functions $\nu^{(j)}$ ensures that for all $l=0,\dots,d$, $j=0,\dots,s$ and $g\in M$, we have $\| \psi_{l,j} - \tilde{\alpha}_{\infty,g}\circ\psi_{l,j}\|\leq\eps$.

Since $M\fin G$ and $\eps>0$ were arbitrary, we can apply a standard diagonal sequence argument to construct new c.p.c.~order zero maps $\psi_{l,j}: \CD\to F(A)^{\tilde{\alpha}_\infty}$ for $l=0,\dots,d$ and $j=0,\dots,s$ with pairwise commuting ranges and
\[
\sum_{l=0}^d\sum_{j=0}^s \psi_{l,j}(\eins) = \eins.
\]
We can now apply~\ref{order zero glueing} to obtain the existence of some unital $*$-homomorphism from $\CD$ to $F_\infty(A)^{\tilde{\alpha}_\infty}$. The claim now follows from~\ref{Matui}.
\end{proof}

\begin{rem}
The usefulness of a result like~\ref{dimrokc sigma D} is showcased in the following two non-trivial applications:
\begin{enumerate}
\item Let $G$ be a finite group. If $\alpha: G\curvearrowright\CO_2$ is an arbitrary action with $\dimrokc(\alpha)<\infty$, then it follows from \ref{dimrokc sigma D} that $\alpha$ absorbs the trivial $G$-action on $\CO_2$. Thus by \cite[4.2]{Izumi04}, $\alpha$ is uniquely determined up to conjugacy. In fact, the same argument works for residually finite groups with $\asdim(\square_s G)<\infty$ due to \cite[5.5]{Szabo16kp}. 
\item In contrast to the above, nice torsion-free groups such as $G=\IZ^n$ exhibit a different behavior. There are results such as \cite{Liao15, Liao16} showing that strongly outer actions of $G$ on nice, classifiable finite \cstar-algebras have finite Rokhlin dimension; although commuting towers are not arranged in the cited articles, this will be done in the first author's forthcoming work. It then follows from~\ref{dimrokc sigma D} that equivariant $\CD$-absorption is an automatic phenomenon for nice, torsion-free group actions on classifiable $\CD$-stable \cstar-algebras, as one would expect from $K$-theoretic considerations; cf.\ \cite[4.12 and 4.16]{Szabo16kp}.
\end{enumerate}
In light of the fact that all outer $G$-actions on a Kirchberg algebra have Rokhlin dimension at most one (see \cite[6.3]{Szabo16kp}), the above discussion supports the idea that, as compared to the condition of finite Rokhlin dimension, finite Rokhlin dimension with commuting towers is a more rigid notion when the acting group has torsion. See \cite{HirshbergPhillips15}, where such a phenomenon was discovered much earlier by $K$-theoretic methods.
\end{rem}


\section{Genericity of finite Rokhlin dimension on $\CZ$-stable \cstar-algebras}\label{sec:genericity}

\noindent
In this section, we investigate the genericity of cocycle actions with finite Rokhlin dimension on \cstar-algebras absorbing either $\CZ$ or a certain UHF algebra of infinite type, in analogy to what has been done in~\cite[Section 3]{HirshbergWinterZacharias14}. As we will see, the main argument in the proof of generecity relies on a more general principle that has nothing to do with Rokhlin dimension in particular.
For instance, we also prove that for every strongly self-absorbing \cstar-algebra $\CD$ and every separable, $\CD$-stable \cstar-algebra $A$, a cocycle action on $A$ will generically absorb the trivial action on $\CD$ tensorially.

Let us first establish some notation.

\begin{defi} \label{polish}
Let $G$ be a countable, discrete group and $A$ a \cstar-algebra. Consider the set $\coact(G,A)$ of all cocycle actions $(\alpha, w): G\curvearrowright A$. For a pair of $F\fin A$ and $M\fin G$, the map
\[
d_{F,M}: \coact(G,A)\times \coact(G,A) \to \IR^{\geq 0}
\]
given by
\[
\begin{array}{ccl}
d_{F,M}\bigl( (\alpha^1, w^1) , (\alpha^2, w^2) \bigl) &=& \dst\max_{a\in F, g\in M} \|\alpha^1_g(a)-\alpha^2_g(a)\| \\
&& \dst + \max_{a\in F, g\in M} \big( \|(w^1_g - w^2_g)a\|+\|a(w^1_g-w^2_g)\| \big)
\end{array}
\]
defines a pseudometric on $\coact(G,A)$. We equip $\coact(G,A)$ with the topology induced by all these pseudometrics. In other words, a net $(\alpha^\lambda,w^\lambda)$ of cocycle $G$-actions on a \cstar-algebra $A$ converges to a cocycle action $(\alpha,w)$, if $\alpha^\lambda_g\to\alpha_g$ in point-norm and $w^\lambda_g\to w_g$ strictly, for all $g\in G$.
\end{defi}

\begin{rem}
By choosing a countable, dense set, one can see that when $A$ is separable, this topology is metrizable, and in fact $\coact(G,A)$ becomes a Polish space with any such metric.
A subset of cocycle actions in $\coact(G,A)$ is called generic, if it contains a countable intersection of dense open subsets.
\end{rem}

{
\prop \label{Gdelta}
Let $G$ be a countable, discrete group and $A$ a separable \cstar-algebra. Let $\CD$ be a strongly self-absorbing \cstar-algebra. Assume that $A\cong A\otimes\CD$. Let $d\in\IN$ be a natural number. Then 
\begin{enumerate}
\item The set of all cocycle actions $(\alpha, w): G\curvearrowright A$, that are cocycle conjugate to $(\alpha\otimes\id_\CD, w\otimes\eins_\CD)$, forms a $G_\delta$-subset in $\coact(G,A)$.
\item For every subgroup $H\subset G$ of finite index, the set of all cocycle actions $(\alpha,w): G\curvearrowright A$ with $\dimrok(\alpha, H)\leq d$ forms a $G_\delta$-subset in $\coact(G,A)$.
\item Assume that $G$ is residually finite. Given a regular approximation $\sigma$ of $G$, the set of all cocycle actions $(\alpha, w): G\curvearrowright A$ with $\dimrok(\alpha, \sigma)\leq d$ forms a $G_\delta$-subset in $\coact(G,A)$.
\item Assume that $G$ is residually finite and finitely generated. The set of all cocycle actions $(\alpha, w): G\curvearrowright A$ with $\dimrok(\alpha)\leq d$ forms a $G_\delta$-subset in $\coact(G,A)$.
\end{enumerate}
}
\begin{proof}
(1) We make use of~\ref{Matui}.
Let $S_n\fin\CD$ be an increasing sequence of finite subsets whose union is dense in $\CD$, and let $F_n\fin A_1$ be an increasing sequence of finite subsets whose union is dense in the unit ball of $A$. Let $M_n\fin G$ be an increasing sequence of finite subsets whose union is $G$. For each $n\in\IN$, define the set $V_n\subset\coact(G,A)$ by saying that $(\alpha,w)\in V_n$, if there exists a *-linear map $\phi: \CD\to A$ satisfying:
{
\renewcommand{\theenumi}{\alph{enumi}}
\begin{enumerate}
\item $\|[a,\phi(d)]\|<\frac{1}{n}$ for all $a\in F_n$ and $d\in S_n$;
\item $\| \big( \phi(d_1d_2)-\phi(d_1)\phi(d_2) \big) a\|<\frac{1}{n}$ for all $a\in F_n$ and $d_1,d_2\in S_n$;
\item $\|\phi(d)a\|<\frac{n+1}{n}\|d\|$ for all $a\in F_n$ and $d\in S_n$;
\item $\|\phi(\eins_\CD)a-a\|<\frac{1}{n}$ for all $a\in F_n$;
\item $\|\big( (\alpha_g\circ\phi)(d)-\phi(d) \big) a\|<\frac{1}{n}$ for all $a\in F_n, d\in S_n$ and $g\in M_n$.
\end{enumerate}
}
Looking at the above conditions (in fact only the last one), it is clear that each $V_n$ defines an open subset of $\coact(G,A)$ with respect to the topology from~\ref{polish}. A cocycle action $(\alpha,w)\in\coact(G,A)$ admits a unital $*$-homomorphism $\CD\to F_\infty(A)^{\tilde{\alpha}_\infty}$ if and only if $(\alpha,w)\in\bigcap_{n\in\IN} V_n$.

(2) In a similar fashion as (1), this follows directly from the equivalent reformulation of $\dimrok(\alpha,H)\leq d$ in~\ref{quantifier dimrok}.

(3) Let $\sigma=(G_n)_n$. By definition, we have that
\[
\set{(\alpha,w)\in\coact(G,A)~|\dimrok(\alpha, \sigma)\leq d }
\]
coincides with
\[
\bigcap_{n\in\IN} \set{(\alpha,w)\in\coact(G,A)~|\dimrok(\alpha, G_n)\leq d }.
\]
So it follows from (2) that one obtains a $G_\delta$-set. 

(4) follows from (3) by applying~\ref{dominating fg},~\ref{dimrok-domi} and inserting a dominating regular approximation.
\end{proof}

\begin{nota}
Let $G$ be a countable, discrete group. Consider the category $\CC_G$ of all (twisted) \cstar-dynamical systems $(A,\alpha,w)$ for separable \cstar-algebras $A$ and cocycle actions $(\alpha,w): G\curvearrowright A$, where the morphisms are equivariant, non-degenerate $*$-homomorphisms respecting the cocycles. We have a natural (maximal) tensor product between (twisted) \cstar-dynamical systems given by 
\[
(A,\alpha,w)\otimes_{\max}(B,\beta,v) = (A\otimes_{\max} B, \alpha\otimes_{\max}\beta, w\otimes v).
\]
Now let $\CS$ be a strictly full\footnote{This means that $\CS$ retains all morphisms from $\CC_G$ between its objects and that $\CS$ is closed under isomorphism in $\CC_G$. In this case, two systems are isomorphic if they are conjugate.} subcategory of $\CC_G$. 
We call $\CS$ tensorially absorbing, if whenever $(\alpha,w): G\curvearrowright A$ and $(\beta,v): G\curvearrowright B$ are cocycle actions with $(B,\beta,v)$ in $\CS$, then $(A,\alpha,w)\otimes_{\max}(B,\beta,v)$ is also in $\CS$.\footnote{Since we will only apply this in the case when $B$ is nuclear, the choice of maximal tensor product is rather arbitrary.}
\end{nota}

Let us record a general observation that will lead to the genericity of various properties of cocycle actions:

\theorem \label{absorbing generic}
Let $G$ be a countable, discrete group. Let $\CS$ be a strictly full subcategory of $\CC_G$ that is tensorially absorbing in the above sense.
Let $A$ and $\CD$ be two separable \cstar-algebras with $\CD$ being unital and strongly self-absorbing. Assume that $A\cong A\otimes\CD$. If there is some action $\gamma: G\curvearrowright\CD$ such that $(\CD,\gamma)$ is in $\CS$, then the set of all cocycle actions $(\alpha,w): G\curvearrowright A$ yielding an object in $\CS$ is dense in $\coact(G,A)$.
\begin{proof}
Let $(\alpha,w): G\curvearrowright A$ be a cocycle action. Let $\phi: A\to A\otimes\CD$ be an isomorphism that is approximately unitarily equivalent to $\id_A\otimes\eins$, see~\cite[2.3]{TomsWinter07}. Let $u_n$ be a sequence of unitaries in the unitalization $(A\otimes\CD)^\sim$ with $u_n\phi(a)u_n^*\stackrel{n\to\infty}{\longrightarrow} a\otimes\eins$ for all $a\in A$. Put $\phi_n = \ad(u_n)\circ\phi$. Then for every pair of non-degenerate $*$-endomorphisms $\psi: A\to A$ and $\theta:\CD\to\CD$, we have for all $x\in A$ that
\[
\begin{array}{cl}
\multicolumn{2}{l}{ \| \phi_n^{-1}\circ(\psi\otimes\theta)\circ\phi_n(x) - \psi(x)\|} \\\\
\leq& \| \phi_n(x)-x\otimes\eins\| + \| \phi_n^{-1}(\psi(x)\otimes\eins) - \psi(x)\| \\\\
=& \| \phi_n(x)-x\otimes\eins\| + \| u_n^*(\psi(x)\otimes\eins)u_n - \phi(\psi(x))\| \\\\
\stackrel{n\to\infty}{\longrightarrow} & 0.
\end{array}
\]
An analogous calculation shows that $\phi_n^{-1}(m\otimes\eins)$ converges to $m$ strictly for every $m\in\CM(A)$.
In particular, the sequence of cocycle actions 
\[
(\alpha^{(n)}, w^{(n)}) = \bigl( \phi_n^{-1}\circ(\alpha\otimes\gamma)\circ\phi_n, \phi_n^{-1}(w\otimes\eins) \bigl)
\]
converges to $(\alpha,w)$ in the topology from~\ref{polish}. By the assumptions on $\CS$, each cocycle action $(\alpha^{(n)},w^{(n)})$ yields an object in $\CS$. This finishes the proof.
\end{proof}

\reme By the same proof as above, an analogous variant of~\ref{absorbing generic} is true even if $G$ is assumed to be a $\sigma$-compact, locally compact group. One merely has to pay attention to the topology on the set of all point-norm continuous cocycle $G$-actions on a \cstar-algebra, since it is defined in a slightly different way than in the discrete case~\ref{polish}.

\cor \label{D absorbing generic}
Let $G$ be a countable, discrete group and $A$ a separable \cstar-algebra. Let $\CD$ be a strongly self-absorbing \cstar-algebra. Assume that $A\cong A\otimes\CD$. Then the set of all cocycle actions $(\alpha, w): G\curvearrowright A$ that are cocycle conjugate to $(\alpha\otimes\id_\CD, w\otimes\eins_\CD)$ is generic in $\coact(G,A)$. Hence so is the set of all cocycle actions $(\alpha, w): G\curvearrowright A$ the crossed products $A \rtimes_{\alpha, w} G$ of which are $\CD$-stable. 
\begin{proof}
The only thing left to show is that such cocycle actions are dense in $\coact(G,A)$. But this follows from~\ref{absorbing generic}. Firstly, the given property of cocycle actions is trivially invariant under conjugation, and in fact defines a strictly full and tensorially absorbing subcategory of $\CC_G$. Secondly, the trivial action of $G$ on $\CD$ is clearly conjugate to its tensorial product with itself.
\end{proof}

\prop \label{dimrok absorbing}
Let $G$ be a countable, discrete and residually finite group with a regular approximation $\sigma$. Let $d\in\IN$ be a natural number. Then the full subcategory defined by all (twisted) \cstar-dynamical systems $(A,\alpha,w)$ with $\dimrok(\alpha, \sigma)\leq d$ is strictly full and tensorially absorbing.
\begin{proof}
This follows directly from the definitions~\ref{dimrok central sequence},~\ref{dimrok allgemein}.
\end{proof}

\rem \label{UHF action}
Let $G$ be a countable, discrete and residually finite group. Let $\sigma=(G_n)_n$ be a regular approximation consisting of normal subgroups. For each $n\in\IN$, denote by $\lambda_n: \Gamma_n\to\CU(M_{|\Gamma_n|})$ the left-regular representation of $\Gamma_n=G/G_n$. For each $n\in\IN$, let $\pi_n: G\onto\Gamma_n$ be the canonical surjection.
Let $M_\sigma = \bigotimes_{n=1}^\infty M_{|\Gamma_n|^\infty}$ and define $\beta^\sigma: G\curvearrowright M_\sigma$ via $\beta^\sigma_g = \bigotimes_{n=1}^\infty \bigl( \bigotimes_\IN \ad\big( (\lambda_n\circ\pi_n)(g) \big) \bigl)$.
By design, this yields an action of Rokhlin dimension zero along $\sigma$. 

{\theorem 
Let $G$ be a countable, discrete and residually finite group. Let $\sigma=(G_n)_n$ be a regular approximation consisting of normal subgroups. Assume that $A$ is a separable \cstar-algebra with $A\cong A\otimes M_\sigma$. Then the cocycle actions $(\alpha,w): G\curvearrowright A$ with $\dimrok(\alpha, \sigma)=0$ are generic in $\coact(G,A)$.}

\begin{proof}
By~\ref{Gdelta}, we only have to show that said actions are dense.
But this follows directly from~\ref{absorbing generic},~\ref{dimrok absorbing} and~\ref{UHF action}. 
\end{proof}

\theorem \label{generic dimrok UHF}
Let $G$ be a discrete, finitely generated and residually finite group. Assume that $A$ is a separable \cstar-algebra that tensorially absorbs the universal UHF algebra. Then the cocycle actions $(\alpha,w): G\curvearrowright A$ with Rokhlin dimension zero are generic in $\coact(G,A)$.

\defi Let $\Gamma$ be a finite group. Identifying $\CZ\cong\bigotimes_{\Gamma}\CZ$, we define the noncommutative Bernoulli shift $\gamma^\Gamma: \Gamma\curvearrowright\CZ$ via
\[
\gamma^\Gamma_g\Big( \bigotimes_{h\in\Gamma} a_h \Big) = \bigotimes_{h\in\Gamma} a_{gh}.
\] 

\prop \label{dimrok bernoulli}
For all finite groups $\Gamma$, the action $\gamma^\Gamma$ has Rokhlin dimension 1.
\begin{proof}
First, let us observe 
\[
\bigotimes_{g\in\Gamma} \CZ \cong \bigotimes_{g\in\Gamma} \Big(\bigotimes_\IN \CZ \Big) \cong \bigotimes_\IN \Big(\bigotimes_{g\in\Gamma} \CZ \Big).
\]
With this rearrangement of tensor factors, we get that $\gamma^\Gamma$ is conjugate to $\bigotimes_\IN\gamma^\Gamma$. Since $\gamma^\Gamma$ is a faithful action, its infinite tensor power is pointwise strongly outer. Now it follows from~\cite[3.4]{MatuiSato12_2} that $\gamma^\Gamma$ has the weak Rokhlin property in the sense of~\cite{MatuiSato12_2, MatuiSato14}. Let $\sigma: \Gamma\curvearrowright\CC(\Gamma)$ denote the canonical $\Gamma$-shift action. Let $\omega\in\beta\IN\setminus\IN$ be a free ultrafilter.
In this particular instance, the weak Rokhlin property means that there is an equivariant c.p.c.~order zero map
\[
\psi: (\CC(\Gamma), \sigma) \to (\CZ_\omega, \gamma^\Gamma_\omega)
\]
with $\tau_{\CZ,\omega}(\psi(\eins))=1$. (Recall the definition of $\CZ_\omega, \gamma^\Gamma_\omega$ from the first section.)

Now notice that the fixed point algbera $\CZ^{\gamma^\Gamma}$ is separable, unital, nuclear, simple and has a unique tracial state. Moreover, the element $\psi(\eins)$ is an element of $(\CZ^{\gamma^\Gamma})_\omega$ and has full trace. Let $h_0\in\CZ$ be a positive contraction with full spectrum $[0,1]$, and define $h_1=\eins_\CZ-h_0$. It follows from~\cite[4.3]{SatoWhiteWinter14} that there are unitaries $u_0$ and $u_1$ in $(\CZ^{\gamma^\Gamma}\otimes\CZ)_\omega$ with
\[
u_i(\psi(\eins)\otimes h_i)u_i^* = \eins\otimes h_i\quad\text{for}~i=0,1.
\]
Viewing these as unitaries in $(\CZ\otimes\CZ)_\omega$, we define the two c.p.c.~order zero maps
\[
\psi_0,\psi_1: \CC(\Gamma)\to (\CZ\otimes\CZ)_\omega\quad\text{via}\quad\psi_i(x) = u_i(\psi(x)\otimes h_i)u_i^*\quad\text{for}~i=0,1.
\]
As $\psi$ was equivariant and the unitaries $u_0,u_1$ are both fixed by $({\gamma^\Gamma}\otimes\id_\CZ)_\omega$, this actually yields two equivariant c.p.c.~order zero maps
\[
\psi_0,\psi_1: (\CC(\Gamma),\sigma)\to \bigl( (\CZ\otimes\CZ)_\omega, ({\gamma^\Gamma}\otimes\id)_\omega \bigl)\quad\text{with}\quad \psi_0(\eins)+\psi_1(\eins)=\eins.
\]
Now by~\cite[4.7, 4.8]{MatuiSato12_2}, ${\gamma^\Gamma}$ is conjugate to ${\gamma^\Gamma}\otimes\id_\CZ$. Since ${\gamma^\Gamma}$ is also conjugate to its own infinite tensor power, this verifies $\dimrok({\gamma^\Gamma})\leq 1$. Note that $\dimrok({\gamma^\Gamma})=0$ is impossible, since $\CZ$ contains no non-trivial projections.
\end{proof}

\rem \label{Z action} 
Let $G$ be a countable, discrete and residually finite group. Let $\sigma=(G_n)_n$ be a regular approximation consisting of normal subgroups. For each $n\in\IN$, let $\Gamma_n=G/G_n$ and $\pi_n: G\onto\Gamma_n$ the canonical surjection.
Identifying $\CZ\cong\bigotimes_\IN\CZ$, we define
\[
\gamma^\sigma: G\curvearrowright \CZ\quad\text{via}\quad \gamma^\sigma_g = \bigotimes_{n\in\IN} \gamma^{\Gamma_n}_{\pi_n(g)}.
\]
Notice that by design (i.e.~by~\ref{dimrok bernoulli}), this action has Rokhlin dimension 1.

{\theorem \label{generic dimrok Z}
Let $G$ be a countable, discrete and residually finite group with a chosen regular approximation $\sigma$ consisting of normal subgroups. Let $A$ be a separable \cstar-algebra with $A\cong A\otimes\CZ$. Then the cocycle actions $(\alpha,w): G\curvearrowright A$ with $\dimrok(\alpha,\sigma)\leq 1$ are generic in $\coact(G,A)$. If $G$ is finitely generated, then the cocycle actions $(\alpha,w): G\curvearrowright A$ with $\dimrok(\alpha)\leq 1$ are generic in $\coact(G,A)$.}
\begin{proof}
By~\ref{Gdelta}, we only have to show that said actions form a dense subset.
But this follows directly from~\ref{absorbing generic},~\ref{dimrok absorbing} and~\ref{Z action}.
\end{proof}

\rem It is very straightforward to see that if one restricts one's attention to genuine actions, then the genericity results of this section (in particular,~\ref{D absorbing generic},~\ref{generic dimrok UHF},~\ref{generic dimrok Z} and the general principle~\ref{absorbing generic}) remain true within the Polish space of all genuine $G$-actions on a separable \cstar-algebra $A$. In order to modify the proofs, one merely has to omit the cocycles.

\rem We would like to emphasize that unlike most of the previous results concerning genericity (such as in~\cite{HirshbergWinterZacharias14} or~\cite{Phillips12}), the results of this section were proved rather directly. In particular, we did not have to appeal to a Baire category argument in order to show denseness of any subsets in $\coact(G,A)$.


\bibliographystyle{gabor}
\bibliography{master}

\end{document}